\documentclass[11pt]{article}

\usepackage{amsmath,amsfonts,amsthm,amssymb}
\usepackage[utf8]{inputenc}
\usepackage[T1]{fontenc}
\usepackage{color}
\usepackage{enumitem}
\usepackage{hyperref}

\newtheorem{theorem}{Theorem}
\newtheorem{lemma}[theorem]{Lemma}
\newtheorem{proposition}[theorem]{Proposition}
\newtheorem{corollary}[theorem]{Corollary}
\theoremstyle{definition}

\newtheorem{definition}{Definition}
\newtheorem{remark}{Remark}
\newtheorem{assumption}{Assumption}
\newtheorem*{UnnumberedLemma}{Lemma}
\newtheorem*{unumberedtheorem}{Theorem}

\newcommand{\N}{\mathbf{N}}

\newcommand{\R}{\mathbf{R}}
\newcommand{\C}{\mathbf{C}}

\newcommand{\Hc}{\mathcal{H}}
\newcommand{\X}{\mathcal{X}}
\newcommand{\Y}{\mathcal{Y}}

\newcommand{\Bc}{L}
\newcommand{\Dc}{\mathcal{D}}

\newcommand{\nuovo}[1]{#1}

\numberwithin{equation}{section}

\def \p{\partial}

\def \rmi{{\rm i}}

\newcommand{\pro}{\Upsilon}

\newcommand{\radon}[1]{\mathcal{R}(#1)}

\newcommand{\TV}[3]{\mathrm{TV}(#1,(#2,#3))}
\newcommand{\Bor}{\mathcal{L}^\infty}

\providecommand{\keywords}[1]{\textbf{\textit{Index terms---}} #1}

\title{Regular propagators of bilinear quantum systems}

\author{Nabile Boussa\"id \\
 Laboratoire de Math\'ematiques de Besan\c{c}on, UMR 6623 \\
 Univ. Bourgogne Franche-Comt\'e, Besan\c{c}on, France\\
 \texttt{Nabile.Boussaid@univ-fcomte.fr}\\
~\\
Marco Caponigro\\
  \'Equipe M2N\\
 Conservatoire National des Arts et M\'etiers, Paris, France\\
\texttt{Marco.Caponigro@cnam.fr}\\
~\\
Thomas Chambrion\\
Universit\'e de Lorraine, IECL, UMR 7502, Vandoeuvre-l\`es-Nancy, F-54506, 
France\\
CNRS, IECL, UMR 7502, Vandoeuvre-l\`es-Nancy, F-54506, France\\
Inria, SPHINX, Villers-l\`es-Nancy, F-54600, France\\
 \texttt{Thomas.Chambrion@univ-lorraine.fr}
}

\begin{document}
\maketitle

\begin{abstract}
The present analysis deals with the regularity of solutions of bilinear control 
systems of the type $x'=(A+u(t)B)x$
where the state $x$ belongs to some complex infinite dimensional 
Hilbert space, the (possibly unbounded) linear operators $A$ and 
$B$ are skew-adjoint and the control $u$ is a real valued function. 
Such systems arise, for instance, in 
quantum control with the bilinear Schr\"{o}dinger equation. For the sake of the regularity analysis, we consider a more general framework where 
$A$ and $B$ are generators of contraction semigroups.

Under some hypotheses on the commutator of the operators $A$ and $B$, it is possible to 
extend the definition of solution for controls in the set of Radon measures to obtain precise \emph{a priori} 
energy estimates on the solutions, leading to a natural extension of the celebrated 
noncontrollability result of Ball, Marsden, and Slemrod in 1982. 
\end{abstract}

\newpage 

\tableofcontents

\section{Introduction}

A bilinear control 
 system in  a Banach space $\X$ is given by an evolution equation
\begin{equation}\label{EQ_bilinear_abstrait}
\frac{d}{dt}x(t)=(A+u(t)B)x(t)
\end{equation} 
 where  $A$ and $B$ are two (possibly unbounded) 
  linear operators on $\X$ and $u$ is a real-valued function, the 
 control. Well-posedness of bilinear evolution equations of 
 type~\eqref{EQ_bilinear_abstrait} for a given 
 control $u$  is usually a difficult question.   If $K$ is a subset of $\R$, we define 
 $PC(K)$ the set of right-continuous 
  piecewise constant functions  taking values in $K$.

 If $K$, $A$ and $B$ are such that
 for every $u$ in $K$,
  $A+uB$ generates a $C^0$ semigroup $t\mapsto e^{t(A+uB)}$, then  for every $T\geq 0$ and every $u$ in $PC(K)$, the restriction of $u$ on $[0,T)$ writes  
 \begin{equation}\label{eq:forme_controles_PC}
u=\sum_{j=1}^p u_j \mathbb{I}_{[\tau_j,\tau_{j+1})}
  \end{equation}
with $p \in \N$, $u_1,\ldots,u_p \in K$ and 
  $\tau_1<\tau_2<\ldots <\tau_{p+1}=T$, and   
   one defines the associated propagator of 
  (\ref{EQ_bilinear_abstrait}) by 
  $$\Upsilon^u_{t,\tau_1}=e^{(t-\tau_j)(A+u_jB)}\circ e^{(\tau_j-\tau_{j-1})
  (A+u_{j-1}B)}\circ \cdots \circ 
   e^{(\tau_2-\tau_1)  (A+u_1B)},
   $$ for every $t$ in $(\tau_j,\tau_{j+1})$. The solution 
   of~(\ref{EQ_bilinear_abstrait}) with 
  initial value $x_0$ at time $\tau_1$ is $t\mapsto \Upsilon^u_{t,\tau_1}\psi_0$. When 
  $\tau_1=0$, we  denote $\Upsilon^u_t:=\Upsilon^u_{t,0}$. 
  
It is of particular interest in the applications to study the set of points that can be attained in finite time from a given initial datum $\psi_0$ using a set of admissible controls in ${\mathcal Z}$
\[
{\mathcal Att}_{\mathcal{Z}}(\psi_0)= \cup_{t \geq 0} \{\Upsilon^u_t \psi_0 | u \in \mathcal{Z} \}
\]
where $\mathcal{Z}$ is a subset of $PC(K)$ or, possibly, a larger set (provided that a suitable extension of $\Upsilon$ to $\mathcal{Z}$ makes sense). The set ${\mathcal Att}_{\mathcal{Z}}(\psi_0)$ is called attainable set from $\psi_0$ with controls ${\mathcal Z}$.
  

Providing a precise description of the propagator is, in principle, a hard task and, in turns, so is studying the controllability of~\eqref{EQ_bilinear_abstrait}. On the other hand, one could
use the regularity of the solutions of~\eqref{EQ_bilinear_abstrait} to provide upper bounds of the attainable sets of the bilinear system in order to determine obstructions to controllability. The present analysis focuses on this second approach.

\subsection{Elementary obstructions to controllability in a Banach space}
There are several upper bounds on the attainable sets that can be deduced from natural properties of the system. 
We list below some of them.

\subsubsection{Conservation of the norm}
In the Hilbertian case, in which $\X$ is a Hilbert space, the propagator $t\mapsto \Upsilon^u_t$ is unitary as soon as $A+uB$ is essentially skew-adjoint for every $u$ in $K$. If $PC(K)$ is endowed with a topology for which $u\mapsto \Upsilon^u_T \psi_0$ is continuous for every $T>0$ and every $\psi_0$ in $\X$, then the continuous extension of the mapping $u\mapsto \Upsilon^u_T \psi_0$ takes value in the sphere of radius $\|\psi_0\|$.

\subsubsection{Continuity of the propagators }\label{SEC_principe_obstruction}

In the general case in which $\X$ is a Banach space, let $\mathcal Z$ be a topological space, \nuovo{containing $PC(K)$, endowed with a topology such that $PC(K)$ is dense in $\mathcal Z$ and
 $u\in {\mathcal Z}\mapsto \Upsilon^u_T \psi_0\in \X$ is continuous for every $T>0$ and every $\psi_0$ in $\X$.} Assume, moreover, that 
$u\mapsto \Upsilon^u_T \psi_0$ admits a (necessarily unique) continuous extension to $\mathcal Z$.
If ${\mathcal Z}_0 \subset {\mathcal Z}$, endowed with a topology finer than the one induced by ${\mathcal Z}$,  is sequentially compact (for its own topology), then for every $\psi_0$ in $\X$, for every $T>0$, the attainable set at time $T$ from $\psi_0$ with controls in ${\mathcal Z}_0$,
$
\{\Upsilon^u_T \psi_0 | u \in \mathcal{Z}_0 \} 
$ 
is compact.

If  $(\mathcal{Z}_i)_{i\in \N}$  is a countable covering of $\mathcal{Z}$, $\mathcal{Z}=\cup_{i \in \N} \mathcal{Z}_i$, $\mathcal{Z}_i$ is sequentially compact for every $i$, and
the topology of $\mathcal{Z}_i$  is finer than the topology induced by $\mathcal{Z}$,  then the attainable set at time $T$ from $\psi_0$ with controls in ${\mathcal Z}$,
$
\{\Upsilon^u_T \psi_0 | u \in \mathcal{Z} \} =\cup_{i\in \N} \{\Upsilon^u_T \psi_0 | u \in \mathcal{Z}_i \}  
$ 
is a countable union of compact sets in $\X$ (hence is a meager set in the sense of Baire as soon as $\X$ is infinite dimensional).

Notice that if the input-output mapping $PC(K) \ni u \mapsto  \Upsilon^u \psi_0 \in C^0([0,T],\X)$ is continuous, then the above results can be generalized to show that the attainable set from $\psi_0$ at time less than $T$ 
$
\cup_{0\leq t \leq T} \{\Upsilon^u_t \psi_0 | u \in \mathcal{Z} \} =\cup_{i\in \N} 
\cup_{0\leq t \leq T}
\{\Upsilon^u_T \psi_0 | u \in \mathcal{Z}_i \}  
$ 
is an union of compact sets.

This is the underlying idea of the proof
 of the following  result by Ball, Marsden, and Slemrod.
\begin{unumberedtheorem}[Theorem 3.6 in \cite{BMS}]
Let $\X$ be  an infinite dimensional Banach space, $A$ generate a $C^0$ semigroup of 
bounded linear 
operators on $\X$, and $B$ be a bounded linear operator on $\X$. Then for any $T\geq 0$,  the 
input-output mapping $u\mapsto 
\Upsilon^u_{T}$ admits a unique continuous extension to $L^{1}([0,T],\R)$
and the attainable set
\begin{equation}\label{EQ_attainable_set}
\bigcup_{r > 1} \bigcup_{T\geq 0} \bigcup_{u \in L^r([0,T],\R)} \{ \Upsilon^u_{t}\psi_0 \mid t\in [0,T]\}
\end{equation}
is contained in a countable union of compact subsets of $\X$, and, in particular, has dense complement. 
\end{unumberedtheorem} 
In this case, for any $T\geq0$, $\mathcal{Z}=\cup_{r>1} L^r([0,T],\R)$, \nuovo{ $\mathcal {Z}=\cup_{i,j} \mathcal{Z}_{ij}$ with 
  \[ 
  \mathcal{Z}_{i,j}=\{f \in L^{1+\frac{1}{j}}([0,T],\R) \mid \|f\|_{L^{1+\frac{1}{j}}([0,T])} \leq i\},
  \] 
  and $L^{1+\frac{1}{j}}([0,T],\R)$ is endowed with the weak $L^1$-topology.} The sequential-compactness of $\mathcal{Z}_{i,j}$ is granted by Banach--Alaoglu--Bourbaki Theorem. The main difficulty of Theorem 3.6 in \cite{BMS} is to prove that, for any $\psi_0$ in $\X$,    the weak convergence of $(u_n)$ to $u$ in $L^1$ implies strong convergence of the associated sequence of solutions
of \eqref{EQ_bilinear_abstrait} $(t\mapsto \Upsilon^{u_n}_t \psi_0)$ to $t\mapsto \Upsilon^{u}_t \psi_0$.   

 \begin{remark}\label{Remark:OpenQuestionBMS}
 The above argument does not hold anymore if one considers controls in $L^1$, since $L^1$ is not a reflexive
 space.   This is the content of \cite[Remark 3.8]{BMS}, where the question of  possible extensions of
 the above result to $r=1$ is  left open except in the so-called (see~\cite{Slemrod}) diagonal case, see \cite[Theorem 5.5]{BMS}. 
  \end{remark}

Another example of the same obstruction is given below  in Corollary \ref{COR_obstruction_contr_exacte_BV} with $\mathcal Z$ equal to the set of functions with bounded variations. In this case, the sequential compactness in $\mathcal Z$ is given by Helly's selection theorem.

\subsubsection{Invariance of the domain}\label{Sec:InvarianceDomain}

In the case in which $A$ and $B$ are bounded operators on $\X$, if $\mathcal F$ is a closed subspace of $\X$ left invariant by $A+uB$ for every $u$ in $K$, then for every $u$, the $C^0$ semigroup generated by $A+uB$ leaves  $\mathcal F$ invariant. Thus, for every $u$ in $PC(K)$ and every $t\geq 0$, $\Upsilon^u_t$ leaves $\mathcal F$ invariant. 
If, moreover, the dynamics is time-reversible, then for every $\psi_0$ in $\X$, for every $u$ in $PC(K)$, for every $t>0$, 
$\Upsilon^u_t \psi_0 \in \mathcal{F}$ if and only if $\psi_0 \in\mathcal{F}$ . 

Even in the unbounded case, the same conclusion holds if $\mathcal F$ a subspace of $\X$ left invariant by the dynamics $\Upsilon^u_t$ and its  time-reverse dynamics (when it makes sense).

We will see in  Section \ref{SEC_Higher_regularity} below that these invariance properties remain true in the Hilbert case when $\mathcal F$ is the domain of a power of $A$ left invariant by $B$.

\subsection{Attainable sets in quantum control}

The main motivation for the present analysis comes from the problem of controllability for closed quantum systems. The state of a quantum system evolving on a finite dimensional Riemannian
manifold $\Omega$, with associated measure $\mu$, is described by its \emph{wave
function},
represented as
a point in the unit sphere of $L^2(\Omega, \mathbf{C})$. 
In absence of interactions with the environment and neglecting the
relativistic effects, the time evolution of the
wave function is given by the Schr\"odinger equation
\begin{equation*}
\rmi \frac{\partial \psi}{\partial t}=-\frac{1}{2}\Delta \psi +V(x) \psi(x,t),
\end{equation*}
where $\Delta$ is the Laplace-Beltrami operator on $\Omega$
and
$V:\Omega\rightarrow \mathbf{R}$ is a real function (usually called potential)
accounting for the physical properties of the system.  When submitted to an
excitation by an external field (e.g. a laser), the Schr\"odinger
equation reads
\begin{equation}\label{eq:blse}
\rmi \frac{\partial \psi}{\partial t}=-\frac{1}{2}\Delta \psi +V(x) \psi(x,t)
+u(t)
W(x) \psi(x,t),
\end{equation}
where $W:\Omega\rightarrow \mathbf{R}$ is a real function accounting for the
physical properties of the external field and $u$ is a real function of  time
accounting for the intensity of it.

In the last decades, many efforts have been made to describe the attainable set 
of \eqref{eq:blse}. In~\cite{turinici}, Turinici adapted the result by Ball, Marsden and 
Slemrod~\cite[Theorem~3.6]{BMS} to \eqref{eq:blse} with a measurable bounded 
$W$. \nuovo{The first known positive result  has been obtained by Beauchard 
in~\cite{beauchard}, and improved in \cite{camillo, BeauchardMorancey}, for 
$\Omega=(0,\pi)$  with Dirichlet boundary conditions, $V=0$ and 
$W:x\mapsto x^2$, see Section~\ref{Sec:PotentialWell} for more details. }In the 
case of the quantum harmonic oscillator: $\Omega=\R$, $V(x)=x^2$ and 
$W:x\mapsto x$, the attainable set is finite dimensional due to symmetries of 
the system, see Rouchon and Mirrahimi in~\cite{Rouchon} and 
Section~\ref{SEC_harmonic_oscillator}.

\nuovo{At present, very few  results providing a precise description of the attainable sets of  \eqref{eq:blse}  in the spirit of \cite{beauchard} have been obtained, and only with strong restrictions on the dimension and the boundedness of the domain $\Omega$.
} Instead of lower bounds of the attainable sets, many works have considered 
lower bounds of its closure (in different natural norms) which is sufficient 
from a physical point of view. We can, for instance, 
cite the work by Nersesyan~\cite{Nersy}, in Sobolev spaces by means of Lyapunov 
technics for bounded domains and potentials. Concerning unbounded domains but 
with bounded potentials, we can cite~\cite{mirrahimi-continuous} with Lyapunov 
technics as well.  
Geometric methods have been used to prove the density of the attainable set in 
$L^2$-norm when the spectrum is purely discrete and nonresonance conditions  are 
satisfied, see~\cite{Schrod, genericity-mario-paolo, Schrod2, periodic, BCS14, 
Esatta}. \nuovo{The common strategy used in the above mentioned approximate controllability results is to fix an initial condition and to design  a sequence of controls such that the associated trajectories converge, in some appropriate sense, to the target. 
If bounds  were known, both on the controllability time and on some $L^p$ norms, for these sequences of control, one could extract a convergent subsequence of controls and use the continuity of the endpoint mapping to infer exact controllability.  
}
The present work, similarly to \cite{Nersy}, considers the question of 
the regularity of a solution of \eqref{EQ_bilinear_abstrait} but in a more 
general way, following the spirit of \cite{BoussaidCaponigroChambrion}.

\nuovo{
\subsection{Impulsive control}
As mentioned in Remark \ref{Remark:OpenQuestionBMS}, a major difficulty in 
extending the result of obstruction of Ball-Marsden-Slemrod is the lack of reflectiveness of 
$L^1$ and the fact that a bounded sequence in $L^1$ does not necessarily admits 
a weakly-convergent subsequence. On the other hand, the set of Radon measures, when endowed with the total variation topology (see Appendix~\ref{sec:notations}), possesses weak sequential compactness properties.  As a consequence, throughout this work we will deduce properties for integrable controls from the equivalent statement set in the framework of Radon measures controls.

The idea to consider measures (instead of functions) as controls has given rise to a large literature.  Let us cite, among many others contributions,    \cite{4045697}, \cite{MR0484623}, \cite{MR1108058}, \cite{MR963323}, \cite{10.1007/978-1-4613-8489-2_1},  \cite{MR2350058} , \cite{MR3147261}. For an introduction to the subject we refer to the book \cite{MR2024011}. Nowadays, the question of the definition of solution for 
dynamics such as \eqref{EQ_bilinear_abstrait} in finite dimensional spaces is 
essentially well-understood. However in an infinite  dimensional framework several questions are still open. Our construction  of generalized propagators for Radon 
measures (Section~\ref{sec:radon} below) can be compared with the strategy used by Miller and Rubinovich in Section 4.2.1 of \cite{MR2024011} for finite-dimensional systems. 
}

\subsection{Main results}\label{SEC_main_results}

\subsubsection{Upper bound for attainable sets of bilinear control systems}
Our aim is to give upper bounds for the attainable set of a bilinear control system. 
The main result is the following.
\begin{theorem}\label{THE_main_introduction}
Let $\Hc$ be an infinite dimensional Hilbert space, $A$ be a maximal dissipative operator on $\Hc$ with domain $D(A)$, and $B$ be an 
operator on $\Hc$ such that $B-c$ and $-B-c'$ generate contraction semigroups leaving $D(A)$ invariant for some real constants $c \geq 0$ and $c'\geq 0$. 
Assume that $A+uB$ is maximal dissipative with domain $D(A)$ for every $u$ in $\R$ and that
the map $t \in \mathbf{R} \mapsto e^{tB} A e^{-tB} \in L(D(A),\Hc)$, \nuovo{where $D(A)$ is endowed with the graph norm}, is locally Lipschitz. 
Then, for every $T>0$ and \nuovo{for every 
$\psi_0$ in $\Hc$, there 
exists a unique continuous extension to $L^1([0,T],\R)$ of the enpoint mapping $u\mapsto \Upsilon^u_{T}\psi_0 \in  \Hc$ of~\eqref{EQ_bilinear_abstrait}, moreover} the set
$$
{\mathcal Att}_{L^1}(\psi_0):=\bigcup_{T\geq 0} \bigcup_{u \in L^1([0,T],\R)} \{\Upsilon^u_{t,0}\psi_0 \mid t\in [0,T]\}
$$
is contained in a countable union of compact subsets of $\Hc$.
\end{theorem}
\begin{proof}
See Section \ref{SEC_attainable_set_interaction}.
\end{proof}

\nuovo{When $A$ and $B$ are skew-adjoint, the orbits lie in the sphere of $\Hc$ of radius $\|\psi_0\|$, which is, of course, a  meager set in $\Hc$.}
As a consequence of Theorem~\ref{THE_main_introduction}, 
$$
\bigcup_{\alpha \geq 0} \bigcup_{T\geq 0} \bigcup_{u \in L^1([0,T],\R)} \{\alpha \Upsilon^u_{t,0}\psi_0 \mid t\in [0,T]\}
$$
is a meager set in $\Hc$ and hence it has dense complement. \nuovo{Nonetheless, in the skew-adjoint case,
the attainable set is a meager set (and hence has dense complement) in the sphere of radius $\|\psi_0\|$.}

In the special case where the control operator $B$ is bounded, using a different construction, we obtain the simplified 
statement below similar to the one of~\cite{BMS}, but dealing with $L^1$ controls.

\begin{proposition}\label{PRO_introduction_BMS_Bborne}
Let $\X$ be  an infinite dimensional  Banach space, $A$ generate a $C^0$ semigroup of bounded linear 
operators on $\X$, and $B$ be a bounded linear operator on $\X$. Then for every $T>0$, there 
exists a unique continuous extension to $L^1([0,T],\R)$ of the input-output mapping 
$u\mapsto \Upsilon^u_{T} \in  L(\Hc,\Hc)$
of 
\eqref{EQ_bilinear_abstrait}
 and, for every $\psi_0$ in $\Hc$, 
$$ 
{\mathcal Att}_{L^1}(\psi_0):=\bigcup_{T\geq 0} \bigcup_{u \in L^1([0,T],\R)} \{ \Upsilon^u_{t}\psi_0 \mid t\in [0,T]\}
$$
is contained in a countable union of compact subsets of $\X$ and, in particular, has dense complement. 
\end{proposition}
\begin{proof}
See Section~\ref{Sec:BoundedPotential}.
\end{proof}

These results set the open question by Ball, Marsden, and Slemrod in~\cite[Remark 3.8]{BMS}.
The scheme of the proofs of Theorem \ref{THE_main_introduction} and Proposition 
\ref{PRO_introduction_BMS_Bborne} follows the structure of the proof of \cite[Theorem 3.6]{BMS}.
The lack of reflectiveness of  $L^1$ leads us to consider Radon measures as controls, the weak-compactness of bounded sequences is ensured by Helly's Selection 
Theorem. The  main difficulty 
is  to define a continuous input-output mapping associated with 
(\ref{EQ_bilinear_abstrait}) in such 
a way to guarantee compactness properties for the attainable sets. 

\begin{remark}
  Theorem~\ref{THE_main_introduction} still holds true for Radon measures controls, as stated in Corollary~\ref{Cor:NoExactControllability} below.
Here the result is presented in term of $L^1$ controls for the sake of readability, indeed the definition of propagator associated with a Radon measure requires  preliminary notions presented in Section~\ref{sec:radon}. 
The hypotheses of Theorem~\ref{THE_main_introduction} are needed in order to prove  continuity of the propagators after a particular change of variable (the \emph{interaction framework} presented in Section~\ref{Sec:Interaction}). The key result in the proof of the continuity is an adaptation of a classical result by Kato~\cite{Kato1953} (see Proposition~\ref{prop:continuity}). 
\end{remark}

\subsubsection{Higher regularity}

The Lipschitz assumption 
on the map $t \in \mathbf{R} \mapsto e^{tB} A e^{-tB} \in L(D(A),\Hc)$
in Theorem~\ref{THE_main_introduction} is crucial for our analysis when $B$ is 
unbounded, however it may be hard to check in practice. For bilinear systems encountered in quantum physics, one can take advantage of the skew-adjointness of the operators to simplify the analysis. For instance, 
we have the following result. 

 \begin{theorem}\label{THE_intro_weak_coupling}
Let $\Hc$ be an infinite dimensional Hilbert space, $k$ a positive 
number, $A$ and $B$ be two skew-adjoint operators such that:
\begin{itemize}
 \item[$(i)$] $A$ is invertible with bounded 
inverse from $D(A)$ to $\Hc$,
 \item[$(ii)$]  for any  $t \in \R$, $e^{tB}D(|A|^{k/2})\subset D(|A|^{k/2})$,
 \item[$(iii)$] there exists $c\geq 0$ and $c'\geq 0$ such that $B-c$ 
and $-B-c'$ generate contraction semigroups on $D(|A|^{k/2})$ for the norm $\|\cdot\|_{k/2}$,
\item[$(iv)$]  $B$ is $A$-bounded with $\|B\|_A=0$ (see \eqref{def:Abound} below for the precise definition).
 \end{itemize}
Then, for every $T>0$, there exists a unique strongly continuous 
extension to $BV([0,T],\R)$, endowed with the $\|\cdot \|_{L^1} + \TV{\cdot}{[0,T]}{\R}$-norm, of the end-point mapping $u\mapsto \Upsilon^u_T$ of~\eqref{EQ_bilinear_abstrait}. 
Moreover, for every $\psi_0$ in $D(|A|^{k/2})$, the set
$$ 
\bigcup_{\alpha \geq 0} 
\bigcup_{T\geq 0} 
\bigcup_{u \in BV([0,T],\R)} \{ \alpha \Upsilon^u_t\psi_0, t\in [0,T]\},
$$
is contained in a countable union of compact subsets of
$D(|A|^{k/2})$.
\end{theorem}

\begin{proof}
See Section~\ref{SEC_Higher_regularity}.
\end{proof}

\begin{remark}
 Theorem~\ref{THE_intro_weak_coupling}
 is a reformulation of 
 Theorem~\ref{THE_main_introduction}
 in the smaller functional framework of conservative
  dynamics. 
  \nuovo{The proof of Theorem~\ref{THE_intro_weak_coupling} is a consequence of Corollary~\ref{Cor:NoExactControllability2} below in the case of bounded variation controls. In Section~\ref{SEC_weak_coupling_for_radon} the result is then generalized to Radon measures controls. Corollary~\ref{Cor:NoExactControllability2} also provides an extension of Theorem~\ref{THE_intro_weak_coupling} from $D(|A|^{k/2})$ to $D(|A|^{k/2+1-\varepsilon})$ if $\psi_0$ is in $D(|A|^{k/2+1-\varepsilon})$, for $\varepsilon\in (0,1)$. }
 \end{remark}

\begin{remark}
A simple checkable condition for a pair of skew-adjoint operators $(A,B)$
to satisfy assumptions $(i)-(iii)$ in Theorem~\ref{THE_intro_weak_coupling} 
is to be weakly coupled in the sense of~\cite[Definition~1]{BoussaidCaponigroChambrion}.
See Lemma~\ref{LEM_lien_WC_nouvelle_definition_ancienne_definition} below. 
\end{remark}

\begin{remark}
 Recall that there exists $c\geq 0$ and $c'\geq 0$ such that $B-c$ 
and $-B-c'$ generate contraction semigroups on $D(|A|^{k/2})$ if and only if these operators are maximal dissipative in the functional space $D(|A|^{k/2})$. Assumption $(iii)$ in Theorem~\ref{THE_intro_weak_coupling} is, in some sense, an assumption on the commutator of $A$ and $B$, see Section~\ref{Sec:HigherOrderFEPS}. \nuovo{This condition replaces the Lipschitz assumption  
on the map $t \in \mathbf{R} \mapsto e^{tB} A e^{-tB} \in L(D(A),\Hc)$
of Theorem~\ref{THE_main_introduction}. }
\end{remark}

\subsubsection{Applications to the bilinear Schr\"odinger equation}

Here we consider the motion of a nonrelativistic quantum charged particle trapped in an infinite square potential well excited by an external electric field. That is the dynamics governed by a Schr\"odinger equation on the interval $(0,1)$ with a control potential $W: (0,1) \to \R$, which writes
\begin{equation}\label{Eq:Potential} 
\begin{cases}
\displaystyle \rmi \frac{\partial \psi}{\partial t}(t,x) = 
- \frac{\partial^2 \psi}{\partial x^2}(t,x)
- u(t) W(x) \psi (t,x) , \qquad x \in (0,1), t \in (0,T), \\
\psi(t,0)=\psi(t,1)=0.
\end{cases}
\end{equation}
We denote by  $H^s_{(0)}((0,1),\C)$
the domain of $|A|^{s/2}$ where $A$ is the Laplace--Dirichlet operator on $(0,1)$, and by $\varphi_k$, $k\in \N$ its (normalized) eigenvectors associated respectively to $\lambda_k$, $k\in\N$ its increasing sequence of eigenvalues (which are known to be simple). 
Let us recall the main result of \cite{camillo}. 
\begin{unumberedtheorem}[Theorem~1 in \cite{camillo}]
Let $T>0$ and $W \in H^{3}((0,1),\R)$ be such that there 
exists $c>0$ verifying  $
\frac{c}{k^3} \leqslant |\langle W \varphi_1 , \varphi_k \rangle |$, 
for all $k \in \N$.
Then there exists $\delta>0$ and a $C^1$ map 
$\Gamma:  \mathcal{V}_T  \rightarrow  L^2((0,T),\R)$
where
$$
\mathcal{V}_T := \{ \psi_f \in H^3_{(0)}((0,1),\C) \mid
\|\psi_f\|=1, \| \psi_f - \psi_1(T) \|_{H^3} < \delta  \},
$$
such that, $\Gamma( \psi_1(T) ) =0$ and for every $\psi_f \in  \mathcal{V}_T$,
the solution of~\eqref{Eq:Potential}
with initial condition $\psi(0)=\varphi_1$ and control $u=\Gamma(\psi_f)$ satisfies $\psi(T)=\psi_f$.
\end{unumberedtheorem}

The above result applies for instance to $W:x\mapsto x^2$.
\nuovo{Since the control potential is bounded, the input-output mapping $u\mapsto 
\Upsilon^u_T$ of \eqref{Eq:Potential} admits a unique continuous extension
to $L^1([0,T],\R)$.} The techniques introduced in the present analysis provide the following 
estimates from above for the attainable set when using different classes of admissible controls.
\begin{proposition}\nuovo{Let $T>0$ and $W :x\mapsto x^2$. Then:}
 \begin{itemize}
  \item 
 The attainable set from $\varphi_1$ with $L^1$ controls,
 $$
 \displaystyle{\mathcal{A}tt_{L^1}(\varphi_1)=\bigcup_{T\geq 0}
 \bigcup_{u \in L^1([0,T],\R)}\{\Upsilon^u_{t}\varphi_1|0\leq t \leq T\},}
 $$
satisfies 
$\displaystyle{\mathcal{A}tt_{L^1}(\varphi_1) \subset \bigcap_{s<5/2} 
H^s_{(0)}((0,1),\C)}$. 

\item
The attainable set from $\varphi_1$ with bounded variation ($BV$) controls, 
$$
\displaystyle{\mathcal{A}tt_{BV}(\varphi_1)=\bigcup_{T\geq 0}
 \bigcup_{u \in BV{((0,T],\R)}}\{\Upsilon^u_{t}\varphi_1|0\leq t \leq T\}},
$$
 is a $H^s$-dense  subset of
 $\{\psi\in L^2((0,1),\C)\mid \|\psi\|=1\}\cap H^{s}_{(0)}((0,1),\C)$
for every $s<9/2$.
 \end{itemize}
\end{proposition}

\begin{proof}
See Section~\ref{Sec:PotentialWell}.
\end{proof}

\subsection{Contents}

In Section~\ref{sec:abstract-framework}, we consider  bilinear evolution
 equations (not necessarily conservative) from an abstract point of view and
 we define the solution for controls with bounded variations. We also prove the
 well-posedness
 within this framework and the continuity of the propagators with respect
 to the control parameters.
 In Section~\ref{Sec:Interaction}, we use a reparametrization, inspired by the widely used 
 interaction framework, to extend the results of Section~\ref{sec:abstract-framework} to the case where 
 the control is a Radon measure. This provides a proof of Theorem \ref{THE_main_introduction}.  
  Section~\ref{Sec:HigherOrderFEPS} 
 is devoted to the regularity analysis of the solution obtained so far when further 
 assumptions are made on the control potential and to the proof of Theorem~\ref{THE_intro_weak_coupling}. 
    Section~\ref{Sec:BoundedPotential} is dedicated to  
 the case where $B$ is bounded and to the proof of Proposition 
\ref{PRO_introduction_BMS_Bborne}.
Section~\ref{sec:example} presents various examples.
The appendices contain notations and technical tools useful for the rest of the analysis.

\paragraph*{Acknowledgments.}
We acknowledge for this work the  financial support of the INRIA Color, the CNRS D\'efi InFIniti project DISQUO and the French Agence Nationale de la Recherche, contract number 17-CE40-0007-01. This long term
analysis was possible due to the support of our respective institutions (the universities of Bourgogne Franche-Comt\'e and 
Lorraine and the CNAM) as well as the facilities offered by the CNRS, the PIMS and the University of Victoria during the stay
of the first author and the facilities of the University of La Laguna, during the stay of the second author.

We are also grateful to many colleagues for the useful discussions that lead to many 
improvements since our first result, among them we especially thank Farid Ammar-Khodja, Alain Haraux, Marcelo Laca, and  Gilles Lancien. 
We are also grateful to the anonymous referee for his carefull reading and for the many comments which helped us in  improving the present work.

\section{Well-posedness and continuity for BV controls}\label{sec:abstract-framework}

In this section, we present  global well-posedness results for a
class of nonautonomous perturbations of a maximal dissipative linear Cauchy problem 
as well as a continuity criterion for a convergence problem.

\subsection{Abstract framework: definitions and notations}

We consider
a general framework for bilinear dynamics in Hilbert spaces. Classical
definitions and tools in this context can be found in \cite[Section X.8]{reed-simon-2}, as well as the associated notes and problems. Notice that however we consider an opposite sign for the generators, thus, following~\cite{Phillips}, we use the word
\emph{dissipative} instead of \emph{accretive} (see also~\cite[Notes of Section
X.8]{reed-simon-2}). As we restrict our analysis to the Hilbert space framework, the notion of
generators of contraction semigroups  (linear maps with norm less than one) and maximal dissipative 
operators coincide (see~\cite[Theorem 1.1.3]{Phillips}). The equivalence between these two notions is used in  our analysis at many levels, in particular, for what concerns {mild coupling} in Section~\ref{Sec:HigherOrderFEPS}.

Let $\Hc$ be a Hilbert space (possibly infinite dimensional) with scalar 
product $\langle \cdot,\cdot\rangle$ and  corresponding norm $\|\cdot\|$. Let
$A, B$ be two (possibly unbounded) dissipative operators on $\Hc$.  We consider the
formal bilinear control system
\begin{equation}\label{eq:main}
\frac{d}{dt} \psi(t)=A \psi(t) + u(t) B \psi(t),
\end{equation}
where the scalar control $u$ is to be chosen in a set of real functions.

In general, given an initial data $\psi(0)=\psi_0 \in \Hc$, the solution of
system~\eqref{eq:main} may not be well-defined. Indeed, even the definition of $A+B$ is
not obvious when $A$ and $B$ are unbounded. To this aim it is
usually assumed that the operators $A$ and $B$ satisfy the following
condition.
\begin{definition}\label{Def:KatoRellich}
Let $(A,B)$ be a couple of operators acting on $\Hc$. Then $B$ is said \emph{relatively bounded} with respect to $A$, or $A$-bounded, if 
$D(A) \subset D(B)$ and there exist $a, b>0$ such that for every $\psi$ in $D(A)$,
$\|B \psi \|\leq a \|A \psi\| +b \|\psi \|$.
\end{definition}
It is well-known that if $A$ is skew-adjoint and $B$ skew-symmetric, from Kato--Rellich Theorem, (see for
example~\cite[Theorem X.12]{reed-simon-2}),
if $B$ is relatively bounded with respect to $A$, then for every real constant
$u$ such that $|u| < 1/a$ (with $a$ from Definition~\ref{Def:KatoRellich}), $A+uB$ is skew-adjoint with domain $D(A)$ and
generates a group of unitary operators. System~\eqref{eq:main} is then
well-posed for every initial condition. 
From~\cite[Corollary to
Theorem X.50]{reed-simon-2}, 
$A+uB$ is maximal dissipative with domain $D(A)$ and
generates a contraction semigroup when $A$ is maximal dissipative, $B$ is dissipative, $B$ is $A$-bounded and  $0\leq u < 1/a$ (again $a$ is from Definition~\ref{Def:KatoRellich}).

In most of the examples presented in Section~\ref{sec:example} below, we consider the skew-adjoint case 
and $a$ arbitrary small, so that we can define the solutions of~\eqref{eq:main} for
every piecewise constant control $u$ with real values.

In the general case, we will refer to the following assumptions.
\begin{assumption}\label{ass:ass}
$(A,B,K)$ is a triple where $A$ is a maximal dissipative operator on $\Hc$, $B$ is an operator on $\Hc$ with $D(A)\subset D(B)$,
and $K$ a real interval containing $0$, such that for any $u\in K$, $A+uB$ is 
a maximal dissipative operator on $\Hc$ with domain $D(A)$.
\end{assumption}

Assumption~\ref{ass:ass} implies that the operator $B$ is $A$-bounded from $D(A)$ to $\Hc$ and 
allows us to define
\begin{equation}\label{def:Abound}
 \|B\|_{A}:=\inf_{\lambda >0}
\|B(\lambda-A)^{-1}\|.
\end{equation}
The number $\|B\|_{A}$ is the lower bound of all possible constants $a$ in Definition~\ref{Def:KatoRellich} and in principle it can be zero. We also have,
\begin{equation}\label{def:Abound2}
 \|B\|_{A}=\liminf_{\lambda\to+\infty }
\|B(\lambda-A)^{-1}\|.
\end{equation}

We  consider also the following assumption in order to extend the definition of propagator to the case of Radon measures controls (see Section~\ref{sec:radon}).
\begin{assumption}\label{ass:Interacting}
$(A,B,K)$ is a triple where $A$ is a maximal dissipative operator on $\Hc$, 
$K$ a real interval containing $0$, and  
\begin{enumerate}[label={{(A\arabic{assumption}.\arabic{*})}}, ref={{(A\arabic{assumption}.\arabic{*})}}]
\label{ASS_A_skew_adjoint2}
 \item there exists $c\geq 0$ and $c'\geq 0$ such that $B-c$ 
and $-B-c'$ generate contraction semigroups on $\Hc$ leaving $D(A)$ invariant,
\label{ASS_B_skew_symmetric}
\item for every $u\in \radon{[0,T]}$, with $u((0,t])\in K$ for any $t\in[0,T]$,
\[
t\in [0,T]\mapsto \mathcal{A}(t):=e^{u((0,t])B}Ae^{-u((0,t])B}, 
\]
is a family of
maximal dissipative operators with common domain $D(A)$
such that :
\begin{itemize}
 \item $\sup_{t\in [0,T]}
\left\|(1-\mathcal{A}(t))^{-1}\right\|_{{L}(\Hc,D(A))}<+\infty$,
 \item ${\mathcal A}$ has finite total variation from $[0,T]$ to $L(D(A),\Hc)$.
\end{itemize}
\label{ASS_calA_bounded_variation}
\end{enumerate}
\end{assumption}

\nuovo{A basic example of operators satisfying Assumption~\ref{ass:Interacting} is given by $A=-\partial_x$ and $B={\mathrm i}W(x)$, the 
operator of multiplication by a smooth bounded potential ${\mathrm i}W$ ($W$ real-valued),
acting on $L^2(\R)$ (the set $K$ being the real line $\mathbf{R}$). Then $\mathcal{A}(t)=-\partial_x + u((0,t])(\partial_xW)$.}

\begin{remark}
From Assumption~\ref{ASS_A_skew_adjoint2}, $B$ et $-B$, with same domains, are generators of continuous semigroups. We can prove $e^{-tB}=(e^{tB})^{-1}$, for any real $t$, and thus $B$ generates a continuous group. 
\end{remark}

The triple $(A,B,K)$ satisfies Assumption~\ref{ass:Interacting} for any 
interval $K$ containing $0$ if  the pair $(A,B)$ satisfies the following one.
\begin{assumption}\label{ass:StrongInteraction}
$(A,B)$ is a pair such that 
\begin{enumerate}[label={{(A\arabic{assumption}.\arabic{*})}}, ref={{(A\arabic{assumption}.\arabic{*})}}]
 \item $A$ is a maximal dissipative operator on $\Hc$ with domain
$D(A)$,
\label{ASS_A_skew_adjoint3}
 \item there exists $c\geq 0$ and $c'\geq 0$ such that $B-c$ 
and $-B-c'$ generate contraction semigroups on $\Hc$ leaving $D(A)$ invariant,
\label{ASS_B_skew_symmetric2}
\item the map $t\in \R \mapsto e^{t B}Ae^{-t B} \in
L(D(A),\Hc)$ is locally Lipschitz.\label{ass:commutator}
\end{enumerate}
\end{assumption}
\begin{remark}
Assumption~\ref{ass:commutator} is a strong assumption on the regularity of $B$ 
with respect to the scale of $A$. Indeed it implies that $B$ is the generator of 
a strongly continuous semigroup on $D(A)$ since the semigroups generated by $B$ or $-B$ are continuous on $\Hc$ from Assumption~\ref{ASS_B_skew_symmetric2} and 
\begin{align*}
 \left\|Ae^{-tB}\psi-A\psi\right\|&\leq e^{c't}\left\|e^{tB}Ae^{-tB}\psi-e^{tB}A\psi\right\|\\ 
&\leq e^{c't}\left\|e^{tB}Ae^{-tB}-A\|_{L(D(A),\Hc)}\|\psi\|_{D(A)}+\|e^{tB}A\psi-A\psi\right\|,
\end{align*}
for $t>0$ and $\psi \in D(A)$,
which provides the continuity on $D(A)$. In Section~\ref{Sec:HigherOrderFEPS} below, we consider higher regularity assumptions in the skew-adjoint case and operators on $D(|A|^k)$ with $k >1$. 
\end{remark}

\subsection{Propagators}\label{Sec:Propagators}

Since the problem~\eqref{eq:main} is nonautonomous, the notion of semigroup is replaced by the following 
\begin{definition}[Propagator on a Hilbert space]\label{Definition:Propagator}
 A family $(s,t)\in
\Delta_I\mapsto X(s,t)$ of linear contractions on a Hilbert space $\Hc$,
strongly
continuous in $t$ and $s$  and such that 
\begin{itemize}
\item[$(i)$]  $X(t,s)=X(t,r)X(r,s)$, for any $s<r<t$,
\item[$(ii)$] $X(t,t)=I_{\Hc}$,
\end{itemize}
is called a \emph{contraction propagator} on $\Hc$.
\end{definition}
\begin{remark}
In Section~\ref{Sec:Interaction} below, we
introduce a \nuovo{notion of generalized propagators,}
see Definition~\ref{def:propagator-radon}, with
relaxed assumptions on the continuity of
$(s,t)\mapsto X(s,t)$ in order to extend it to the framework of Radon measure controls. 
\end{remark}

Following \cite{Kato1953} in the construction of propagators, we introduce the following
\begin{assumption}\label{ass:kato-generale}
Let $\Dc$ be a dense subset of $\Hc$
\begin{enumerate}[label={{(A\arabic{assumption}.\arabic{*})}}, ref={{(A\arabic{assumption}.\arabic{*})}}]
\item \label{ASS:maximalDissipative} $A(t)$ is a maximal dissipative operator on $\Hc$ with domain $\Dc$ for every $t\in I$,
\item \label{Ass:BV} $t\mapsto A(t)$ has bounded variation from $I$ to $L(\Dc,\Hc )$, 
where $\Dc$ is endowed with the graph topology associated with $A(a)$ for some $a \in  
I$, 
 \item \label{Ass:ResolventBound}
 $
 M:=\sup_{t\in I} \left\|(1-A(t))^{-1}\right\|_{L(\Hc,\Dc)}< \infty.
 $
\end{enumerate}
\end{assumption}

In the following, Assumption~\ref{ass:kato-generale} will apply mainly to the family of operators 
$A(t) = A+u(t)B$ or $A(t) = e^{ - u((0,t])B}\!Ae^{u((0,t])B}$.

\begin{remark}
In Assumption \ref{Ass:BV}, the bounded variation of $t\mapsto A(t)$ ensures that any choice of 
$a\in I$ will be equivalent.
\end{remark}

\begin{remark}\label{Remark:AssumptionKato}
%
We do not assume $t \mapsto A(t)$ to be continuous. However, as a
consequence of Assumption~\ref{Ass:BV} (see~\cite[Theorem~3]{Edwards})
it admits right and left limit in $L(\Dc,\Hc )$, denoted by $A(t - 0):=\lim_{\varepsilon
\to 0^{+}}{A(t-\varepsilon)}$ and $A(t + 0) :=\lim_{\varepsilon
\to 0^{+}}A(t+\varepsilon)$, for all $t \in I$, and $ A(t-0)= A(t+0)$ for all $t
\in I$ except, at most, a countable set.
\end{remark}

The core of our analysis is the following result due to Kato 
(see~\cite[Theorem 2 and Theorem 3]{Kato1953}) providing sufficient conditions 
for the well-posedness of system~\eqref{eq:main}.
\begin{theorem}\label{thm:kato}
If $t \in I \mapsto A(t)$ satisfies
Assumption~\ref{ass:kato-generale}, then
there exists a unique contraction propagator 
$X:\Delta_{I}\to L(\Hc)$ such that 
if $\psi_0\in \Dc$ then $X(t,s)\psi_0\in \Dc$ and is strongly right differentiable 
in $t$ with derivative $A(t+0)X(t,s)\psi_0$.

Moreover, with $M$ from Assumption~\ref{Ass:ResolventBound}, 
$$
 \|A(t)X(t,s)\psi_0\|\leq M e^{M\TV{A}{I}{L(\Dc,\Hc)}}\|A(s)\psi_0\|, \quad \mbox{ for } (t,s)\in \Delta_{I}\mbox{ and } \psi_0 \in \Dc,
$$
and
$X(t,s)\psi_0$ is  left 
differentiable in $s$ with derivative
$-A(s-0)\psi_0$ when $t=s$.

In the case in which $t \mapsto A(t)$ is continuous and
skew-adjoint,
if $\psi_0\in \Dc$  then $t\in(s,+\infty)\mapsto X(t,s)\psi_0$  is 
strongly continuously differentiable in $\Hc$ with derivative $A(t)X(t,s)\psi_0$.
\end{theorem}
\begin{proof}
The statement of this theorem is  obtained by collecting statements 
of~\cite{Kato1953}. 
The point not clearly stated in~\cite{Kato1953} is 
the existence of $C>0$ such that
\[
 \|A(t)X(t,s)\psi_0\|\leq C\|A(s)\psi_0\|,
\]
 for $(t,s)\in \Delta_{I}$  and for any $\psi_0\in \Dc$. This is in~\cite[\S 3.10]{Kato1953} with $C=M\exp(MN)$ and 
\[
M =\sup_{t\in I} \left\|(1-A(t))^{-1}\right\|_{L(\Hc,\Dc)}
\quad\text{ and }\quad N = \TV{A}{I}{L(\Dc,\Hc)}.\qedhere
\]
\end{proof}

We call $t\mapsto X(t,s)\phi_0$ a ``mild'' solution in $\Hc$ of
\begin{equation}\label{Eq:NonAutonomousCauchy}
 \begin{cases}
  \frac{d}{dt} \phi(t)=A(t)\phi(t),\\
\phi(s)=\phi_0,
 \end{cases}
\end{equation}
even if, in general, it is not differentiable. 

\begin{remark}
If $(A,B,K)$ satisfies Assumption~\ref{ass:Interacting}, the operator $t \in [0,T] 
\mapsto \mathcal{A}(t)\!:=e^{u((0,t])B}\!Ae^{-u((0,t])B}$ defined in
Assumption~\ref{ASS_calA_bounded_variation} satisfies 
Assumption~\ref{ass:kato-generale} for any Radon measure $u$ on $(0,T)$ with 
 $u((0,t])\in K$ for any $t\in(0,T]$. If $(A,B)$ satisfies Assumption~\ref{ass:StrongInteraction} then
$(A,B,\R)$ satisfies Assumption~\ref{ass:Interacting}.
\end{remark}

The fact that Assumption~\ref{ass:ass} is stronger, in some sense, than  
Assumption~\ref{ass:kato-generale} is the content of the following 
lemma.
\begin{lemma}\label{Rem:NonInteractingCase}
If $(A,B,K)$  satisfies Assumption~\ref{ass:ass} and $u:[0,T]\mapsto K$ 
has bounded total variation such that $\overline{u([0,T])}\subset K$
then $A(t):=A+u(t)B$ satisfies
Assumption~\ref{ass:kato-generale} with $I=[0,T]$. 
\end{lemma}
\begin{proof}
The only point to verify is Assumption~\ref{Ass:ResolventBound}. 
First the set $C:=\overline{u([0,T])}$ is a bounded closed subset of $K$ and thus is a compact of $K$.
Then the map
\[
u\mapsto (1-A)(1-A-uB)^{-1},
\]
is continuous from $K$ to $L(\Hc)$. Indeed
\begin{align*}
(1-A)(1-A-uB)^{-1}-&(1-A)(1-A-v B)^{-1}\\
&=(1-A)\left((1-A-uB)^{-1}-(1-A-v B)^{-1}\right)\\
&=(v-u)(1-A)\left((1-A-uB)^{-1}B(1-A-v B)^{-1}\right)\\
&=(v-u)(1-A)\left((1-A-uB)^{-1}B(1-A)^{-1}(1-A)(1-A-v B)^{-1}\right)
\end{align*}
so that
\begin{align*}
 (1-&A)(1-A-u B)^{-1}- (1-A)(1-A-v B)^{-1}\\
&-(v-u)(1-A)(1-A-uB)^{-1}B(1-A)^{-1}\left((1-A)(1-A-v B)^{-1}-(1-A)(1-A-u B)^{-1}\right)\\
&=(v-u)(1-A)\left((1-A-uB)^{-1}B(1-A)^{-1}(1-A)(1-A-u B)^{-1}\right).
\end{align*}
Define
\[
 L(u)=\|(1-A)(1-A-uB)^{-1}\|_{L(\Hc)}\quad\text{and}\quad b=\|B(1-A)^{-1}\|,
\]
so that
\begin{equation}\label{EQ_norme_difference_resolvante}
 (1-|v-u|bL(u))\|(1-A)(1-A-uB)^{-1}-(1-A)(1-A-v B)^{-1}\|
\leq|v-u|L(u)^2b,
\end{equation}
which provides the desired continuity at $u$. Then as $|u(t)-u(0)|\leq \|u\|_{BV(I)}$ for any $t\in I$,
$u(t)$ is in $C$ a compact subset of $K$ for all $t\in I$ thus the closure of its image is compact and
\[
 t\in I\mapsto \|(1-A-u(t)B)^{-1}\|_{L(\Hc,\Dc)}
\]
is bounded.
\end{proof}

\subsection{Continuity}
In this section we focus on the continuity of the propagators with respect to the control $u$. The main tool, Proposition~\ref{prop:continuity} below, is a consequence of the work~\cite{Kato1953} by Kato.

\begin{definition}
Let $(A_n)_n$ be a family of generators of contraction semigroups and $A$ a generator of a contraction semigroup.
The family $(A_n)_n$
tends to $A$ in the \emph{strong resolvent 
sense} if 
$$
(\lambda-A_n)^{-1} \phi \to (\lambda-A)^{-1}\phi\quad \mbox{ as } n \to \infty,
$$
for every $\phi$ in $\Hc$ and for some $\lambda >0$ (and hence all $\lambda >0$, see 
\cite[Section~VIII.7]{reed-simon-1}).   
\end{definition}

\begin{proposition}\label{prop:continuity} 
Let $(A_{n})_{n\in \N}$ and $A$  satisfy Assumption~\ref{ass:kato-generale}.  
Let 
$(\Dc_n)_{n\in \N}$ and 
$\Dc$ be their respective domains (for any
$t\in I$). 
Let $X_n$ (respectively $X$) be the contraction propagator associated with
$A_n$ (respectively $A$).

Assume that:
\begin{itemize}
 \item[$(i)$] $\sup_{n\in N} \sup_{t\in I}\|(1-A_{n}(t))^{-1}\|_{L(\Hc,\Dc_n)} < 
+ \infty$,
 \item[$(ii)$] $A_{n}(\tau)$ converges to $A(\tau)$ in the strong resolvent 
sense for
almost every $\tau \in I$ as $n \to \infty$,
 \item[$(iii)$] $\sup_{n\in \N}
 \TV{A_{n}}{I}{L(\Dc_n,\Hc)} < + \infty$, 
 \item[$(iv)$]
 For every $\phi \in \Hc$, $\delta >0$, $n\in \N$ there exists $\psi^n \in \Dc_n$ with $\|\phi - \psi^n\|<\delta$ such that $\sup_{n\in \N}\|A_{n}(a) \psi^n\| < +\infty$ for some $a\in 
I$.

\end{itemize}
Then $X_{n}(t,s)$ tends strongly to $X(t,s)$ locally uniformly in $s,t\in 
\Delta_I$.
\end{proposition}
\begin{proof}
Using~\cite[\S 3.8]{Kato1953} it is sufficient to prove the statement for piecewise constant operator-valued functions (i.e. replacing $X_n$ and $X$ by any of their Riemann products) as follows: Let 
$\Delta:=\{s=t_0<t_1<\ldots<t_n=t\}$ be  a partition of the interval
$(t,s)$ and $X_{n}(\Delta)$ be the propagator associated with
$
\displaystyle \sum_{j =1}^n A_{n}(t_{j-1})
\chi_{[t_{j-1},t_{j})}$.
Then, from \cite[Equation (3.16)]{Kato1953}, for every $n$,
$$
\|\left(X_{n}(t,s;\Delta)-X_{n}(t,s)\right)\phi\|\leq \overline{M}e^{\overline{M}\, \overline{N}}\overline{N}|\Delta|\|A_n(a)\phi\|, \mbox{ for every } \phi \in \Dc_n
$$
where
\begin{align*}
 \overline{M} &=\max  \{\sup_{t\in I} \sup_{n\in \N} \left\|(1-A_n(t))^{-1}\right\|_{L(\Hc,\Dc_n)},
 \sup_{t\in I} \left\|(1-A(t))^{-1}\right\|_{L(\Hc,\Dc)}\},\\
 \overline{N} &= \max\{ \sup_{n\in\N}\TV{A_{n}}{I}{L(\Dc_n,\Hc)}, \TV{A}{I}{L(\Dc,\Hc)}\},
\end{align*}
and $|\Delta|=\sup_{1\leq j \leq n}|t_j-t_{j-1}|$. Similarly we define 
$X(\Delta)$ as the propagator associated with
$$
\displaystyle \sum_{j =1}^n A(t_{j-1})
\chi_{[t_{j-1},t_{j})}.
$$ 
We have
$$
\|\left(X(t,s;\Delta)-X(t,s)\right)\phi\|\leq \overline{M}e^{\overline{M}\, \overline{N}}\overline{N}|\Delta|\|A(a)\phi\|, \mbox{ for every } \phi \in \Dc.
$$
Following the
proof of \cite[Theorem X.47a (Hille--Yosida)]{reed-simon-2} (see also
Proposition~\ref{Prop:ThmHilleYosida} below), we have that 
\[
 \left\|e^{t A_n(\tau)}\phi - e^{t A^\lambda_n(\tau)}\phi\right\| \leq t 
\left\|A_n(\tau)\phi -A^\lambda_n(\tau)\phi\right\|,\quad  \mbox{ for every } \phi \in 
\mathcal{D}_n,
\]
with $A^\lambda_n(\tau):=\lambda(\lambda-A_n(\tau))^{-1}A_n(\tau)$, for
$\lambda >0$. \nuovo{This estimates can be obtained by integrating for $s\in[0,t]$ the derivative with respect to $s$ of
$$
 e^{s A_n(\tau)}e^{(t-s) A^\lambda_n(\tau)}\phi,
$$
using the fact that $A_n(\tau)$ and $A^\lambda_n(\tau)$ commute, the triangle inequality, and the fact that the $e^{s A_n(\tau)}$  and $e^{(t-s) A^\lambda_n(\tau)}$ are contractions (see, for instance,~\cite[Proof of Theorem X.47a (Hille--Yosida)]{reed-simon-2}.}
 
We also have
\[
 \left\|e^{t A(\tau)}\phi - e^{t A^\lambda(\tau)}\phi\right\| \leq t 
\left\|A(\tau)\phi -A^\lambda(\tau)\phi\right\|,\quad  \mbox{ for every } \phi \in 
\mathcal{D},
\]
with $A^\lambda(\tau):=\lambda(\lambda-A(\tau))^{-1}A(\tau)$. 

Since $A_n$  are generators of contraction semigroups, then 
$\|\lambda(\lambda-A_n(\tau))^{-1}\| \leq 1$ for every $\lambda>0$, in particular it is uniformly bounded in $n$ and $\tau$. 

By assumption $(iv)$ for every $\phi\in\Hc$ and $\delta>0$ there exist $\psi \in \Dc$ and $\psi^n\in\Dc_n$ such 
that 
$$
 \|\phi-\psi\|\leq \delta\quad\mbox{ and }\quad \|\phi-\psi^n\|\leq \delta,
$$
and 
$\sup_{n\in \N}\|A_{n}(a) \psi^n\| < +\infty$ for $a\in I$.
We deduce that $\lambda(\lambda-A_n(\tau))^{-1}\psi^n$ tends to $\psi^n$ as $\lambda\to \infty$ uniformly in $n$ and $\tau$. 
Similarly $A^\lambda(\tau)\psi$ tends strongly to $A(\tau)\psi$ 
uniformly in $\tau$ as $\lambda \to \infty$. 
So that
\begin{eqnarray*}
\left  \| e^{tA(\tau)}\phi -e^{tA_n(\tau)}\phi \right \|& 
\leq & 2\delta+\left \|
e^{tA(\tau)}\psi -
e^{tA^\lambda(\tau)}\psi \right \|+ \left \| e^{tA^\lambda (\tau)}\phi
-e^{tA^\lambda_n(\tau)}\phi \right 
\|+
\left \| e^{tA^\lambda_n (\tau)}\psi^n -e^{tA_n(\tau)}\psi^n \right \|\\
&\leq& 
2 \delta+ t \left \| A (\tau)\psi -A^\lambda(\tau)\psi \right \| + 
\left \| e^{tA^\lambda (\tau)}\phi
-e^{tA^\lambda_n(\tau)}\phi \right \| +
t \left \| A_n
(\tau)\psi^n -
A^\lambda_n(\tau)\psi^n \right \|.
\end{eqnarray*}
It is sufficient to show convergence of $\left \| e^{tA^\lambda (\tau)}\phi
-e^{tA^\lambda_n(\tau)}\phi \right \|$ as $n \to \infty$ in order to conclude the proof. 
Since $e^{t A_n^\lambda(\tau)}=e^{-\lambda t } e^{t \lambda^2
(\lambda-A_n(\tau))^{-1}}$ and $e^{t A^\lambda(\tau)}=e^{-\lambda t } e^{t
\lambda^2
(\lambda-A(\tau))^{-1}}$ (see \cite[Theorem X.47a
(Hille-Yosida)]{reed-simon-2}), we have that 
\begin{align*}
 \left \| e^{tA^\lambda (\tau)}\phi -e^{tA^\lambda_n(\tau)}\phi \right \| & = 
\left \|e^{-\lambda t } e^{t \lambda^2
(\lambda-A_n(\tau))^{-1}} \phi - e^{-\lambda t } e^{t \lambda^2
(\lambda-A(\tau))^{-1}} \phi \right \|\\
& = e^{-\lambda t} \left \|e^{t \lambda^2
(\lambda-A_n(\tau))^{-1}} \phi - e^{t \lambda^2
(\lambda-A(\tau))^{-1}} \phi \right \|.
\end{align*}
Now, we have that $\|
(\lambda-A_n(\tau))^{-1} \| \leq \frac{1}{\lambda}$ (see Proposition
\ref{Prop:ThmHilleYosida} below for $\omega=0$) and hence $\|e^{t\lambda^2
(\lambda-A_n(\tau))^{-1}} \| \leq e^{\lambda t} $.
Duhamel's identity then writes, for $0\leq t\leq T$,
\begin{eqnarray}
\lefteqn{
 \left \|e^{t \lambda^2
(\lambda-A_n(\tau))^{-1}} \phi - e^{t \lambda^2
(\lambda-A(\tau))^{-1}} \phi \right \|} \nonumber \\
 & =& \left \| \int_0^t \lambda^2 e^{(t-s) \lambda^2
(\lambda-A_n(\tau))^{-1}} \left \{
(\lambda-A_n(\tau))^{-1} - 
(\lambda-A(\tau))^{-1} \right \} e^{s \lambda^2
(\lambda-A(\tau))^{-1}} \phi \; \mathrm{d}s \right \|
\label{EQ_identite_resolvante_Lebesgues}\\
&\leq &  \lambda^2 e^{T \lambda}
 \int_0^T \left \| \left \{
(\lambda-A_n(\tau))^{-1} - 
(\lambda-A(\tau))^{-1} \right \} e^{s \lambda^2
(\lambda-A(\tau))^{-1}} \phi \right \|\; \mathrm{d}s. \nonumber 
\end{eqnarray}
The result follows from Lebesgue Dominated Convergence Theorem, using the 
convergence of  
$A_{n}(\tau)$  to $A(\tau)$ in the strong resolvent sense for
almost every $\tau \in I$ as $n$ tends to infinity. 
\end{proof}
\nuovo{
\begin{lemma}\label{Lem:AssumptionsPropositionContinuity}
Let $(A_{n})_{n\in \N}$ and $A$ 
satisfy Assumption~\ref{ass:kato-generale} with a common domain $\Dc$ (for any
$t\in I$ and any $n\in \N$). 
Let $X_n$, respectively $X$, be the contraction propagator associated with
$A_n$, respectively $A$.

Then the assumptions of Proposition~\ref{prop:continuity} are
verified whenever:
\begin{itemize}
 \item[$(i)'$] $\sup_{n\in \N} \sup_{t\in I}\|(1-A_{n}(t))^{-1}\|_{L(\Hc,\Dc)} < + \infty$,
 \item[$(ii)'$] $A_{n}(\tau)$ converges to $A(\tau)$ in the strong sense in $\Dc$ for almost every $\tau \in I$ as $n \to \infty$,
 \item[$(iii)'$] $\sup_{n\in\N}
 \TV{A_n}{I}{L(\Dc,\Hc)} < + \infty$. 
\end{itemize}
\end{lemma}
\begin{proof}
Assumptions $(i)$ and $(iii)$ of Proposition~\ref{prop:continuity} coincide, respectively, with assumptions $(i)'$ and $(iii)'$.

We have for any $\phi$ in $\Hc$
\[
 (1-A_n(t))^{-1}\phi-(1-A(t))^{-1}\phi=(1-A_n(t))^{-1}(A(t)-A_n(t))(1-A(t))^{-1}\phi
\]
and hence 
\[
 \left\|
 (1-A_n(t))^{-1}\phi-(1-A(t))^{-1}\phi
 \right\|\leq \left\|(A(t)-A_n(t))(1-A(t))^{-1}\phi\right\|.
\]
Since $(1-A(t))^{-1}\phi\in \Dc$ by assumption $(ii)'$ we conclude that $(ii)$ of Proposition~\ref{prop:continuity} is verified as well.


As $\Dc$ is dense for every $\phi \in \Hc$, $\delta >0$, there exists $\psi \in \Dc$ with $\|\phi - \psi\|<\delta$.
Since for any $a\in I$ $A_n(a)$ converges strongly to $A(a)$ in $\Dc$, $\sup_{n\in \N}\|A_{n}(a) \psi\| < +\infty$.
This is assumption $(iv)$ of Proposition~\ref{prop:continuity} 
\end{proof}
}
\begin{corollary}\label{cor:cont_input_output_bv}
Let $(A,B,K)$ satisfy Assumption~\ref{ass:ass}. Let $(u_{n})_{n\in N} $ be a sequence in $BV(I,K)$  converging to $u \in BV(I,K)$. 
Let ${A}_{n}(t) = A+u_n(t)B$, ${A}(t) =
A+u(t)B$ and
let $X_n$, respectively $X$, be the contraction propagators associated with
${A}_n$, respectively ${A}$. If $\overline{\cup_{n \in \N} u_n([0,T])} \subset K$,
then $X_{n}(t,s)$ tends strongly to $X(t,s)$ locally uniformly in $(s,t)\in \Delta_I$.
\end{corollary}

\begin{proof}

The proof consists in verifying that the hypotheses of Proposition \ref{prop:continuity} are satisfied. To this aim, we just have to check items $(i)'$, $(ii)'$ and $(iii)'$ of  Lemma~\ref{Lem:AssumptionsPropositionContinuity}.

Assumption $(i)'$:    the mapping $L: s\in K \mapsto \|(1-A)(1-A-sB)^{-1}\|$ has been defined in the proof of Lemma \ref{Rem:NonInteractingCase}
where it is shown to be continuous. By hypothesis, there exists a compact set $K_1 \subset K$ such that for every $n$ in $\N$ and every $t$ in $[0,T]$, $u_n(t) \in K_1$. Hence, 
$\sup_{n \in \N} \sup_{y \in [0,T]} \|(1-A)(1-A-u_n(t)B)^{-1}\| \leq C(K_1) <+\infty$ which proves point $(i)'$.

Assumption $(ii)'$ follows from the assumption that $(u_n)_{n\in \N}$ converges to $u$.

Assumption $(iii)'$: for every $n$ in $\N$ ,
$$
\TV{A_n}{I}{L(\Dc,\Hc)}  =\TV{u_n B}{I}{L(\Dc,\Hc)} = \|B\|_{L(\Dc,\Hc)} \TV{u_n}{I}{\R}. 
$$
This last quantity is bounded as $n$ tends to infinity since $(u_n)_{n\in \N}$ converges to $u$.  
\end{proof}

\begin{corollary}\label{COR_obstruction_contr_exacte_BV}
Assume that $(A,B,K)$ satisfy Assumption~\ref{ass:ass}.
Let $\psi_{0} \in \Hc$.  Then
$$
\left\{ \pro^{u}_{t}(\psi_{0})\mid u\in
BV([0,\infty),K), 
t\geq 0\right\}
$$
is contained in a countable union of
compact subsets of $\Hc$.
\end{corollary}
\begin{proof}
We follow the principle presented in Section \ref{SEC_principe_obstruction}. 
We first introduce a nondecreasing sequence $(K_i)_{i \in \N}$ of compact subsets of $K$ such that  $K=\cup_{i\in \N} K_i$, and the subsets 
$$\mathcal{Z}_{i,j,n}=\left\{u\in
BV([0,\infty),K_i), \TV{u}{[0,n]}{K_i} \leq j \right\}$$
of the set of functions with bounded variations. By Helly's selection Theorem, $\mathcal{Z}_{i,j,n}$ is sequentially compact.
 By Corollary \ref{cor:cont_input_output_bv}, the set 
$\{\pro^{u}_{t}(\psi_{0})\mid u \in \mathcal{Z}_{i,j,n}
\}$ is compact in $\Hc$ for every $(n,i,j)$ in $\N^3$. Hence
\begin{eqnarray*}
\lefteqn{\left\{ \pro^{u}_{t}(\psi_{0})\mid u\in
BV([0,\infty),K), 
t\geq 0\right\} \subset}\\
&& \cup_{n \in \N} \cup_{i \in \N} 
\cup_{j \in \N}
\{\pro^{u}_{t}(\psi_{0})\mid u \in \mathcal{Z}_{i,j,n}, 0\leq t \leq n\}
\end{eqnarray*} is contained is a countable union of compact sets of $\Hc$.
\end{proof}

The notion of convergence for a sequence of Radon measures is detailed in Appendix~\ref{sec:notations}.

\begin{corollary}\label{cor:continuity}
Let $(A,B,K)$ satisfy Assumption~\ref{ass:Interacting}. 
Let $I=[0,T]$ for some $T>0$. 
Let $(v_{n})_{n\in N} $ be a sequence in $\radon{I}$ converging to $v \in \radon{I}$. 
Assume that $v_n((0,t])\in K$  and $v((0,t])\in K$ for every $t\in(0,T]$ and $n \in \N$.
Let $\mathcal{A}_{n}(t) = e^{-v_{n}((0,t])B}Ae^{v_{n}((0,t])B}$ and $\mathcal{A}(t) =
e^{-v((0,t])B}Ae^{v((0,t])B}$ and
let $X_n$, respectively $X$, be the contraction propagators associated with
$\mathcal{A}_n$, respectively $\mathcal{A}$.
If $\sup_{n\in \N}
\TV{\mathcal{A}_{n}}{I}{L(D(A),\Hc)}  < + \infty$,
then $X_{n}(t,s)$ tends strongly to $X(t,s)$ locally uniformly in $(s,t) \in \Delta_I$.
\end{corollary}
\begin{proof}
The proof consists in checking that the assumptions of 
Proposition~\ref{prop:continuity} are fulfilled. Here $\Dc = D(A)$.
\begin{itemize}
\item[$(i)$] We have $\sup_{n\in N} \sup_{t\in
I}\|(1-\mathcal{A}_{n}(t))^{-1}\|_{L(\Hc,\Dc)} <\infty$. Indeed 
\begin{align}
\lefteqn{\| (1-A) (1-\mathcal{A}_{n}(t))^{-1}\|_{L(\Hc)}}\nonumber\\& =\| (1-A) 
e^{
v_{n}((0,t])B}(1-A)^{-1}e^{ -
v_{n}((0,t])B}\|_{L(\Hc)}\nonumber\\
&\leq \|e^{ 
v_{n}((0,t])B}\|_{L(\Hc)}\| e^{-v_{n}((0,t])B} (1-A)  e^{
v_{n}((0,t])B}(1-A)^{-1}\|_{L(\Hc)}\|e^{ -
v_{n}((0,t])B}\|_{L(\Hc)}\nonumber\\
& = \|e^{ 
v_{n}((0,t])B}\|_{L(\Hc)}\| (1-\mathcal{A}_{n}(t))
(1-A)^{-1}\|_{L(\Hc)}\|e^{ -
v_{n}((0,t])B}\|_{L(\Hc)} \nonumber\\
& \leq\|e^{ 
v_{n}((0,t])B}\|_{L(\Hc)}\left(\|\mathcal{A}_{n}(t) -
\mathcal{A}_{n}(0)\|_{L(\Dc,\Hc)} +
\|1-A\|_{L(\Dc,\Hc)}\right) \|e^{ -
v_{n}((0,t])B}\|_{L(\Hc)}\nonumber\\
& \leq \|e^{ 
v_{n}((0,t])B}\|_{L(\Hc)}\left(\TV{\mathcal{A}_{n}}{I}{L(\Dc,\Hc)} +
1\right)\|e^{ -
v_{n}((0,t])B}\|_{L(\Hc)}\label{Eq:EstimateOnM}.
\end{align}
Notice that since $(v_{n})_{n\in N}$ converges to $v$ then by definition
$v_{n}((0,t])$ is uniformly bounded in $n\in \N$ and $t\in [0,T]$.
Then from Assumption~\ref{ASS_B_skew_symmetric}, there exists 
$\omega \in \R$ such that
\begin{equation}\label{eq:EwpGrowthB} 
\|e^{ vB}\|_{L(\Hc)} \leq e^{\omega |v|}, \quad \mbox{ for every } v\in \R,
\end{equation}
which provides the desired boundedness.

 \item[$(ii)$] The sequence $\mathcal{A}_{n}(t)$ tends to $\mathcal{A}(t)$ in the 
strong resolvent sense for all $t \in [0,T]$ as $n \to \infty$.
Indeed from
\[
(1-\mathcal{A}_{n}(t))^{-1} - (1-\mathcal{A}(t))^{-1}=
e^{-v_n((0,t])B}(1-A)^{-1}e^{v_n((0,t])B} - e^{-v((0,t])B}(1-A)^{-1}e^{-v((0,t])B}
\]
we have
\begin{align*}
 (1-\mathcal{A}_{n}(t))^{-1} - (1-\mathcal{A}(t))^{-1}&=
(e^{-v_n((0,t])B}-e^{-v((0,t])B})(1-A)^{-1}e^{v_n((0,t])B}\\
&\quad +e^{-v((0,t])B}(1-A)^{-1}(e^{v_n((0,t])B}-e^{v((0,t])B})
\end{align*}
then using \eqref{eq:EwpGrowthB}
the boundedness of the sequence $(v_n)$ and the strong continuity of 
$t\in \R \mapsto e^{tB}$, we conclude the strong resolvent convergence.

\item[$(iii)$] By Assumption~\ref{ASS_calA_bounded_variation} we have $\sup_{n\in 
\N}\TV{\mathcal{A}_{n}}{I}{L(D(A),\Hc)} < + \infty$.
 \item[$(iv)$] Assumption $(iv)$ of Proposition~\ref{prop:continuity} follows from ${\mathcal A}_{n}(0)=A$ and the fact that
 the domain $\Dc$ of $A$ is dense in $\Hc$.\qedhere
\end{itemize}
\end{proof}

\begin{remark}\label{Rem:CommAssumption}
The last assumption of Corollary~\ref{cor:continuity}, namely $\sup_{n\in \N}
\TV{\mathcal{A}_{n}}{I}{L(\Dc,\Hc)} < + \infty$
for
$\mathcal{A}_{n}(t) = e^{-u_{n}((0,t])B}Ae^{u_{n}((0,t])B}$,
is a consequence of 
Assumption~\ref{ass:commutator} since this provides the existence of a real constant $L_I(A,B)$ such that 
for every $s,t \in I$,
\begin{equation}\label{eq:comm}
\|e^{-tB}Ae^{tB} -  e^{-sB}A e^{sB}\|_{L(\Dc,\Hc)} \leq  L_I(A,B) |t-s|.
\end{equation}

Notice also that with $s=0$ inequality~\eqref{eq:comm} reads
\begin{equation}\label{eq:comm2}
\|e^{-tB}Ae^{tB}\|_{L(\Dc,\Hc)} \leq  L_I(A,B) |t| + 1
\end{equation}
as $\|A \|_{L(\Dc,\Hc)} \leq 1$.
\end{remark}

\section{Interaction framework}\label{Sec:Interaction}

In this section we consider the framework of assumptions~\ref{ass:Interacting} or \ref{ass:StrongInteraction}. We show that these assumptions lead to a notion of weak solution for \eqref{eq:main} when the control is integrable and we provide the proofs of Theorem~\ref{THE_main_introduction} and Proposition~\ref{PRO_introduction_BMS_Bborne} in the general Radon measure case.

\nuovo{
\subsection{Heuristic} 
A classical method to deal with time-depending Hamiltonians, as it is the case of bilinear dynamics of the form~\eqref{EQ_bilinear_abstrait}, 
is to considered solution in the mild sense
$$
x(t) = e^{tA}x_0+\int_0^t e^{(t-s)A} u(s) B x(s) \mathrm{d}s, 
$$
using the well-known \emph{interaction picture}, consisting in a change of variable $y(t)=e^{-tA}x(t)$ which gives (formally) $y'(t)=u(t) e^{-tA} B e^{tA}y(t)$. The underlying idea is that the operator $uB$ is ``small'' with respect to $A$ and, hence, $y$ is expected to have slow variations. However, in our framework, since we consider control $u$ in the set of Radon measures, we need a different approach. Indeed when $u$ has atoms, the effect of $A$ on the dynamics is negligible compared to $uB$. So we consider the change of variables $y(t)=e^{-\int u(t) B} x(t)$, and hence (formally) $y'= e^{-\int u B} A e^{\int u B}y$, to finally obtain mild solutions of~\eqref{EQ_bilinear_abstrait} in a generalized mild sense
\begin{equation}\label{EQ_Duhamel_generalisee_formelle}
x(t)=e^{\int_0^t u B }x_0 + \int_0^t e^{\int_s^t u B} A x(s) \mathrm{d}s.
\end{equation}
 This heuristic is purely formal at this point  and the aim of the rest of this section is to formalize it rigorously. In the formal expression~\eqref{EQ_Duhamel_generalisee_formelle}, one of the difficulties lies in the interpretation of the term $\int_0^t u(s)\mathrm{d}s$ when $u$ is a Radon measure and $u(\{t\})\neq 0$: should the limits of $[0,t]$ be included in the integral? This is also the reason for the introduction of a non-standard notion of sequential convergence on the set of Radon measures, needed in order to ensure the continuity of the endpoint mapping.
}

\subsection{Generalized propagators}\label{sec:radon}
In this section, we explain the link between Assumption~\ref{ass:ass} and 
Assumption~\ref{ass:Interacting} and thus emphasize the fact that~\eqref{eq:main} admits
solutions associated with a Radon measure $u$.

We use the following result of approximation of Radon measures by piecewise constant functions.

\begin{lemma}\label{LEM_approximation_measure}
For every $u \in \radon{[0,T]}$  there exists a sequence $(u_n)_n$ 
of piecewise constant functions such that $\int_0^t u_n$ tends to
$u((0,t])$ and $\int_0^t |u_n|$ tends to
$|u|((0,t])$ for all 
$t$ in $[0,T]$ as $n$ tends to infinity with $\int_0^T |u_n|\leq
|u|((0,T])$ for every $n$. 
If $u$ is positive, the sequence  $(u_n)_n$ can be chosen such that $t \mapsto \int_0^t u_n(\tau) \mathrm{d}\tau$ is nondecreasing for every $n$.
If $t\mapsto u((0,t])$ is $M$-Lipschitz continuous on $[0,T]$ then $(u_n)_n$ can be chosen such that $|u_n|\leq M$.

\end{lemma}
\begin{proof}
It is not restrictive to prove the
statement for nonnegative Radon measures since by Hahn--Jordan decomposition, any
Radon measure $u$ is the difference of two nonnegative Radon measures with disjoint supports.

Let us assume $u$ nonnegative. Then $U:t\in (0,T] \mapsto u((0,t])$ is an 
nondecreasing function (with bounded variation). Except on an at most countable set, $U$ is 
continuous. So $U$ is the sum of an nondecreasing step function, possibly with an infinite number of steps, and a nondecreasing continuous function. 
Both can be pointwise approximated by  nondecreasing sequences
of nondecreasing continuous piecewise affine functions.

The last statement follows by considering approximation of Lipschitz continuous functions by continuous piecewise affine ones. 
\end{proof}

\begin{remark}\label{Rem:NonUniformDensity}
Lemma~\ref{LEM_approximation_measure} explains the reason for the choice of the notion of convergence for sequences on Radon measure, given in Appendix~\ref{sec:notations}, in stead of the maybe more natural total variation topology.
Indeed, for a positive $u \in \radon{[0,T]}$ the sequence $(u_n)_n$ 
of piecewise constant functions such that $\int_0^t u_n$ tends to
$u((0,t])$ pointwise is, in our construction, a nondecreasing sequence.
Since each $t\mapsto \int_0^t u_n$ is continuous, if $t\mapsto u((0,t])$ is not continuous then
a result such as Lemma~\ref{LEM_approximation_measure} in the total variation topology is excluded. Indeed the convergence in total variation of sequences of measures  implies the uniform convergence of the corresponding sequence of cumulative functions.
\end{remark}

\begin{definition}\label{def:propagator-radon}
Let $(A,B,K)$ satisfy Assumption~\ref{ass:Interacting}. Let 
$u\in \radon{[0,T]}$. For any $v\in  BV([0,T],K)$ with distributional derivative $u$. 
Let
$t\mapsto 
Y^{v}_{t}$ be the contraction propagator 
with initial time $s=0$ associated with ${\mathcal A}_v(t):=e^{-v(t)B} A
e^{v(t)B}$. We define the \emph{generalized propagator} associated 
with $A+u(t)B$ with initial time zero, to be 
$\Upsilon^{\mathfrak{d}v}_{t,0}=e^{v(t)B} Y^{v}_{t}$ for every $t$ in $[0,T]$ and  $v$ in 
$BV([0,T],K)$ such that $v'=u$ in the distributional sense.
\end{definition}

\begin{proposition}
Let $(A,B,K)$ satisfy Assumption~\ref{ass:Interacting} \nuovo{with $A$ and $B$ skew-adjoint. Let 
$v \in BV([0,T],K)$ continuous}. Then, for every $\psi_0$ in $D(A)$, for every $t$ in $[0,T]$,
\[ 
\Upsilon^{\mathfrak{d}v}_{t,0} \psi_0=e^{v(t) B} \psi_0 + \int_{s=0}^t e^{(v(t)-v(s))B} A \Upsilon^{\mathfrak{d}v}_{s,0} \psi_0\mathrm{d}s.
\]
\end{proposition}
\begin{proof}
Since  $\psi_0$ belongs to $D(A)$, Theorem~\ref{thm:kato} guarantees that $t\mapsto Y^v_t \psi_0$ is a strong solution
of $y'(t)= e^{-v(t)B} A e^{v(t)B} y(t)$. That is, for every $t$,
\[ Y^v_t \psi_0 =\psi_0+\int_{s=0}^t e^{-v(s)B} A e^{v(s)B} Y^v_t \psi_0 \mathrm{d}s,\]
hence,  multiplying by $e^{v(t)B}$, 
 \[
 \Upsilon^{\mathfrak{d}v}_{t,0} \psi_0= e^{v(t) B} \psi_0 + \int_{s=0}^t \!\!\! e^{(v(t)-v(s))B} A e^{v(s)B} Y^v_t \psi_0 \mathrm{d}s =e^{v(t) B} \psi_0 + \int_{s=0}^t e^{(v(t)-v(s))B} A \Upsilon^{\mathfrak{d}v}_{s,0} \psi_0 \mathrm{d}s,
 \]
 which concludes the proof.
\end{proof}

\begin{remark}
Let $u\in \radon{[0,T]}$ and define $v_0(t)=u((0,t])$ the associated right-continuous 
cumulative function and let $v\in BV([0,T],\R)$ be such that $v'=u$. Then $v-v_0$ is in 
$BV([0,T],\R)$ and it is almost everywhere $0$ since it is supported on the, at most countable, set 
where $v$ is not right-continuous. 

The propagator $Y^{u}_{t}$ does not depend on the choice of $v$ being right-continuous 
or not at its discontinuities. Indeed the set of point of discontinuity is negligible and a Duhamel formula
 provides the equality of the propagators. On the other hand, the factor $e^{vB}$ depends 
 crucially on this choice.
This explains the notation $\Upsilon^{\mathfrak{d}v}$ instead of $\Upsilon^{u}$.

The reason for introducing the notion of generalized propagator is that imposing any extra requirement on the choice of $v$ will 
lead to loose the compactness provided by Helly's Selection Theorem. This will, for instance, make the presentation 
of the principle exposed in Section~\ref{SEC_principe_obstruction} more complicate.

Notice that for any $v_1$ and $v_2$ in $BV([0,T],K)$ with the same distributional derivative one has
$$
\Upsilon^{\mathfrak{d}v_1} = e^{(v_1-v_2)B} \Upsilon^{\mathfrak{d}v_2}.
$$
\end{remark}

\begin{lemma}\label{lem:propagatorBV}
Let $(A,B,K_1)$ satisfy Assumption~\ref{ass:ass} 
and $(A,B,K_2)$ satisfy Assumption~\ref{ass:Interacting}.
\nuovo{Let $u:[0,T]\mapsto K_1$ be of bounded total variation and
$U(t):= \int_0^t u(s)\mathrm{d}s \in K_2$ 
for all $t$ in $[0,T]$.}

Let $\pro^u_{t}$ be the propagator associated with $A+u(t)B$
with initial time $s=0$. 
Let $t\mapsto Y^{u}_{t}$ be the contraction propagator  associated with ${\mathcal A}(t):=e^{-U(t)B} A
e^{U(t)B}$ 
with initial time $s=0$.

Then 
$$
\pro^u_{t} = e^{U(t)B}Y^{u(t)} \left( = \pro^{\mathfrak{d}U(t)}_t  \right),
$$
for every $t\in [0,T]$.
\end{lemma}

\begin{proof}
Let $\psi_0 \in D(A)$ and define the continuous function
$\Psi :t \mapsto e^{-\int_0^t u(s)\mathrm{d}s B} \pro^{u}_{t}(\psi_0)$. 
By Theorem~\ref{thm:kato}, $\Psi(t)\in D(A)$ is strongly right
differentiable in $t$
 with 
\nuovo{right} derivative 
 $$
- u(t+0)B\Psi(t) +  
e^{-U(t) B} (A+u(t+0)B) \pro^{u}_{t}(\psi_0) =
e^{-U(t) B} A e^{U(t) B}
\Psi(t).
$$
By uniqueness, see Theorem~\ref{thm:kato}, $\Psi(t) = Y^{u}_{t}\psi_0$ for every $t \in [0,T]$.
\end{proof}

\begin{proposition}\label{prop:extensionBVRadon}
Let $(A,B,[0,+\infty))$ satisfy Assumption~\ref{ass:ass} and  $(A,B,K)$ satisfy Assumption~\ref{ass:Interacting}.
Then for every $\psi_{0} \in \Hc$ and $t\in [0,T]$ the map $\pro_{t}(\psi_{0}):u\mapsto \pro^u_{t}(\psi_{0})
\in \Hc$
admits a unique continuous extension on $\{u\in\radon{[0,T]} \mid u([0,T]) \in K, u $positive$\}$ denoted by $\pro_{t}(\psi_{0})$ 
and satisfying
\begin{equation}\label{Eq:InteractionFramework}
\pro^{u}_{t} (\psi_{0}) = e^{u((0,t]) B} Y^{u}_{t}(\psi_{0}), \quad \mbox{ for every } t\in[0,T].
\end{equation}
\end{proposition}
\begin{proof}
For every $u \in \radon{[0,T]}$ positive with $u([0,T]) \in K$ let $(u_{n})_{n\in \N}$ be a sequence of (right-continuous) positive
piecewise
constant functions on $[0,T]$ such that $\int_0^T u_n \in K$  converging to $u$ and which existence is given by Lemma~\ref{LEM_approximation_measure}.

From Remark~\ref{Rem:CommAssumption} and Corollary~\ref{cor:continuity}, for 
every $\psi_{0} \in \Hc$, 
$Y^{u_{n}}_{t}(\psi_{0}) $ tends to $Y^{u}_{t}(\psi_{0}) $ as $n$  tends to
$\infty$.
We set $\pro^{u}_{t} (\psi_{0}) = e^{u((0,t]) B} Y^{u}_{t}(\psi_{0})$. Then
$\pro^{u_{n}}_{t} (\psi_{0})$ tends to $\pro^{u}_{t} (\psi_{0})$ as $n$ tends to
$\infty$. 
The uniqueness of the extension is guaranteed by Lemma~\ref{lem:propagatorBV}.
\end{proof}

\begin{remark}
 With respect to Definition~\ref{def:propagator-radon}, Proposition~\ref{prop:extensionBVRadon} fixes the choice of antiderivative of $u$ to the right-continuous one, accordingly to the arbitrary choice for the notion of convergence for Radon measures presented in Appendix~\ref{sec:notations}. Different choices would have led to different  choices for the antiderivative of $u$. Qualitatively speaking, any choice provides the same results in the sequel. 
\end{remark}

\begin{remark}
 The definition of propagator associated with positive Radon measures given in~\eqref{Eq:InteractionFramework} can be extended to signed Radon measures provided that $(A,B,\R)$ satisfies Assumption~\ref{ass:ass}. Notice, however, that if $(A,B,\R)$ satisfies Assumption~\ref{ass:ass} then $B$ is necessarily symmetric.
 
In the case in which $B$ is not symmetric the definition of
  propagator can be extended to signed Radon measures provided that $(A,B,K)$ satifies Assumption~\ref{ass:Interacting}.
The uniqueness of the continuous extension can be obtained if   $(A,B-c,[0,\infty))$ and $(A,-B-c',[0,\infty))$ satisfy Assumption~\ref{ass:ass}.
Indeed consider $u \in BV([0,T],\R)$
  and split $u$ in the difference of positive part $u^+:=\max\{u,0\}$ and negative part 
  $u^-:=\max\{-u,0\}$.  Then $A(t) = A+u^+(t)(B-c) + u^-(t)(-B-c')$ satisfies Assumption~\ref{ass:kato-generale}.
\end{remark}

\begin{proposition}\label{prop:propagator-radon2}
Let $(A,B)$ satisfy
Assumption~\ref{ass:StrongInteraction} and $D(A)\subset D(B)$.  Then for every 
$\psi_0$ in $D(A)$, for every $u \in L^1([0,T],\R)$, 
the map $t \mapsto
\pro^{u}_{t}(\psi_{0})$ 
satisfies
\begin{equation}\label{eq:weaksolutions}
\int_{[0,T]} \langle f'(t),  \pro^{u}_{t}(\psi_{0}) \rangle \mathrm{d}t =
\int_{[0,T]}\langle 
f(t),  A \pro^{u}_{t}(\psi_{0}) \rangle \mathrm{d}t+  \int_{[0,T]}\langle  f(t), B 
\pro^{u}_{t}(\psi_{0})   \rangle u(t)\mathrm{d}t\,,
\end{equation}
for every 
$f \in C_0^{1}([0,T],\Hc)$.
\end{proposition}
A mapping $t \mapsto
\pro^{u}_{t}(\psi_{0})$ satisfying~\eqref{eq:weaksolutions} is called \emph{weak solution} of~\eqref{eq:main} with initial condition $\psi_0$.
\begin{proof}
For every $u \in L^1([0,T])
$ let $(u_{n})_{n\in \N}$ be a sequence of
piecewise
constant functions on $[0,T]$ that converges to $u$ 
 in
$\radon{[0,T]}$ 
(in the sense of Appendix~\ref{sec:notations}).

For every $f \in C_{0}^{1}([0,T], \Hc)$,
$$
- \int_{[0,T]} \langle f'(t),  \pro^{u_{n}}_{t}(\psi_{0}) \rangle \mathrm{d}t =
\int_{[0,T]}\langle  f(t), A  \pro^{u_{n}}_{t}(\psi_{0}) \rangle \mathrm{d}t+ 
\int_{[0,T]}\langle  f(t), B \pro^{u_{n}}_{t}(\psi_{0})   \rangle u_{n}(t)\mathrm{d}t\,
$$
since from Theorem~\ref{thm:kato}, $Y^{u_{n}}_{t}(\psi_{0})\in D(A)$ for any 
$t\in[0,T]$.

It is then sufficient to prove the following convergences
\begin{equation}\label{eq:0001}
 \lim_{n\to \infty} \int_{[0,T]} \langle f'(t),  \pro^{u_{n}}_{t}(\psi_{0}) \rangle \mathrm{d}t = \int_{[0,T]} \langle f'(t), 
\pro^{u}_{t}(\psi_{0}) \rangle \mathrm{d}t,
\end{equation}
\begin{equation}\label{eq:0002}
\lim_{n\to \infty} \int_{[0,T]}\langle  f(t),  A \pro^{u_{n}}_{t}(\psi_{0}) \rangle \mathrm{d}t
= \int_{[0,T]}\langle  f(t), A  \pro^{u}_{t}(\psi_{0}) \rangle \mathrm{d}t,
\end{equation}
and
\begin{equation}\label{eq:0003}
\lim_{n\to \infty} \int_{[0,T]}\langle  f(t),  B \pro^{u_{n}}_{t}(\psi_{0})   \rangle u_n(t)\mathrm{d}t =
\int_{[0,T]}\langle f(t),  B \pro^{u}_{t}(\psi_{0})   \rangle u(t)\mathrm{d}t.
\end{equation}
\nuovo{
Convergence~\eqref{eq:0001} is a consequence of Lebesgue Dominated Convergence Theorem being the integrand uniformly bounded.

We can rewrite~\eqref{eq:0002} as
$$
\lim_{n\to \infty} \int_{[0,T]}\langle  f(t),  A \left(\pro^{u_{n}}_{t}(\psi_{0})-\pro^{u}_{t}(\psi_{0}) \right) \rangle \mathrm{d}t
= 0
$$ 
Recall that the adjoint of $A$ is also maximal dissipative as soon as $A$ is maximal dissipative~\cite[Chapter 3.1]{TucsnakWeiss}.
We can then restrict
to $f \in C_{0}^{1}([0,T], D(A^\ast))$
by replacing $f$ with $\lambda(\lambda-A^\ast)^{-1}f$,
where $\lambda$ is a large positive since 
$$
t\in[0,T]\mapsto A \left(\pro^{u_{n}}_{t}(\psi_{0})-\pro^{u}_{t}(\psi_{0}) \right) \in\Hc,
$$ 
is uniformly bounded, and conclude, as in~\eqref{eq:0001}, with Lebesgue Dominated Convergence Theorem.
}

The  last convergence~\eqref{eq:0003} reads
\begin{align}\label{eq:DelicateTerm}
\lefteqn{ \int_{[0,T]}\langle  f(t),  B \pro^{u_{n}}_{t}(\psi_{0})   \rangle u_n(t)\mathrm{d}t-
\int_{[0,T]}\langle f(t),  B \pro^{u}_{t}(\psi_{0})   \rangle u(t)\mathrm{d}t}\nonumber\\
&= \int_{[0,T]}\langle  f(t),  B \pro^{u_{n}}_{t}(\psi_{0})   \rangle (u_n-u)(t)\mathrm{d}t+
\int_{[0,T]}\left(\langle  f(t),  B \pro^{u_{n}}_{t}(\psi_{0})   \rangle -\langle f(t),  B \pro^{u}_{t}(\psi_{0})   \rangle \right)u(t)\mathrm{d}t.
\end{align}
In order to prove the convergence for the second term of the right-hand side we have
$$
\lim_{n\to \infty} \int_{[0,T]}\langle  \left(B(1-A)^{-1}\right)^{\ast}f(t),   (1-A)\left(\pro^{u_{n}}_{t}(\psi_{0})-\pro^{u}_{t}(\psi_{0}) \right) \rangle \mathrm{d}t
= 0.
$$
Indeed $B(1-A)^{-1}$ is bounded, so is its adjoint. The proof is then similar to~\eqref{eq:0001} and~\eqref{eq:0002}.

Finally,
from Theorem~\ref{thm:kato} and estimates~\eqref{Eq:EstimateOnM}, \eqref{eq:EwpGrowthB}, and \eqref{eq:comm}, there exists $C>0$ and $\omega>0$ depending on $A$ and $B$ only, such that
\begin{align*}
\lefteqn{\left |\int_{[0,T]}\langle f(t), 
B \pro^{u_n}_{t}(\psi_{0})\rangle (u-u_{n})(t) \mathrm{d}t \right |}\\
&\leq
C\sup_{t\in [0,T]} \|f(t)\|\|B\|_A (1+CL_{[0,\|u\|_1]}(A,B)\|u\|_1)e^{2\omega\|u\|_1}\times\\
&\times e^{(1+CL_{[0,\|u\|_1]}(A,B)\|u\|_1)e^{2\omega\|u\|_1}CL_{[0,\|u\|_1]}(A,B)\|u\|_1}\|\psi_{0}\|_{D(A)}
\|u-u_{n}\|_1\\
&\to_{n \to \infty} 0
\end{align*}
since $f(0)=0$ and
using Lemma~\ref{LEM_approximation_measure} and $|u-u_{n}|=|u^+-u_{n}^+|+|u^--u_{n}^-|$ the sequence $(u_n)$ converges to $u$ in $L^1$-norm.
\end{proof}

\begin{remark}\label{Rem:TwoCouplingSkewAdjointness}
An interesting question would be to understand the relation between the 
 assumptions associated with the two constructions 
of propagators considered in this section. For example, on what extent 
does Assumption~\ref{ass:StrongInteraction} ensure that $A+uB$ has a maximal 
dissipative closure for $u\in \R$?

This seems to be a hard question. However in the skew-adjoint case, the 
following considerations are in place. Let $A$ and $B$ be skew-adjoint 
with $D(A) \subset D(B)$. For any $\varphi_1\in \Hc$, any $\varphi_2\in D(A)$ the map
\[
t\in K \mapsto \langle (1-\varepsilon A)^{-1}\varphi_1, e^{tB}Ae^{-tB}(1-\varepsilon A)^{-1}\varphi_2\rangle,
\]
is Lipschitz, its distributional derivative is bounded uniformly in $\varepsilon$ by
the Banach-Steinhaus theorem. So that $[A,B]\in L(D(A)\cap D(B),(D(A)\cap D(B))^\ast)$ extends to an operator such that
\[
[A,B]\in L(D(A),\Hc),
\]
(with a slight abuse of notation we denote by $[A,B]$ the extension of $[A,B]$ to $L(D(A),\Hc)$) and, similarly (using the same abuse of notation), for any $u\in \R$,
\[
[A,A+uB]\in L(D(A),\Hc).
\]
The Nelson commutator theorem, see \cite[Section X.5]{reed-simon-2}, gives that $A+uB$ is 
essentially skew-adjoint for any $u\in \R$.
\end{remark}

\begin{remark}\label{Rem:RightContinuity}
Considering Definition~\ref{Definition:Propagator}, $X(t,s)=e^{v(t)B} Y^{u}_{t,s}e^{-v(s)B}$ 
defines a propagator when $v$ is continuous, that is when $u$ has no atoms. Otherwise, we no 
longer require any continuity keeping in mind that $v_0$ the right-continuous cumulative 
function of $u$ will lead to a right-continuous propagator which is compatible with 
the requirements on the initial conditions. 

From Proposition~\ref{prop:propagator-radon2}, when $v$ is absolutely continuous, $X(t,s)=e^{v(t)B} Y^{u}_{t,s}e^{-v(s)B}$ defines a weak solution of \eqref{eq:main}. The question whether or not it is possible to extend this proposition to Radon measures is then natural. If one considers, as in Section~\ref{sec:nsrf} below, $A=0$, $B$ bounded, and $u = \delta_0$, then the solution of \eqref{eq:main} is $1+H(t)B$, where $H$ is a Heaviside function jumping at $0$. Which is different from $e^{H(t)B}$ provided by our analysis. 

Proposition~\ref{prop:propagator-radon2} can be extended to measures with singular continuous part. Indeed any Radon measure is in the sequential closure of the set of absolute continuous measures for the convergence of sequences we consider (see Appendix~\ref{sec:notations}). Notice that in Lemma~\ref{LEM_approximation_measure} the sequence is also narrow convergent. Since the propagators associated with absolute continuous and singular continuous measure are bounded continuous, the  first term in~\eqref{eq:DelicateTerm} will tend to $0$.
\end{remark}

\paragraph{Nonexistence of bounded solution propagators for unbounded control potentials in the skew-adjoint case.}
Let us consider the possible extensions of Proposition~\ref{prop:propagator-radon2} to the case of a pair of skew-adjoint operators $(A,B)$. 
We exhibit here an example of system~\eqref{eq:main} with a Radon measure control for which it is not possible to construct a strong solution by applying a bounded propagator to the initial condition (even if this is in the domain of the generator).

Let $\psi_0 \in D(A)$ and $\psi_1\in D(B)$ with $B\psi_1\in D(A)$ then for any solution of \eqref{eq:main} for $u=\delta_{T/2}$ with initial condition $\psi_0$ at $t=0$ the jump at $T/2$ is exactly $B\psi(T/2)$ (after integration of \eqref{eq:main} around $T/2$). So, setting $\psi(T/2)=\psi_1$, we have
$$
 \pro^{u}_{t}(\psi_{0})=\begin{cases}
  e^{tA}\psi_0 &\mbox{ for }  t\in[0,\frac{T}{2}),\\
 \psi_1 &\mbox{ for }  t=\frac{T}{2},\\
  e^{tA}\psi_0+e^{(t-\frac{T}{2})A}B\psi_1 &\mbox{ for } t\in(\frac{T}{2},T].
 \end{cases}
$$

The determination of $\psi_1$ leads to a uniqueness issue and a modelling interpretation. A way to overcome this issue is to impose a continuity at $t=T/2$. Note that the left-continuity leads to 
\[
 \psi_1 = e^{\frac{T}{2}A}\psi_0+B\psi_1.
\]
So if $1$ is in the spectrum of $B$ then for some $\psi_0$ this is not solvable. Note that with $u=\alpha \delta_{T/2}$ the problem is the same for every $\alpha$ in the spectrum of $B$. This excludes the possibility to construct a left-continuous propagator. 

A natural requirement seems to be the right-continuity and thus $\psi_1=e^{\frac{T}{2}A}\psi_0$. Then, when $B$ is unbounded, the issue is to extend continuously the propagator from $D(A)$ to $\Hc$ as for some $\psi_0\in \Hc$, one may have that $e^{\frac{T}{2}A}\psi_0 \not\in D(B)$. 

By convexity, a linear combination of left and right continuity will lead to the same kind of contradictions.  

In conclusion, when $B$ is unbounded one cannot expect to construct bounded solutions with Radon controls in the skew-adjoint case.  In Section~\ref{Sec:BoundedPotential}, we prove that when $B$ is bounded, there exists a strongly regulated propagator defining a weak solution of \eqref{eq:main}. This propagator is not necessarily a contraction.  

Similar questions arise for ODEs. We refer for instance to \cite{Pandit_Deo}. 

Nonetheless, with absolutely continuous Radon measures, we have built proper propagators and the extension to Radon measures presented in this work has consequences in the analysis of the attainable sets as presented in the sequel.

\subsection{The attainable set}
\label{SEC_attainable_set_interaction}

The key result in the proof of Theorem~\ref{THE_main_introduction} is given by the following proposition.
\begin{proposition}\label{rem:compactness}
Let $T>0$. Let
$\psi_{0} \in \Hc$. Let
$(A,B)$ satisfy
Assumption~\ref{ass:StrongInteraction}. 
Then,
for every $L>0$,   the set 
$$
\{\pro^{u}_{t}(\psi_0)\,:\, u\in
\radon{[0,T]}, |u|({ (0,T]}) \leq L, {t\in
{ [0,T]}}\},
$$ 
is relatively compact in 
{$\Hc$}.
\end{proposition}
\begin{proof}
Consider a sequence $(u_{n})_{n\in \N} \subset \radon{[0,T]}$ such
that $ |u_{n}|((0,T])
\leq  L$ for every $n$.
By Helly's Selection Theorem the sequence $v_{n}: t \mapsto u_{n}((0,t])$,
has a subsequence pointwise converging to some $v \in BV([0,T])$,
$\|v\|\leq L$. \nuovo{We relabel this convergent subsequence by $(v_{n})_{n\in \N}$.}

From \eqref{eq:comm} we have that
$\mathcal{A}_{n}(t) = e^{-u_{n}((0,t])B}Ae^{u_{n}((0,t])B}$ 
is uniformly bounded
in $BV([0,T], L(\Dc,\Hc))$. By Corollary~\ref{cor:continuity}, 
$t \mapsto Y_{t}^{v_{n}}(\psi_0)$ converges uniformly on $[0,T]$ to $t
\mapsto Y_{t}^{v}(\psi_0)$ as $n \to \infty$.
For any sequence $(t_n)_n$, 
$(v_n((0,t_n]))_n$ is a bounded sequence. \nuovo{In particular, it has a strongly convergent subsequence and so is $(e^{v_n((0,t_n])B})_n$.}
\end{proof}

\begin{remark}\label{Rem:Compactness}
 Note that the set $\{\pro^{u}(\psi_0)\,:\, u\in
L^1( [0,T], \R) , \|u\|_{L^1} \leq L \}$ is relatively compact
in 
{$L^\infty({ [0,T]},\Hc)$}.
 However, despite the compactness of $[0,T]$, the set $\{\pro^{u}(\psi_0)\,:\,
u\in
\radon{[0,T]}, |u|((0,T]) \leq L \}$ may be not relatively compact
in {$L^\infty([0,T],\Hc)$}. Indeed, if this set were relatively compact, then
 the \emph{generalized propagator} associated with $A+u(t)B$ 
would be strongly  continuous, due to the pointwise density of 
solutions of \eqref{eq:main} which are continuous. This is not the case in general due to the factor $e^{u((0,t])B}$ in \eqref{Eq:InteractionFramework}.
\end{remark}

From the above result the attainable set is a countable union of
totally bounded sets.
\begin{corollary}\label{Cor:NoExactControllability}
Let $\psi_{0} \in \Hc$.  If $(A,B)$ satisfies
Assumption~\ref{ass:StrongInteraction} then
$$
{\mathcal Att}_{\mathcal R}(\psi_0):=\left\{ \pro^{u}_{t}(\psi_{0}), u\in
\radon{[0,+\infty)}, 
t\geq 0\right\}
$$
is contained in a countable union of
compact subsets of $\Hc$.
\end{corollary}
\begin{proof}
The attainable set can be rewritten as
$$
\bigcup_{L,T> 0}  \left\{ \pro^{u}_{t}(\psi_{0}), u\in
\radon{[0,T]}, 
|u|((0,T]) \leq L, t\in [0,T]\right\}
$$
and this union can be, in fact, restricted to $L,T$ in a countable set, for instance $\N^2$.
Then Proposition~\ref{rem:compactness} tells that  each set of the union is
relatively compact in $\Hc$ and thus with empty
interior.
\end{proof}

We are now ready to prove Theorem~\ref{THE_main_introduction}.
\begin{proof}[Proof of Theorem \ref{THE_main_introduction}] 
The well-posedness result for $L^1$ 
controls is a consequence of Proposition \ref{prop:propagator-radon2} proved for Radon 
controls. The conclusion on the attainable set for $L^1$ controls is a consequence of Corollary 
\ref{Cor:NoExactControllability} proved for Radon controls.
\end{proof}

\section{Higher order norm estimates for {mildly coupled} systems}\label{Sec:HigherOrderFEPS}

In the following we will restrict our analysis to the skew-adjoint case. 
The motivation for this assumption is twofold. On the one hand this is the case for most of the mathematical objects  appearing in quantum mechanics and, on the other hand, the restriction to 
skew-adjoint operators makes the analysis simpler. 

The aim of this section is to analyze under which conditions the solution built in the 
previous sections are smooth in the scale of $A$. This is indeed the rationale for stating assumptions~\ref{ass:ass}, 
\ref{ass:Interacting}, or \ref{ass:StrongInteraction} in $D(|A|^{k/2})$
instead of $\Hc$. Our aim is to provide a somewhat simpler 
criteria showing that the extension of assumptions on $B$ will be sufficient.  To this aim, the $A$-boundedness of $B$ as an operator acting on $D(|A|^{k/2})$ is crucial 
and it is stated in Lemma~\ref{lem:KatoRellichBoundWeakCoupling} below which is the cornerstone of 
the analysis of this section. This is especially important if we want to obtain the 
regularity of propagators in the scale of $A$ up to the order $k/2$. For lower orders, a 
simple interpolation argument provides the desired results.
The criteria will be used in a perturbative framework (Kato-Rellich type argument) and we 
will not consider 
the entire set $K$
for the values of $u$, unless we assume that the domain of powers of $A+uB$ 
are the same for any $u\in K$. We recall that in the dissipative framework in order to use
Kato--Rellich criterion $u$ has to be nonnegative when $B$ is dissipative, below we assume 
that both $B$ and $-B$ have dissipativity properties (up to a shift by a constant as in Assumption~\ref{ASS_B_skew_symmetric} or~\ref{ASS_B_skew_symmetric2}) so 
that the sign of $u$ does not play any role.

This shows that for time reversible systems,
the input-output mapping does not change the regularity with
respect to $A$ in the spirit of Section~\ref{Sec:InvarianceDomain}. Since eigenvectors belong to any $D(|A|^k)$ this shows that exact
controllability clearly relies on the regularity of $B$ in the scale of $A$.

\subsection{The mild coupling}

Given a skew-adjoint operator $A$  and  $k\in \R$, $k \geq 0$, we define
$$
\|\psi\|_{k/2} = \sqrt{ \langle  | A|^{k} \psi, \psi \rangle}.
$$

\begin{definition}[Mild coupling]\label{Def:WeakCoupling}
Let $k$ be a nonnegative real. A pair of \emph{skew-adjoint}
operators $(A,B)$
is \emph{$k$-{mildly coupled} } if 
\begin{itemize}
 \item[$(i)$] $A$ is invertible with bounded 
inverse from $D(A)$ to $\Hc$\label{Assumption:FEPSOnA},
 \item[$(ii)$] for any real $t$, $e^{tB}D(|A|^{k/2})\subset D(|A|^{k/2})$,
 \item[$(iii)$] there exists $c\geq 0$ and $c'\geq 0$ such that $B-c$ 
and $-B-c'$ generate contraction semigroups on $D(|A|^{k/2})$ for the norm
$\|\cdot\|_{k/2}$.
 \end{itemize}
\end{definition}
The \emph{optimal exponential growth}, \nuovo{the growth bound of the semigroups generated by $\pm B$}, is defined by
\begin{equation}\label{eq:ck}
 c_k(A,B):=\sup_{t\in \R}\frac{\log\|e^{tB}\|_{L(D(|A|^{k/2}),
D(|A|^{k/2})}}{|t|}.
\end{equation}

\begin{remark}
As in Section~\ref{Sec:BoundedPotential} below, if 
\begin{equation*}
t \mapsto  e^{tB}  
\end{equation*}
is a strongly continuous semigroup on $D(|A|^{k/2})$ then there
exists
$\omega>0$ and $C>0$ such that
\[
 \|e^{tB}\|_{L(D(|A|^{k/2}),D(|A|^{k/2}))}\leq C e^{\omega t}, \quad \mbox{ for all }
t>0.
\]


The invertibility of $A$ is needed to ensure that $\|\cdot\|_{k/2}$ is a norm 
equivalent to the graph norm of $D(|A|^{k/2})$. The use of the associated norm 
is due to the interpolation criterion used in Lemma~\ref{Lem:kfeps-sfeps} below. 
\end{remark}

\begin{remark}
The quantity $c_k(A,B)$ is related to the growth abscissa of $B$ in $D(|A|^{k/2})$. The link between the growth abscissa and the spectral radius of a semigroup on a Hilbert space is considered in~\cite[Section 3]{Pruss}.
\end{remark}

\begin{remark}
For many systems encountered in the physics literature, the operator $A$ is skew-adjoint 
with a spectral gap. Hence the invertibility of $A$ can be obtained by 
replacing $A$ by $A-\lambda \mathrm{i}$ for a suitable $\lambda$ in $\R$. Notice that this 
translation on $A$ only induces a global phase shift on the propagator that is physically 
irrelevant (i.e., undetectable  by observations).
\end{remark}

The following proposition gives another characterization of mild coupling
using Hille--Yosida Theorem.
\begin{proposition}\label{Prop:ThmHilleYosida}
Let $k$ be a nonnegative real. A pair of skew-adjoint operators
$(A,B)$ { with $A$ invertible}
is \emph{$k$\-{-}{mildly coupled} } if
and only if $B$ is closed in $D(|A|^{k/2})$,
and there exists $\omega$ such that 
\begin{equation}\label{eq:1223}
\|(\lambda
I- B)^{-1}\|_{L(D(|A|^{k/2}),D(|A|^{k/2})}\leq\frac{1}{|\lambda|-\omega},
\end{equation}
for every real $\lambda$, $|\lambda|>\omega$
in the resolvent set of $B$.

Moreover the smallest $\omega$ satisfying~\eqref{eq:1223} is  
$c_k(A,B)$ given by~\eqref{eq:ck}.
\end{proposition}
\begin{proof}
If $(A,B)$
be \emph{$k$-{mildly coupled} } then $B-c_k(A,B)$ is the generator of a contraction
semigroup in $D(|A|^{k/2})$. From Hille--Yosida Theorem, we deduce the  
equivalence with Definition~\ref{Def:WeakCoupling}.
\end{proof}

The following proposition gives an equivalent definition which may be easier to 
check in practice.
\begin{proposition}\label{Prop:ThmLumerPhillips}
Let $k$ be a nonnegative real. A pair of skew-adjoint operators $(A,B)$ 
with $A$ invertible is $k$-mildly coupled, 
if and only if for some $\omega>0$, 
\[
 (\omega\pm B)^{-1}D(|A|^{k/2})\subset D(|A|^{k/2})
\]
and for any $\psi \in (\omega- B)^{-1}D(|A|^{k/2})=(\omega+ B)^{-1}D(|A|^{k/2})$, one has
\begin{equation}\label{eq:2121}
| \Re \langle |A|^k\psi,B\psi\rangle | \leq \omega
\|\psi\|_{D(|A|^{k/2})}^2.
\end{equation}
Moreover the smallest $\omega$ 
satisfying~\eqref{eq:2121} is  
$c_k(A,B)$ given by~\eqref{eq:ck}.
\end{proposition}
\begin{proof}
We first notice that, for any $\omega$ in the resolvent sets of $B$ and $-B$,
\[
 (\omega\pm B)^{-1}D(|A|^{k/2})\subset D(|A|^{k/2})
\]
implies $(\omega- B)^{-1}D(|A|^{k/2})=(\omega+ B)^{-1}D(|A|^{k/2})$. Indeed, from the resolvent identity, we deduce
\[
 (\omega- B)^{-1}D(|A|^{k/2})
\subset (\omega+ B)^{-1}D(|A|^{k/2}).
\]

Assume that $(A,B)$ is $k$-{mildly coupled}, 
then since $B-c_k(A,B)$ and $-B-c_k(A,B)$ are generator of contraction semigroups
on $D(|A|^{k/2})$, 
they are closed and maximal dissipative on $D(|A|^{k/2})$, 
their respective resolvent sets contains positive half lines (by means of 
Hille--Yosida theorem) and their domain, by definition of resolvent, is
$(\omega\pm B)^{-1}D(|A|^{k/2})$ for any $\omega>c_k(A,B)$. Since they are maximal dissipative, 
we have that
\[
|\Re \langle |A|^k\psi,B\psi\rangle|  \leq c_k(A,B) \|\psi\|_{D(|A|^{k/2})}^2.
\]
for any $\psi \in (\omega- B)^{-1}D(|A|^{k/2})=(\omega+ B)^{-1}D(|A|^{k/2})$.

Reciprocally, $ B+\omega$ and $-B+\omega$ are closed as operators on 
$\Hc$ and so they are closed on $D(|A|^{k/2})$. Since $ B+\omega$ and $-B+\omega$ are 
dissipative on $D(|A|^{k/2})$, they are generators of a 
contractions semigroups if they are surjective. 
So they are since 
$(\pm B +\omega)^{-1}f \in D(|A|^{k/2})$ for any $f\in D(|A|^{k/2})$.
\end{proof}

The notion of mild coupling is related to the notion of ``weak coupling'' 
introduced in~\cite{BoussaidCaponigroChambrion}. The relation  between 
these two definitions is given by the following 
lemma.
\begin{lemma}\label{LEM_lien_WC_nouvelle_definition_ancienne_definition}
Let $(A,B)$ be a pair of linear operators such that $A$ is invertible and skew-adjoint with domain $D(A)$, $B$ is skew-symmetric with $D(A)\subset D(B)$,  $A+uB$ (seen as an operator acting on $\Hc$) is essentially skew-adjoint on $D(A)$ for every $u$ in $\R$, $D(|A+uB|^{k/2})=D(|A|^{k/2})$  for some $k\geq 1$ and for any real $u$, 
 and
 there exists
a constant $C$ such that for every $\psi$ in $D(|A|^k)$, 
\[
|\Re \langle |A|^k
\psi,B\psi \rangle |\leq C |\langle |A|^k \psi, \psi \rangle|. 
\]
Then $(A,B)$ is $k$-{mildly coupled}  and $c_k(A,B)$ is the best possible constant $C$ in the above inequality.
\end{lemma}
\begin{proof}
The assumption that there exists $k\geq 1$ and 
a constant $C$ such that for every $\psi$ in $D(|A|^k)$, 
\[
|\Re \langle |A|^k
\psi,B\psi \rangle |\leq C |\langle |A|^k \psi, \psi \rangle|
\]
and the Nelson Commutator Theorem, see \cite[Section X.5]{reed-simon-2}, imply that 
$B$  is essentially skew-adjoint  on the domain $D(|A|^{k/2})$. Therefore  $B$ is essentially 
skew-adjoint on $D(A)$. Then Trotter Product Formula, see~\cite[Theorem VIII.31]{reed-simon-1}, implies that
\[
 \left(e^{\frac{t}{n}(A+uB)}e^{-\frac{t}{n}A}\right)^n \to e^{t u B}
\]
in the strong sense as $n$ goes to infinity. Since  each of 
the term of the above  sequence is bounded on $D(|A|^{k/2})$ with a bound  
$e^{C|t||u|}$, see~\cite[Proposition 2]{BoussaidCaponigroChambrion}, we conclude 
that $e^{tB}$ is bounded on $D(|A|^{k/2})$ with the same bound $e^{C|t||u|}$. 
Then $(A,B)$ is $k$-mildly coupled.
\end{proof}
\begin{remark}
In general, $(A,B)$ can be $k$-{mildly coupled} without being weakly coupled (in the sense of~\cite[Definition~1]{BoussaidCaponigroChambrion}) or without satisfying the assumption of Lemma~\ref{LEM_lien_WC_nouvelle_definition_ancienne_definition}. Indeed for any invertible skew-adjoint unbounded operator $(A,\rmi A^2)$ is $2$-{mildly coupled} and $D(A)\not\subset D(\rmi A^2)$ or $D(A+\rmi A^2)=D(A^2)\neq D(A)$.
\end{remark}

Let us state state an interpolation result.
\begin{lemma}\label{Lem:kfeps-sfeps}
Let $k$ be a positive real. If $(A,B)$
is $k$-{mildly coupled}  then $(A,B)$
is $s$-{mildly coupled}  for any $s\in [0,k]$
and
\[
 c_s(A,B)\leq \frac{s}{k}c_k(A,B).
\]
\end{lemma}
\begin{proof}
We will consider $s\in(0,k)$. Indeed, for $s=k$ this is obvious and $s=0$ there is nothing to prove since $B$ is skew-adjoint by assumption.

Moreover since $B$ is skew-adjoint, for every $\psi$ in $D(|A|^\frac{k}{2})$,
\[
 \|e^{tB}\psi \|_{D(|A|^{k/2})}=\||A|^{k/2}e^{tB}\psi\|=\||{\cal A}(t)|^{k/2}\psi \|.
\]
where ${\cal A}(t)=e^{-tB}Ae^{tB}$ (which is skew-adjoint with domain $D(A)$).

Since $(A,B)$ is $k$-{mildly coupled} we deduce
\[
\frac1{\|A^{-1}\|^k}\leq |{\cal A}(t)|^k\leq e^{2c |t|} |A|^k.
\]
which from Proposition~\ref{Prop:Interpolation} in Appendix~\ref{sec:interpolation} yields
\[
 |{\cal A}(t)|^s\leq e^{2cs|t|/k} |A|^s.
\]
This concludes the proof.
\end{proof}

A corollary of this interpolation result is the following 
result which is crucial in our analysis.
It shows 
that  if $(A,B)$ is $k$-{mildly coupled} the $A$-boundedness of $B$
extends naturally to $D(|A|^{k/2})$. Hence, from now on, we will work in $D(|A|^{k/2})$, that is we  consider $\Hc=D(|A|^{k/2})$.
\begin{lemma}\label{lem:KatoRellichBoundWeakCoupling}
 Let $k$ be a nonnegative real. Let $(A,B)$ be $k$-{mildly coupled} and such that $B$ is $A$-bounded.
Then
\[
\inf_{\lambda >0}
\|B(A-\lambda)^{-1}\|_{L(D(|A|^{\frac{k}{2}}),D(|A|^{\frac{k}{2}}))}
\leq \|B\|_A
\]
\end{lemma}
\begin{proof}\nuovo{Denote by $\Hc_s=D(|A|^s)$, endowed with the graph norm,  and by $\Hc_{-s}=D(|A|^s)^\ast$,  for $s\in [0,+\infty]$. Note that,  if $s'\leq s$, $\Hc_s\subset \Hc_{s'}$ with continuous embeddings. The domain of $|A|^\sigma$, $\sigma\geq 0$, in $\Hc_s$ is $\Hc_{s+\sigma}$ and $|A|^\sigma\Hc_s=\Hc_{s-\sigma}$ as $A$ is invertible. The real interpolation $(\Hc_{s_0},\Hc_{s_1})_{\theta,2}$ is $\Hc_{s_\theta}$ with $s_\theta=\theta s_1 +(1-\theta)s_0$. See, for instance,~\cite[Section 2.8]{amrein}.}

The proof follows \cite[Section X.5]{reed-simon-2}. 
The commutator $[|A|^{k},B]=|A|^{k}B-B|A|^{k}$ is defined from $\Hc_{k+1}$ to $\Hc_{-k}$ and since 
$B$ is $k$-mildly coupled, for $\psi\in \Hc_{k+1}$, we have
\[
\left| \left\langle \psi, (|A|^{k}B-B|A|^{k})\psi\right\rangle \right|=2\left| \Re \left\langle B\psi,|A|^{k}\psi\right\rangle \right|\leq 2 c_k(A,B)\||A|^{k/2} \psi\|^2.
\]
This provides, after polarization, the boundedness of $[|A|^{k},B]$ from $\Hc_{k/2}$ to $\Hc_{-k/2}$.
\nuovo{For any $s$ in $[0,k]$ due to Lemma~\ref{Lem:kfeps-sfeps}, $B$ is $s$-mildly coupled. Hence the commutator $[|A|^{s},B]=|A|^{s}B-B|A|^{s}$ extends as a bounded operator from $\Hc_{s/2}$ to $\Hc_{-s/2}$ for any  $s$ in $[0,k]$.}

Recall $B$ is $A$-bounded, that is $B$ bounded from $\Hc_1$ to $\Hc$. As $-B^\ast$ is an extension of $B$ and is bounded from $\Hc$ to $\Hc_{-1}$, with same norm, by interpolation, $B$ is bounded from $\Hc_{1+s}$ to $\Hc_s$ for any $s\in [-1,0]$. The bound on the norm of $B$ as an operator from $\Hc_{1+s}$ to $\Hc_s$,  $s\in [-1,0]$ is thus smaller than the norm of $B$ as an operator from $\Hc_1$ to $\Hc$.
Hence, we have
\[
\inf_{\lambda >0}
\|B(A-\lambda)^{-1}\|_{L(\Hc_{s},\Hc_{s})}
\leq \|B\|_A.
\]
Therefore, for $k\in [0, 2]$, $B|A|^{k}$ is bounded from $\Hc_{k/2+1}$ to $\Hc_{-k/2}$. 
As $|A|^{k}B = B|A|^{k} + [|A|^{k},B]$, $|A|^{k}B$ extends as a bounded operator from $\Hc_{k/2+1}$ to $\Hc_{-k/2}$ and $B$ is bounded from $\Hc_{k/2+1}$ to $\Hc_{k/2}$. Hence $B$ is bounded from $\Hc_{1+s}$ to $\Hc_s$ for any $s\in [-1,k/2]$. 
As for the norm, for any $\epsilon$ positive, there exists $\Lambda_\epsilon$ such that, for any $\psi \in \Hc_{k/2+1}$, 
\begin{align*}
 \|B \psi \|_{k/2}&=\||A|^{-k/2}|A|^kB\psi\|\\
 &\leq \||A|^{-k/2}B|A|^{k}\psi\| + \||A|^{-k/2}[|A|^{k},B]\psi\|\\
 &\leq (\|B\|_A+\epsilon)\||A|^{k/2}|A+\Lambda_\epsilon|\psi\| + 2 c_k(A,B)\||A|^{k/2}\psi\|\\
  &\leq (\|B\|_A+\epsilon)\|A|^{k/2}|A+L_\epsilon|\psi\|
\end{align*}
for $L_\epsilon\geq \Lambda_\epsilon$ large enough.
Then, we deduce
\[
\inf_{\lambda >0}
\|B(A-\lambda)^{-1}\|_{L(\Hc_{k/2},\Hc_{k/2})}
\leq \|B\|_A.
\]

Hence the Lemma is proved for $k$ in $[0,2]$. Assume that $k\geq 2$, we now extend the lemma to $k$ by an induction. 
Assume that 
$B$ is bounded from $\Hc_{1+s}$ to $\Hc_s$ for any $s\in [-1,(k-2)/2]$.
So $B$ is bounded from $\Hc_{1+s}$ to $\Hc_{s}$ for any $s\in [-k/2,0]$ by duality.

We have $B|A|^{k}$ is bounded from $\Hc_{k/2+1}$ to $\Hc_{-k/2}$.  
As $|A|^{k}B = B|A|^{k} + [|A|^{k},B]$, $|A|^{k}B$ extends as a bounded operator from $\Hc_{k/2+1}$ to $\Hc_{-k/2}$, 
and $B$ is bounded from $\Hc_{k/2+1}$ to $\Hc_{k/2}$. 
Hence $B$ is bounded from $\Hc_{1+s}$ to $\Hc_s$ for any $s\in [-1,k/2]$. 
This concludes the induction since 
the norm estimate is obtained as in the initialisation.
\end{proof}

The following result shows sufficient conditions to have $D(|A+uB|^{k/2})=D(|A|^{k/2})$ as required in Lemma~\ref{LEM_lien_WC_nouvelle_definition_ancienne_definition}. 
Checking this property in practice may be a hard task in general.
Recall that as $D(A)\subset D(B)$, $A+uB$ is self-adjoint
 with $D(A+uB)=D(A)$ for sufficiently small $u$  by
 Kato--Rellich theorem.
\begin{lemma}\label{lem_domaines_egaux}\label{lem:NormesEquivalentes}
 Let $k$ be a positive real, $(A,B)$ be $k$-mildly coupled, and $u\in \R$ such that
$|u|<1/\|B\|_A$. Then $D(|A|^s)=D(|A+u B|^s)$ for
every $s\in[0, k/2 +1]$. 
\end{lemma}
\begin{proof}
We proceed by induction on $j$ to prove $D(|A|^{k/2-\lfloor k/2\rfloor +j})=D(|A+uB|^{k/2-\lfloor k/2\rfloor +j})$ for $j\leq \lfloor k/2\rfloor +1$.
By Kato--Rellich theorem, $D(A)=D(A+uB)$ for every $u$ in $(-1/\|B\|_A,1/\|B\|_A)$. By interpolation, see Corollary~\ref{Lem:DomainInterpolation} in Appendix~\ref{sec:interpolation}, $D(|A|^s)=D(|A+uB|^s)$ for $0\leq s \leq 1$ and in particular for $s=\frac{k}{2}-\lfloor \frac{k}{2} \rfloor$. This initializes the induction for $j=0$.

Let us assume that 
 $D(A^\ell)=D((A+u B)^\ell)$ for some $\ell \leq \lfloor k/2\rfloor$.
By definition, 
$$D(A^{\ell+1})=\{f \in D(A^\ell)|Af \in D(A^\ell)\},$$ 
and, using the
inductive hypothesis,
$$
D(|A+u B|^{\ell+1})=\{f \!\in\! D(|A+u B|^\ell)||A+u B| f \! \in D(|A+u B|^\ell)\}=
\{f \!\in D(|A|^\ell)|(A+u B)f \in D(|A|^\ell)\}.
$$

So that $D(|A+u B|^{\ell+1})$ is the domain of $A+uB$ as an operator acting on $D(|A|^\ell)$. The domain of $A$
as an operator acting on $D(|A|^\ell)$ is $D(|A|^{\ell+1})$. Since $A$ is skew adjoint on $D(|A|^\ell)$ and $B-c$ (or
$-B-c'$) is dissipative, since $\ell \leq k/2$,  in $D(|A|^\ell)$ due to Proposition~\ref{Prop:ThmLumerPhillips}, using Lemma~\ref{lem:KatoRellichBoundWeakCoupling}
and Kato-Rellich theorem we conclude that $A+uB-c''$ with domain $D(|A|^{\ell+1})$ is maximal dissipative in $D(|A|^\ell)$ for some constant $c''$ sufficiently large. This implies $D(|A+u B|^{\ell+1})=D(|A|^{\ell+1})$.

This completes the iteration and provides the conclusion.
\end{proof}

\subsection{Higher regularity}
\label{SEC_Higher_regularity}
From Lemma~\ref{Rem:NonInteractingCase} and Proposition~\ref{prop:continuity}, we
deduce the following statement.

\begin{proposition}\label{PRO_cont_entree_sortie_Hk-BV}
Let $k$ be a nonnegative real, $(A,B)$ be $k$-mildly coupled, $B$ be $A$-bounded, and $K = [-1/(2\|B\|_{A}),1/(2\|B\|_{A})]$.
For any $u\in BV([0,T],K)$ consider the
family of contraction propagators $\Upsilon^u$ associated with $A+u(t)B$. Then
$\Upsilon_{t,s}^u(D(|A|^{k/2}))\subset D(|A|^{k/2})$, for any $(s,t)\in\Delta_{[0,T]}$, and: 
\begin{itemize}
 \item[$(i)$] for any $t\in [0,T]$ and for any $\psi_0\in D(|A|^{k/2})$
 $$
  \|\pro^{u}_{t}(\psi_0)\|_{k/2}\leq e^{c_k(A,B) \int_0^t|u|}  \|\psi_0\|_{k/2}.
$$
 \item[$(ii)$] for any $t\in [0,T]$ and for any $\psi_0\in D(|A|^{1+k/2})$
 there exists $M$ 
  (depending only on $A$, $B$, and $\|u\|_{L^\infty([0,T])}$) such that
$$
 \|\pro^{u}_{t}(\psi_0)\|_{1+k/2}\leq Me^{M \TV{u}{[0,t]}{K}}e^{c_k(A,B)\int_0^t \left|u\right|}  \|\psi_0\|_{1+k/2}.
 $$
 \end{itemize}
Moreover, for every $\varepsilon$ in $(0,1+k/2)$, for every
$\psi_{0}$ in $D(|A|^{k/2+1-\varepsilon})$, the end-point
mapping
\begin{align*}
\pro_T(\psi_{0}):BV([0,T],K)&\rightarrow  D(|A|^{k/2+1-\varepsilon})\\
u&\mapsto \Upsilon^u_{T}(\psi_{0})
\end{align*}
is continuous.
\end{proposition}
\begin{proof}
Let us begin with the case $k=0$. By hypothesis, $(A,B,K)$ satisfies Assumption~\ref{ass:ass} and, by Lemma~\ref{Rem:NonInteractingCase},  $t\mapsto A+u(t)B$ satisfies Assumption~\ref{ass:kato-generale} for every $u$ in $BV(I,K)$. The statements $(i)$ and $(ii)$ for $k=0$ follow from Theorem \ref{thm:kato}. The continuity of the end-point mapping with value in $\Hc$ follows from Corollary~\ref{cor:cont_input_output_bv} and item $(ii)$.

The idea to deal with the case $k>0$ is then to prove the existence of propagator in $D(|A|^{k/2})$. By Lemma~\ref{Lem:kfeps-sfeps}, this implies existence  of a propagator in $D(|A|^{s/2})$ for any $s\in[0,k]$.  Since  
$D(|A|^{k/2})\subset \Hc$, by uniqueness, Theorem~\ref{thm:kato}, each propagator is the restriction of the one defined for $s=0$.

Consider $k>0$.  If $c_k(A,B)=0$, Lemma \ref{lem:KatoRellichBoundWeakCoupling} ensures that the triple $(A,B,K)$ satisfies Assumption \ref{ass:ass} in $D(|A|^\frac{k}{2})$ and, for any $u$ in $BV(I,K)$, the mapping $t\mapsto A+u(t)B$ satisfies Assumption \ref{ass:kato-generale} in  $D(|A|^\frac{k}{2})$. Statements $(i)$ and $(ii)$ follows from Theorem \ref{thm:kato}. In the case where $c_k(A,B)>0$, in order to obtain contraction semigroups, we  consider $A(t)=A+u(t)B-c_k(A,B)|u(t)|$.  This induces minor technical variations in the proof to check that $t\mapsto A(t)$ satisfies Assumption \ref{ass:kato-generale}. For the reader's sake, we detail them below.

    Using 
$(A,B)$ to be $k$-{mildly coupled}, in Lemma~\ref{lem:KatoRellichBoundWeakCoupling} and 
Kato--Rellich theorem for dissipative operators (see \cite[Corollary of Theorem X.50]{reed-simon-4}) provides that $A(t)$ satisfies
Assumption~\ref{ASS:maximalDissipative} with $I=[0,T]$ and $D(|A|^{\frac{k}{2}})$ instead of $\Hc$. Notice that indeed the domain of $A$ as an operator acting on $D(|A|^{\frac{k}{2}})$ is $D(|A|^{1+\frac{k}{2}})$. In the following, we check Assumptions~\ref{Ass:BV} and \ref{Ass:ResolventBound}.

 From Lemma~\ref{lem:KatoRellichBoundWeakCoupling},  
if $a \in (\|B\|_A,2\|B\|_{A})$ we deduce that there exists $b_a$ such 
that for any $\psi\in D(|A|^{1+k/2})$
\begin{align*}
 \|(1-A-uB+ c_k(A,B)|u|) \psi\|_{k/2}&\geq & \|(1-A)\psi\|_{k/2}- |u|\|B \psi\|_{k/2} -  c_k(A,B)|u| \|\psi \|_{k/2} \\
 &\geq& (1-a|u|)\|(1-A)\psi\|_{k/2}-|u|(b_a +c_k(A,B)) \|\psi\|_{k/2} 
\end{align*}
or
\[
 \|(1-A-uB + c_k(A,B)|u|) \psi\|_{k/2}+|u|(b_a+ c_k(A,B))\|\psi\|_{k/2}\geq  (1-a|u|)\|(1-A)\psi\|_{k/2}.
\]
Note that for the choice of $K$ and $a$ we have that $a|u| <1$. Therefore
\[
 \|B\psi\|_{k/2}\leq a\|(1-A)\psi\|_{k/2}\leq \frac{a}{1-a|u|}\|(1-A-uB+ c_k(A,B)|u|) \psi\|_{k/2}+\frac{b_a+ c_k(A,B)}{1-a|u|}\|\psi\|_{k/2}.
\]
Hence, 
\begin{align*}
 \|(1-A(t))^{-1}\|_{L(D(|A|^{k/2}),D(|A|^{1+k/2}))}&=\|A(1-A-u(t)B+ c_k(A,B)|u(t)|)^{-1}\|_{k/2}\\
 &\leq \|(A+u(t)B- c_k(A,B)|u(t)|)(1-A-u(t)B+ c_k(A,B)|u(t)|)^{-1}\|_{k/2}\\
 &\qquad+|u(t)|\|B(1-A-u(t)B+ c_k(A,B)|u(t)|)^{-1}\|_{k/2}\\
 &\leq 2+\frac{|u(t)|}{1-a|u(t)|} \left ({a}+{b_a+ c_k(A,B)} \right ).
 \end{align*}
 Recall that, by assumption,  $\sup_{t \in [0,T]} a|u(t)|\leq \frac{a}{2\|B\|_A}<1$.
Taking the supremum on $t \in [0,T]$  leads to 
\begin{equation}\label{eq:boundon1-A}
\sup_{t\in[0,T]} \|(1-A(t))^{-1}\|_{L(D(|A|^{k/2}),D(|A|^{1+k/2}))}
\leq 2+\frac{a}{2 \|B\|_A -a} \left ( a+ b_a +c_k(A,B)\right ). 
\end{equation}

As moreover for $A_n(t)=A+u_n(t)B-|u_n(t)|c_k(A,B)$ and $A(t)=A+u(t)B-|u(t)|c_k(A,B)$ and $\lambda$ sufficiently large such that
$\TV{A_n}{[0,T]}{L(D(A),\Hc)} \leq \TV{u_n}{[0,T]}{K}(\|B\|_{L(D(A),\Hc)}+c_k(A,B)) $,
$\|A_n(0)\|_{L(D(A),\Hc))}\leq 1+|u_n(0)|(\|B\|_{L(D(A),\Hc)}+c_k(A,B))$, 
and
\begin{align*}
 (A_n(t)-\lambda)^{-1}-(A(t)-\lambda)^{-1}=&(u_n(t)-u(t))(A_n(t)-\lambda)^{-1}B(A(t)-\lambda)^{-1} \\&+ (|u_n(t)|-|u(t)|)(A_n(t)-\lambda)^{-1}c_k(A,B)(A(t)-\lambda)^{-1} 
\end{align*}
so that the strong resolvent convergence of $A_n$ to $A$ turns to be a
consequence of the convergence of $u_n$ to $u$ in $BV([0,T],K)$.  
\end{proof}

\begin{remark}
The bound on the control $|u| \leq 1/(2\|B\|_{A})$ in Proposition~\ref{PRO_cont_entree_sortie_Hk-BV} is  technical. We could enlarge the set of admissible control  and consider $K = [-1/\|B\|_{A} + \varepsilon, 1/\|B\|_{A}- \varepsilon]$ for some $\varepsilon>0$.  In this case the constant $a$ in the proof would be in the open interval $(\|B\|_{A}, \|B\|_{A}/(1-\varepsilon \|B\|_{A}) )$, the bound~\eqref{eq:boundon1-A} would depend on $\varepsilon$, and would tend to infinity as $\varepsilon$ goes to $0$. 
\end{remark}

We now state another version of
Corollary~\ref{Cor:NoExactControllability}.  
\begin{corollary}\label{Cor:NoExactControllability2}
Let $k$ be a nonnegative real. Let
$(A,B)$ be $k$-{mildly coupled}, $B$ be $A$-bounded,   and $K = (-1/(2\|B\|_A),1/(2\|B\|_A))$. Then, for every $\varepsilon$ in $(0,1+k/2)$ and every $\psi_0$ in $D(|A|^{1+k/2-\varepsilon})$ 
$
 \left\{ \ \Upsilon_t^{u}(\psi_{0}), 
u \in BV([0,+\infty),K),  t\geq 0 \right\}
$
is  a  countable union of relatively compact subsets in 
$D(|A|^{\frac{k}{2}+1-\varepsilon})$.
\end{corollary}

\begin{proof}
The proof follows step-by-step the principle exposed in Section \ref{SEC_principe_obstruction} and the proof of Corollary \ref{COR_obstruction_contr_exacte_BV}.
\end{proof}

\begin{proof}[Proof of Theorem~\ref{THE_intro_weak_coupling}]
Theorem~\ref{THE_intro_weak_coupling} is  consequence of Corollary~\ref{Cor:NoExactControllability2} when $\|B\|_{A}$ vanishes. 
\end{proof}

\paragraph{Remarks on the exact controllability associated with the time reversibility.}\label{sec:timereversibility}
Let $(A,B,K)$ satisfies Assumption~\ref{ass:ass} (or Assumptions~\ref{ass:Interacting}) with $A$ skew-adjoint and $B$ skew-symmetric then $(-A,-B,K)$ satisfies Assumption~\ref{ass:ass} (or Assumptions~\ref{ass:Interacting}). If $(A,B)$ is $k$-{mildly coupled} then $(-A,-B)$ is $k$-mildly coupled.

For $u$, a bounded variation function (or a Radon measure, see Section~\ref{SEC_weak_coupling_for_radon} below) on $(0,T]$ with value in $K$ 
and  $\Upsilon^u$ the associated contraction propagator.
For any $(t,s)\in \Delta_{[0,T]}$, $\Upsilon^u_{t,s}$ is unitary and its inverse coincides with
$\Upsilon^{u(T-\cdot)}_{T-s,T-t}$ where $u(T-\cdot)$ denotes $t\in[0,T]\mapsto u(T-t)$ in the framework of Assumption~\ref{ass:ass} (or $t\in[0,T]\mapsto u((0,T])-u((0,t])=u([t,T))$ in the framework of Assumption~\ref{ass:Interacting}).

\subsection{Extension to Radon measures } \label{SEC_weak_coupling_for_radon}

The conclusion of Proposition~\ref{rem:compactness} can be extended 
to $D(|A|^{k/2})$ if Assumption~\ref{ass:commutator} holds true in $D(|A|^{\frac{k}{2}})$ 
instead of $\Hc$. 
This is indeed the only missing assumption needed in order to apply 
Corollary~\ref{cor:continuity} with $D(|A|^{\frac{k}{2}})$ 
instead of $\Hc$. Without this assumption the following result together with
the interpolation result of Lemma~\ref{LEM_interpolation} gives an interesting
extension.

\begin{proposition}\label{prop:0987}
Let $k$ be a positive real. Let 
$(A,B)$ satisfy Assumption \ref{ass:StrongInteraction} and be $k$-mildly coupled. Then, for every $s\in[0,k]$, $\psi_{0} \in
D({|A|}^{s/2})$, for every
$T\geq 0$, one has $\pro^{\mathfrak{d}v}_{T}(\psi_{0})\in D({|A|}^{s/2})$ and
 \[
 \|\pro^{\mathfrak{d}v}_{T}(\psi_{0})\|_{{s/2}} \leq
e^{\frac{s}{k}c_k(A,B) |u|([0,T]) }\|\psi_{0}\|_{{s/2}}
 \]
for every $v$ in $BV([0,T],K)$ with derivative $v'=u\in \radon{[0,T]}$.
\end{proposition}
\begin{proof}
We give the proof for $s=k$, then by Lemma~\ref{Lem:kfeps-sfeps} the proof
applies to the case $s<k$.

Consider a sequence $v_n$ of piecewise constant functions converging to
$v$ pointwise with $\|v_n\|_{BV([0,T])}\leq K$. Then
$v_n$ is the cumulative function of $v_n'$, a discrete sum of Dirac delta
functions and, from \eqref{Eq:InteractionFramework}, $\pro^{\mathfrak{d}v_n}_t$ is a
product of unitary operators of the form 
$$e^{vB} e^{-vB} e^{tA} e^{vB}=e^{t  A} e^{vB}.$$ 
So that, for every $\psi$ in $D(|A|^{k/2})$, 
$$\|e^{vB} e^{-vB} e^{tA} e^{vB}\psi\|_{k/2}=\|e^{vB}\psi\|_{k/2}\leq
M(v)\|\psi\|_{k/2}$$
where $M(v):=\|e^{vB}\|_{L(D(|A|^{k/2}),D(|A|^{k/2}))}$. From  
Definition~\ref{Def:WeakCoupling}, equation~\eqref{eq:ck}, and $ M(v_1+v_2)\leq M(v_1)M(v_2)$ for any
pair $(v_1,v_2)$ in $[0,\delta]^2$
we have
\[
 M(v)\leq e^{c_k(A,B)|v|}, \quad \mbox{ for all } v \in \R.
\]
Hence, for every $n$, 
$$
 \|\pro^{v'_n}_{t}(\psi_{0})\|_{{k/2}} \leq e^{c_k(A,B)K} \| \psi_0\|_{k/2}.
 $$
For every $f$ in $D(|A|^k)$,
$$|\langle |A|^k f, \pro^{v'_n}_t \psi_0 \rangle| \leq \|f\|_{k/2} \|\psi_0\|_{k/2}
e^{c_k(A,B) K}.
$$
Because of the continuity result (Proposition \ref{prop:continuity}, Corollary \ref{cor:continuity} and Remark~\ref{Rem:CommAssumption}), the left
hand side tends to $|\langle |A|^k f, \pro^{u}_t \psi_0 \rangle|$ as $n $ tends
to infinity. 
Hence, for every $f$ in $D(|A|^k)$
$$|\langle |A|^k f, \pro^{u}_t \psi_0 \rangle|\leq \|f\|_{k/2} \|\psi_0\|_{k/2} e^{
c_k(A,B)
K}.$$
As a consequence, $\pro^{u}_t \psi_0 $ belongs to $D((|A|^{k/2})^{\ast})=D(|A|^{k/2})$
and 
\[\| |A|^{k/2}  \pro^{u}_t \psi_0 \|\leq  \|\psi_0\|_{k/2} e^{ c_k(A,B)
K}.\qedhere\]
\end{proof}
\begin{remark}
Assumption~\ref{ass:commutator} implies that 
$(A,B)$ is $2$-mildly coupled. Indeed, if $(A,B)$ is a pair of skew-adjoint operators satisfying 
Assumption~\ref{ass:StrongInteraction}, then Assumption~\ref{ass:commutator} implies, see~Remark \ref{Rem:CommAssumption}, for small $|t|$ that, for every $\psi$ in $D(A)$,
\begin{align*}
\||A|e^{-tB}\psi\|=\|Ae^{-tB}\psi\|=\|e^{tB}Ae^{-tB}\psi\|
&\leq \|e^{tB}Ae^{-tB}\psi-A\psi\| +\|A\psi\|\\
&\leq (1+L|t|)\|A\psi\|\leq e^{L|t|}\|A\psi\|=e^{L|t|}\||A|\psi\|
\end{align*}
as the map $t\in \R \mapsto e^{t B}Ae^{-t B} \in
L(D(A),\Hc)$ is locally Lipschitz with constant $L$. Thus $(A,B)$ is $2$-mildly coupled. 
\end{remark}

As a consequence of Corollary~\ref{cor:continuity} and
Lemma~\ref{LEM_interpolation} we have the following proposition.
\begin{proposition}\label{PRO_cont_entree_sortie_Hk}
Let $k$ be a positive real,
let
$(A,B)$ satisfy Assumption~\ref{ass:StrongInteraction}, and let $(A,B)$ be $k$-mildly coupled.
Then for any $s\in [0,k)$, for every $\psi_{0}$ in
$D(|A|^{s/2})$, 
 the end-point
mapping
\begin{align*}
\pro(\psi_{0}):BV([0,T],\R)&\rightarrow D(|A|^{s/2})\\
v&\mapsto \pro_T^{\mathfrak{d}v} (\psi_{0}),
\end{align*}
is continuous.
\end{proposition}
\begin{proof}
 Let $(v_n)_{n\in \N}$ be a converging sequence in  $BV([0,T],\R)$ to some $v$ in  $BV([0,T],\R)$. Then 
 $
  \pro_T^{\mathfrak{d}v_n}(\psi_{0})-\pro_T^{\mathfrak{d}v}(\psi_{0})
 $
 is uniformly bounded in $D(|A|^{k/2})$ (by Proposition \ref{prop:0987}) and converges to $0$ in $\Hc$ (by Proposition \ref{prop:continuity}, Corollary \ref{cor:continuity} and Remark \ref{Rem:CommAssumption}). By
 Lemma~\ref{LEM_interpolation} it converges to $0$ in $D(|A|^{s/2})$ for $s<k$.
\end{proof}

\begin{remark}\label{Rem:CompactnessRadonMidlyCoupled}
Under the assumptions of
Proposition~\ref{PRO_cont_entree_sortie_Hk}
both  Proposition~\ref{rem:compactness}
and Corollary~\ref{Cor:NoExactControllability} 
extend to $D(|A|^{s/2})$ for $s \in [0,k)$.
\end{remark}

\section{Bounded control potentials}\label{Sec:BoundedPotential}

\subsection{Dyson expansion solutions}
In the Hilbert setting, if $A$ is maximal dissipative and $B$ stabilizes $D(A)$ Corollary~\ref{Cor:NoExactControllability} provides an extension of \cite[Theorem 3.6]{BMS} to $L^1$ controls. This result can be extended to the Banach framework, with $A$ being generator of a strongly continuous semigroup.
Below we extend Corollary~\ref{Cor:NoExactControllability} to bounded control potentials $B$, for $L^1$ controls, to the Banach framework without any further assumption in $B$, providing a proof of Proposition~\ref{PRO_introduction_BMS_Bborne}.

Troughout this section only, we consider a Banach space $\X$ and we assume that $A$, acting on $\X$, is the generator of a
strongly continuous semigroup with domain $D(A)$ and $B$ is bounded. Then for every
$u$ in $\mathbf{R}$, $A+uB$ is also a generator of a strongly continuous
semigroup with domain $D(A)$. This can be deduced form an analysis of the Dyson expansion.

Since $A$ generates a strongly continuous semigroup there exist $C_A>0$ and $\omega\in \R$
such that
\begin{equation}\label{eq:ExpGrowthA}
 \|e^{tA}\|\leq C_A e^{\omega t},\quad \forall t>0.
\end{equation}
For the equivalent norm
\begin{equation}\label{eq:Nnorm}
N(\psi)=\sup_{t>0} \|e^{t(A-\omega)}\psi\|,
\end{equation}
we have that $A-\omega$ is the generator of a contraction semigroup. 
If $B \in \Bc(\X)$ is bounded for the norm  $N$ let $\|B\|_N$ be its norm. Now for every $u \in BV([0,T],[-R,R])$ we consider the family of operators $A-\omega+u(t)B-R\|B\|_N$ which
satisfies the assumptions of~\cite{Kato1953} in the Banach space structure associated with the norm $N$. 
So that in this case the results of Section~\ref{Sec:Propagators} are still valid.

It is classical (see \cite{BMS}) that the input-output mapping $\Upsilon$ admits a unique continuous
extension to $L^1(\R,\R)$. We consider below the extension to Radon measures.

\nuovo{
\begin{definition}\label{DEF_solution_mild_Radon}
Let $A$, with domain $D(A)$, be the generator of a strongly continuous semigroup on $\X$
and let $B$ be bounded on $\X$. We say that $\psi:[0,T]\to \X$ is a mild solution of
\eqref{eq:main} on $[0,T]$ if $\psi$ is bounded on $[0,T]$ and there exists $\psi_0$ in $\X$ such that, for every $t$ in $[0,T]$,
\begin{equation}\label{EQ_def_solution_mild_radon}
\psi(t)=e^{tA}\psi_0 + \int_{[0,t)}
e^{(t-s_1)A}B\psi(s) \mathrm{d}u(s_1).
\end{equation}
\end{definition}

If $X_1$ and $X_2$ are two metric spaces,  we denote by 
$\Bor(X_1,X_2)$ the space of bounded borelian functions from $X_1$ to $X_2$. If 
 $X_2$ is a Banach spaces, then $\Bor(X_1,X_2)$ is a Banach space as 
well, endowed with the $\sup$ norm $\| \psi\|_{\Bor(X_1,X_2)}=\sup_{t\in X_1}\|\psi(t)\|_{X_2}$ . 

\begin{remark}
Definition \ref{DEF_solution_mild_Radon} only makes  sense 
if \eqref{EQ_def_solution_mild_radon} holds for every $t$ in 
$[0,T]$, since an equality valid almost everywhere only would miss the atoms of 
$u$. For this reason, we will consider solutions in $\Bor([0,T],\X)$ instead of $L^\infty$.   
\end{remark}

Proposition \ref{PRO_proprietes_elementaires_solutions_mild_Radon} below states  immediate regularity properties of the mild solutions.  
\begin{proposition}\label{PRO_proprietes_elementaires_solutions_mild_Radon}
Let $T>0$, $u$ in $\radon{[0,T]}$ and $\psi$ be a mild solution of \eqref{eq:main} on $[0,T]$, associated with $\psi_0$ in $D(A)$. 
 Then $\psi$ has bounded variation, is left continuous everywhere on $[0,T]$ and $\psi(0)=\psi_0$.  Moreover, the discontinuities of $\psi$ (if any) happen on atoms of $u$, and, for every $t$ in $[0,T]$, 
\[ \psi(t+0)-\psi(t)= \psi(t+0)-\psi(t-0)=u(\{t\}) B \psi(t)= u(\{t\}) B \psi(t-0).\]
\end{proposition}
\begin{proof}
 We have that 
\begin{equation*}
\psi(t-0)=e^{tA}\psi_0 + \int_{[0,t)}
e^{(t-s_1)A}B\psi(s) \mathrm{d}u(s_1)\quad\mbox{ and }\quad \psi(t+0)=e^{tA}\psi_0 + \int_{[0,t]}
e^{(t-s_1)A}B\psi(s) \mathrm{d}u(s_1),
\end{equation*}
due to inner and outer regularity of Radon measures.
\end{proof}
}

\begin{theorem}\label{Thm:ExistenceBoundedPotential}
Let $A$, with domain $D(A)$, be the generator of a strongly continuous semigroup on $\X$
and let $B$ be bounded on $\X$. 
Then, for every $\psi_0$ in $\X$, for every $T>0$, for any $u\in \radon{[0,T]}$, for every $s$ in $[0,T)$,  the Cauchy problem \eqref{eq:main} with initial condition $\psi_0$ at time $s$,  admits a unique mild solution $t\mapsto \Xi^u_{t,s}\psi_0$  bounded in $\X$ uniformly on $[s,T]$. That is for every $t \in \Delta_{[0,T]}$
\begin{align*}
\Xi^u_{t,s}\psi_0&=e^{(t-s)A}\psi_0 + \int_{[s,t)}
e^{(t-s_1)A}B\Xi^u_{s_1,s}\psi_0 \mathrm{d}u(s_1)\\
\sup_{(s,t)\in \Delta_{[0,T]}}& \|\Xi^u_{t,s}\psi_0 \| <\infty.
\end{align*}

Moreover
\begin{itemize}
\item[$(i)$] $\Xi^u(s,s)=I_{\X}$,
\item[$(ii)$] $\Xi^u_{t,s}=\Xi^u(t,r)\Xi^u(r,s)$, for any $s<r<t$, 
\item[$(iii)$] if $u$ has bounded variation on $[0,T]$, for any $\psi_0\in \X$, $(s,t)\in\Delta_{[0,T]}\mapsto \Xi^u_{t,s}\psi_0$
is strongly continuous in $s$ and $t$ and
if $\psi_0 \in D(A)$ then it is
strongly right differentiable in $t$ with 
derivative $(A+u(t+0)B)\pro^u(t,s)\psi_0$,
\item[$(iv)$] for any $u \in \radon{[0,T]}$, 
$\Xi^u$ satisfies
$$
 \|\Xi^u_{t,s}\|_{L(\X)}\leq C_Ae^{\omega|t-s|+|u|([s,t])C_A\|B\|},
$$
\item[$(v)$] for any $r>0$, $R>0$, $\psi_0\in \X$ with $\|\psi_0\|=r>0$, the set
\[
 \left\{\Xi^u_{t,s}\psi_0 \mid u\in \radon{[0,T]}, |u|((0,T])\leq R, (s,t)\in \Delta_{[0,T]}\right\}
\]
is relatively compact.
\end{itemize}
\end{theorem}

\begin{proof}
Let $u\in\radon{[0,T]}$.  We first consider the case where $\psi_0$ belongs to $D(A)$ and we extend the result by density. Then, by Proposition \ref{PRO_proprietes_elementaires_solutions_mild_Radon}, any mild solution of \eqref{eq:main} taking value $\psi_0$ at time $0$ satisfies 
\[
\psi(t)=e^{tA} \psi_0^+ +  \int_{s\in (0,t)} e^{(t-s)A} B \psi(s) \mathrm{d}u(s),  
\]
with $\psi_0^+=\psi_0+u(\{0\})B \psi_0$. In other words, the restriction to $(0,T]$ of the mild  solutions (if any) of  \eqref{eq:main} on $[0,T]$ taking value $\psi_0$ at time $0$ are exactly the fixed points of 
 \[
 \begin{array}{llcl}
 F^u_T:&\Bor((0,T],\X) &\to & \Bor((0,T],\X) \\
 & \psi & \mapsto & e^{tA} \psi_0^+ + \int_{s\in (0,t)} e^{(t-s)A} B \psi(s) \mathrm{d}u(s).  
 \end{array}
 \]
 We aim to prove that $F_T^u$ has a unique fixed point in $\Bor((0,T],\X)$. For 
 this, we will prove that $(F_T^u)^j$ is a contraction from $\Bor((0,T],\X)$ to itself, for an integer 
 $j$ large enough.   
We define
\[
 \begin{array}{llcl}
 G^u_T:&\Bor((0,T],\X) &\to & \Bor((0,T],\X) \\
 & \psi & \mapsto & \int_{s\in (0,t)} e^{(t-s)A} B \psi(s) \mathrm{d}u(s).
 \end{array}
 \]
As $B$ is bounded, for every $\psi$ in $\Bor((0,T],\X)$, for every $n$ in $\N$ 
\begin{eqnarray*}
\lefteqn{\|\left ( G^u_T\right )^n(\psi)\|_{\Bor((0,T],\X)}} \\
 &\leq & \left \| \int_{0<s_1<s_2 <\ldots <s_n< t}\!\!\!\!\!\!\!
e^{(t-s_n)A}Be^{(s_{n}-s_{n-1})A}\ldots Be^{(s_2-s_1)A}B e^{(s_1-s)A} \psi_0 \mathrm{d}u(s_1)
\mathrm{d}u(s_2)\ldots \mathrm{d} u(s_n) \right \|_{\Bor((0,T],\X)}\\
&\leq & e^{\omega (t-s)}
C_A^{n+1}\|B\|^n\|\psi\|_{\Bor((0,T],\X)} \int_{0<s_1< s_2 < \ldots < s_n< T}
\mathrm{d}|u|(s_1)\mathrm{d}|u|(s_2)| \dots \mathrm{d}|u|(s_n),\\
\end{eqnarray*}
and since $(0,T)^n$ contains the disjoint union of $\{0 <s_{\sigma(1)}< s_{\sigma(2)} < \dots < s_{\sigma(n)} < T\}$ over all permutations $\sigma$ of $\{1,2,\ldots,n\}$
\[ 
\|\left ( G^u_T(\psi)\right )^n\|_{\Bor((0,T],\X)}
\leq e^{\omega T}
C_A^{n+1}\|B\|^n\|\psi\|_{\Bor((0,T],\X)}\frac{|u|((0,T))^n}{n!}.
\] 
Note that, for $\phi_0^+(t) :=e^{tA} \psi_0^+$, 
\[
(F^u_T)^n(\psi)= \sum_{k=0}^{n-1}(G^u_T)^k(\phi_0^+)+(G^u_T)^n(\psi),
\]
converges in $\Bor((0,T],\X)$. 

In particular, there exists $n$ large enough such that $(F^u_T)^n$ is a contraction from $\Bor([0,T],\X)$ to itself, hence admits a unique fixed point $\psi_\infty$.  Since
 \[ (F^u_T)^{n} \circ F^u_T (\psi_\infty) =  (F^u_T)^{n+1} (\psi_\infty)= F^u_T \circ (F^u_T)^{n} (\psi_\infty)= F^u_T(\psi_\infty), 
 \]
 we have that $F^u_T(\psi_\infty)$ is also a fixed point of $(F^u_T)^n$, hence 
 $F^u_T(\psi_\infty)=\psi_\infty$, or, in other words, $\psi_\infty$ is the restriction to $(0,T]$ of the unique mild solution $t\mapsto \Xi^u_{t,0}(\psi_0)$ of \eqref{eq:main} on $[0,T]$ taking value $\psi_0$
 at time $0$. 
\nuovo{Conversely, noticing that the restriction to $(0,T]$  of any mild solution of \eqref{eq:main}, with initial condition $\psi_0$, is a fixed point of $F^u_T$ provides the uniqueness of the mild solution of \eqref{eq:main}.}

Setting the initial time at $s$ instead of $0$, the linear map $\Xi^u_{t,0}:\X \to \X$ extends to $t\geq s$, by $\Xi^u_{t,s}:\X \to \X$ such that
\[
\Xi^u_{t,s}\psi_0=e^{(t-s)A} \psi_0 +  \int_{r\in [s,t)} e^{(t-r)A} B \Xi^u_{r,s}\psi_0 \mathrm{d}u(r).
\]
The core idea of the Dyson expansion is to express $\Xi^u_{t,s}$ as the sum of a series. Precisely, we define for every $n\in \N$, the linear operator
$$
W^u_{(n)}(t,s):\psi_0^+ \in \X \mapsto 
\int_{s<s_1<s_2 <\ldots <s_n< t}\!\!\!\!\!\!\! \!\!\!\!\!\!\!\!\!\!\!\!\!\!
\!\!\!\!\!\!\!
e^{(t-s_n)A}Be^{(s_{n}-s_{n-1})A}\ldots Be^{(s_2-s_1)A}B e^{(s_1-s)A} \psi_0^+ \mathrm{d}u(s_1)
\mathrm{d}u(s_2)\ldots \mathrm{d} u(s_n).  
$$

Notice that
$$
W_{(0)}^u(t,s) \psi_0^+ = e^{(t-s)A}\psi_0^+,
\quad \mbox{ and } \quad
W^u_{(n+1)}(t,s)\psi_0^+ = \int_{(s,t)}e^{(t-\tau)A}  B W^u_{(n)}(\tau,s)\psi_0^+ \mathrm{d}u(\tau).
$$
Using the very same computation used above to prove that $G^u_T$ is a contraction, we get 
\[
\|W^u_{(n)}(t,s)\psi_0^+\|\leq e^{\omega (t-s)}
C_A^{n+1}\|B\|^n\|\psi_0^+\|\frac{|u|((s,t))^n}{n!},
\]
which proves by uniqueness that
\begin{equation}\label{Eq:UpsilonBoundedExpansion}
\Xi^u_{t,s} =  \sum_{n=0}^\infty W^u_{(n)}(t,s) \circ (1+u(\{s\})B),
\end{equation}
converges in norm in the set $L(\X)$ of the bounded operators of $\X$. This also provides 
\[
 \|\Xi^u_{t,s}\|_{L(\X)}\leq C_A (1+|u|([s,t) \|B\|) e^{\omega|t-s|+|u|([s,t))C_A\|B\|},
\]
\nuovo{and the fact that $\Xi^u_{t,s}$ is a bounded map, which admits an extension to $\X$ by density of $D(A)$. By abuse of notation, we still denote this extension with $\Xi^u_{t,s}$.}

Then we have that
\begin{align*}
\Xi^u_{t,s}\psi_0&=e^{(t-s)A}\psi_0 + \int_{[s,t)}e^{(t-s_1)A}B\Xi^u_{s_1,s}\psi_0 \mathrm{d}u(s_1)\\
&=e^{(t-s)A}\psi_0 + \int_{[s,r)}e^{(t-s_1)A}B\Xi^u_{s_1,s}\psi_0 \mathrm{d}u(s_1)
+\int_{[r,t)}e^{(t-s_1)A}B\Xi^u_{s_1,s}\psi_0 \mathrm{d}u(s_1)\\
&=e^{(t-r)A}\Xi^u_{r,s}\psi_0+\int_{[r,t)}e^{(t-s_1)A}B\Xi^u_{s_1,s}\psi_0 \mathrm{d}u(s_1)\\
&=\Xi^u_{t,r}\Xi^u_{r,s}\psi_0
\end{align*}
where we used the uniqueness in the last identity.

The differentiability properties in the bounded variation case are due to \cite[Theorem 1]{Kato1953}. \nuovo{Indeed recall that 
$A-\omega$ is the generator of a contraction semigroup and since $B$ in $\Bc(\X)$ 
then $B$ is bounded for the norm  $N$ defined in~\eqref{eq:Nnorm}. So that $A-\omega+u(t)B-R\|B\|_N$
for any $R>|u|_\infty$ 
satisfies the assumptions of \cite[Theorem 1]{Kato1953} in the Banach space $\X$
with norm $N$.}

We now consider the compactness property in the last statement. Without loss of generality by linearity and up to scaling $B$, we can assume $r=R=1$.
Let us prove that, for $\|\psi_0\|=1$,
\[
 \left\{\Xi^u_{t,s}\psi_0 \mid u\in \radon{[0,T]}, |u|([0,T])\leq 1, (s,t)\in \Delta_{[0,T]}\right\}
\]
is totally bounded for the topology of $\X$. Then its closure will be totally bounded and complete and thus compact.

Let us consider a radius $\epsilon>0$. In place of $\Xi^u_{t,s}\psi_0$, due to its norm convergence we can consider one of the truncated series in \eqref{Eq:UpsilonBoundedExpansion}, namely
\[
\sum_{n=0}^{n_\epsilon} W^u_{(n)}(t,s)\circ (1+u(\{s\})B),
\]
for some $n_\epsilon \in \N$ such that 
\[
\| \Xi^u_{t,s} - \sum_{n=0}^{n_\epsilon} W^u_{(n)}(t,s)\circ (1+u(\{s\})B)\| \leq (1+|u|([s,t) \|B\|) \sum_{n=n_\epsilon+1}^{\infty} e^{\omega T}
C_A^{n+1}\|B\|^n\|\psi_0\|\frac{1}{n!}\leq \epsilon.
\]
Since we consider a finite number of $W^{\cdot}_{(n)}(\cdot,\cdot)\circ (1+u(\{s\})B)$, namely $n_\epsilon$ of them, it is then enough to prove 
that 
\[
\mathcal{W}_n^{T} :=\left\{W^u_n(t,s)\circ (1+u(\{s\})B)\psi_0 \mid u\in \radon{[0,T]}, |u|([0,T])\leq 1, (s,t)\in \Delta_{[0,T]}\right\},
\]
is totally bounded for the $\X$ topology for any integer $n$. This will be done by iteration on $n\in \N\cup\{0\}$:
\begin{itemize}
 \item For $n=0$, $W^u_0(t,s)\circ (1+u(\{s\})B)\psi_0=e^{(t-s)A}\circ (1+u(\{s\})B)\psi_0$ and since $\Delta_{[0,T]}$ is compact and $|u(\{s\})|\leq 1$, the strong continuity in time of the semigroup provides the compactness of $\mathcal{W}_0^{T}$.
 \item For any integer $n$, we now assume $\mathcal{W}_n^{T}$ is totally bounded. The map 
 $$
 (\tau,t,\psi)\in \Delta_{[0,T]}\times \X\mapsto e^{(t-\tau)A}  B \psi\in \X,
 $$
 is  continuous. So 
\[
\mathcal{Z}_n^{T} :=\left\{e^{(t-\tau)A}BW^u_n(\tau,s)\circ (1+u(\{s\})B)\psi_0 \mid u\in \radon{[0,T]}, |u|([0,T])\leq 1, (s,\tau)\in \Delta_{[0,T]}, (\tau,t)\in \Delta_{[0,T]}\right\}
\]
is totally bounded.

 For any $\delta>0$, there exist $\psi_1,\dots,\psi_{N_\delta}$ in $\X$ such that 
 \[
  \mathcal{Z}_n^{T}\subset \cup_{j=1}^{N_\delta} B_{\X}(\psi_j,\delta).
 \]
 
 Let $\phi_1,\dots,\phi_{N_\delta}$ be a partition of the unity in $\mathcal{Z}_n^{T}$ such that $\mathrm {supp}~ \phi_j\subset \overline{B_\X(\psi_j,2\delta)}$ and $\pi : \psi\in\X\mapsto \sum_{j=1}^{N_\delta} \psi_j \phi_j(x)$.
 
 Define $p^u_n(t,\tau,s):=\pi(e^{(t-\tau)A}BW^u_n(\tau,s)\circ (1+u(\{s\})B)\psi_0)$ and $\phi^u_{n,j}(t,\tau,s):=\phi_j(e^{(t-\tau)A}BW^u_n(\tau,s)\circ (1+u(\{s\})B)\psi_0)$, then 
 \[
  p^u_n(t,\tau,s)=\sum_{j=1}^{N_\delta} \psi_j \phi^u_{n,j}(t,\tau,s)
 \]
 and
 \[
  \|e^{(t-\tau)A}BW^u_n(\tau,s)\circ (1+u(\{s\})B)\psi_0-p^u_n(t,\tau,s)\|\leq 2\delta.
 \]
Thus $\mathcal{W}_{n+1}^{T}$ is totally bounded if
 \[
  \mathcal{P}_n^{T}:=\left\{\int_{(s,t)} p^u_n(t,\tau,s) u(\tau) \mathrm{d}\tau, u\in \radon{[0,T]}, |u|([0,T])\leq 1, (s,\tau)\in \Delta_{[0,T]}, (\tau,t)\in \Delta_{[0,T]}\right\}
 \]
is totally bounded. Since for $u\in \radon{[0,T]}$, $|u|((0,T])\leq 1$, $(s,\tau)\in \Delta_{[0,T]}$ and  $(\tau,t)\in \Delta_{[0,T]}$
\[
 \int_{(s,t)} p^u_n(t,\tau,s) \mathrm{d}u(\tau) =\sum_{j=1}^{N_\delta} \psi_j  \int_{(s,t)} \phi^u_{n,j}(t,\tau,s)  \mathrm{d}u(\tau)
\]
and
\[
 \left|\int_{(s,t)} \phi^u_{n,j}(t,\tau,s) \mathrm{d}u(\tau)\right|\leq |u|((0,T])\leq 1
\]
this implies $\mathcal{P}_n^{T}$ is relatively compact (and thus totally bounded).
\end{itemize}
This concludes the iteration. We thus have the relative compactness of 
\[
 \left\{\Xi^u_{t,s}\psi_0 \mid u\in \radon{[0,T]}, |u|([0,T])\leq 1, (s,t)\in \Delta_{[0,T]}\right\},
\]
concluding the proof.
\end{proof}

\begin{corollary}\label{COR_int_ens-attain_B_borne} Let $A$ be the generator of a strongly continuous semigroup on $\X$, let  $B$ be bounded and denote by $\Xi$ the propagator defined in Theorem \ref{Thm:ExistenceBoundedPotential}. Then for every $\psi_0$ in $\X$, the set
$${\mathcal Att}^{\Xi}_{\mathcal R}(\psi_0):=\bigcup_{T\geq 0} \bigcup_{u \in \radon{[0,T]}} 
\{ \Xi^u_{t,0}\psi_0| t\in [0,T]\}$$ is 
contained in a countable union of compact subsets of $\X$.
\end{corollary}
\begin{proof}
The proof is a consequence of Theorem~\ref{Thm:ExistenceBoundedPotential} and follows the idea of the proof of Corollary~\ref{Cor:NoExactControllability}.
\end{proof}

We are now ready to prove Proposition~\ref{PRO_introduction_BMS_Bborne}.

\begin{proof}[Proof of Proposition \ref{PRO_introduction_BMS_Bborne}] 
As already mentioned, the well-posedness result is classical (see \cite{BMS} for instance), while the property of the attainable set with $L^1$ controls follows from Corollary 
\ref{COR_int_ens-attain_B_borne}. 
\end{proof}

\begin{remark}\label{Rem:RemarkBoundedOperator}
If $\X$ is an Hilbert space $\Hc$, $A$ is skew-adjoint, $B$ is bounded in $D(|A|^{k/2})$, then $D(|A|^{k/2})$ can be considered in stead of
$\X$ in the whole analysis of the present section. This leads to results similar to the ones presented in  Section~\ref{Sec:HigherOrderFEPS} on the mild coupling, but in a simpler way.
\end{remark}

\subsection{On the notion of solution in the Radon framework}\label{sec:nsrf}
Theorem~\ref{Thm:ExistenceBoundedPotential} does not deal with the continuity with respect to the control $u$.
With a Dirac measure $\delta_{t_0}$, $t_0\in(0,T]$, it turns out that the solution built here is 
\begin{equation}\label{Eq:DeltaDyson}
\Xi^u_{t,s}\psi_0=e^{(t-s)A}\psi_0+e^{(t-t_0)A}Be^{(t_0-s)A}\psi_0 
 \mathbb{I}_{[s,t)}(t_0).
\end{equation}
This does not coincide with the generalized propagator in Definition~\ref{def:propagator-radon} even if the framework is  similar, for instance, as in Remark~\ref{Rem:RightContinuity}, when $A=0$ as the latter is
\[
 \Upsilon^u_{t,s}\psi_0=e^{B \mathbb{I}_{(s,t]}(t_0)}\psi_0.
\]
Both expansions coincide only up to the first order term in the control. This discrepancy is due to the lack of continuity of the cumulative function of the control. 

If we restrict the analysis to controls with continuous cumulative functions and set the topology to the one of the total variation, the continuity is restored and both constructions thus coincide. 

As a consequence the propagator in Theorem~\ref{Thm:ExistenceBoundedPotential} is not continuous in $u$ and it is not the sequential extension of the corresponding propagator for, say, controls with continuous cumulative functions when the notion of convergence is the one we choose for $\radon{[0,T]}$.  
This also implies that the accumulations of the compact set
\[
  \left\{\Xi^u_{t,s}\psi_0 \mid u\in \radon{[0,T]}, t\mapsto u((0,t]) \mbox{ is continuous }, |u|((0,T])\leq 1, (s,t)\in \Delta_{[0,T]}\right\},
\]
are not necessarily given by values of the propagator in Theorem~\ref{Thm:ExistenceBoundedPotential} and thus actual solutions but, instead, values of the propagator in the sense of Definition~\ref{def:propagator-radon}.

\subsection{Noninvariance of the domain}

In this section, we consider the invariance of the domain of $A$, in the framework of Theorem~\ref{Thm:ExistenceBoundedPotential}, by $\Xi^u$ when $u$ is in $L^1([0,T],\R)$. Notice that if $u$ is in $L^1([0,T],\R)$, $\Upsilon^{u}$ and $\Xi^{u}$   coincide whenever Assumption~\ref{ass:Interacting} is fulfilled and, hence, the invariance of the domain follows from Theorem~\ref{thm:kato}. The question is whether this remains true when $B$ is bounded but the corresponding $C^0$-semigroup does not preserve $D(A)$. The answer is negative and we provide a counter-example. 

Let $\X=L^2(\R)$, $A=\partial_x$ with $D(A)=H^1(\R)$ and $B=\rmi w$ for some bounded measurable function $w$. This provides a controlled transport equation and the corresponding solution of \eqref{EQ_bilinear_abstrait} with $u\in L^1(\R)$ is given by
\[
 \Xi^u_t(\psi_0)(x)=e^{\frac{\rmi}2\int_{-t-x}^{t-x}u(\frac{t-x+\tau}2)w(\frac{t+x-\tau}2)\,\mathrm{d}\tau}\psi_0(t+x)
\]
which rewrites as
\[
 \Xi^u_t(\psi_0)(x)=e^{\rmi\int_{x}^{t+x}u(t-s)w(s)\,\mathrm{d}s}\psi_0(t+x).
\]
Set $w=\mathbb{I}_{[0,+\infty)}$, then for $t\geq 0$ and $x\geq 0$
\[
 \Xi^u_t(\psi_0)(x)=e^{\rmi\int_{[-x,t-x]}u(s)\,\mathrm{d}s}\psi_0(t+x).
\]
For fixed time $t$, the function $x\mapsto e^{\rmi\int_{[-x,t-x]}u(s)\,\mathrm{d}s}$ is absolutely continuous and the distributional derivative of $x\mapsto \Xi^u_t(\psi_0)(x)$ is given by
\[
 \Xi^u_t(\psi_0)(x)=e^{\rmi\int_{[-x,t-x]}u(s)\,\mathrm{d}s}\left(\psi_0'(t+x)+\rmi(u(-x)-u(t-x))\psi_0(t+x)\right),
\]
for $t>0$ and $x>0$.

If $\psi_0$ is in $H^1(\R)$ then $\Xi^u_t(\psi_0)$ is in $H^1(\R)$ if and only if  
\[
v:x\mapsto (u(-x)-u(t-x))\psi_0(t+x)
\]
is in $L^2(\R)$. 

Set $u:t\mapsto |1-t|^{-1/2}$, which is integrable but not square integrable, and $\psi_0$ a smooth compactly supported function equal to $1$ in $[1-\varepsilon,1+\varepsilon]$, for some $\varepsilon\in (0,1/2)$. Consider $t\in [1-\varepsilon/2,1+\varepsilon/2]$, then $x\mapsto \psi_0(t+x)$ is equal to $1$ on $[1-t-\varepsilon,1-t+\varepsilon]\subset [-\frac32\varepsilon,\frac32\varepsilon]\subset [-3/4,3/4]$. Hence $-1\not\in [1-t-\varepsilon,1-t+\varepsilon]$. While $[-\varepsilon/2,\varepsilon/2]\subset [1-t-\varepsilon,1-t+\varepsilon]$ and $x=t-1\in [-\varepsilon/2,\varepsilon/2]$.  This implies that $v$ is not square integrable on  $[1-t-\varepsilon,1-t+\varepsilon]$ for any $t\in [1-\varepsilon/2,1+\varepsilon/2]$.

\section{Examples}\label{sec:example}

Most of the examples of bilinear control systems \eqref{eq:main} encountered in the literature, 
also the ones not related to quantum control,  deal with bounded control operator $B$. 
Proposition \ref{PRO_introduction_BMS_Bborne} applies and allows, 
for instance, to complete the studies of the rod equation with clamped ends made in 
\cite[Section 6, Example 4]{BMS} and \cite{BeauchardBeam}.
In the following, we focus on examples arising in quantum control.

\subsection{Quantum systems with smooth potentials on compact manifolds}\label{SEC_smooth_potentials}

The following example motivates the present analysis because of its physical relevance. We
consider $\Omega$ a compact Riemannian manifold endowed with the associated
Laplace-Beltrami operator $\Delta$ and the associated measure $\mu$. 
For $r$ a positive real, $D(|\Delta|^{\frac{r}2})=H^{r}(\Omega,\mathbf{C})$. Since $\Omega$ is compact, for $r>r'$, 
 $D(|\Delta|^{\frac{r}2})\subset D(|\Delta|^{\frac{r'}2})$ is a compact embedding.

Let $k\in\N$, let $V,W:\Omega
\to \mathbf{R}$ be two functions  of class $C^{2(k-1)}$, and consider bilinear quantum system 
\begin{equation}\label{EQ-blse-compact-manifold}
 \mathrm{i}\frac{\partial \psi}{\partial t}=\Delta \psi + V \psi +u(t) W \psi.
\end{equation}
Following the notations of Section~\ref{sec:abstract-framework}, 
$\Hc=L^2(\Omega,\mathbf{C})$ is endowed with the Hilbert product $\langle f,g \rangle
=\int_{\Omega} \bar{f} g \mathrm{d}\mu$, $A=-\mathrm{i}(\Delta+V)$, and
$B=-\mathrm{i}W$. As $V$ is continuous and so bounded, $A$ has a spectral gap. 
Up to substracting a sufficiently large constant, we can assume that $A$ is positive and invertible. 

For  a positive real $r$ with $r\leq 2k$, $D(|A|^{\frac{r}2})=H^{r}(\Omega,\mathbf{C})$. 
Since $B$ is bounded from $D(|A|^{\frac{s}{2}})$ to $D(|A|^{\frac{s}{2}})$, for $s$ a postive real with $s\leq 2(k-1)$, $(A,B,\R)$ satisfies 
Assumption~\ref{ass:ass} and $(A,B)$ is $s$-mildly coupled by Proposition \ref{Prop:ThmLumerPhillips}.

In particular, the two notions of propagators $\Upsilon$ and $\Xi$ defined in Proposition~\ref{prop:extensionBVRadon} and Theorem~\ref{Thm:ExistenceBoundedPotential} respectively can be used and we have the following statement. 
 \begin{proposition}\label{prop:smoothpotentials}
For every $T>0$, 
for every $\psi_0$ in $H^{2(k-1)}(\Omega,\mathbf{C})$, the sets
$$
\bigcup_{\alpha \geq 0} 
\bigcup_{T\geq 0} 
\bigcup_{u \in \radon{[0,T]}} \{ \alpha \Upsilon^u_t\psi_0\mid t\in [0,T]\},
$$
and
$$
\bigcup_{\alpha \geq 0} 
\bigcup_{T\geq 0} 
\bigcup_{u \in \radon{[0,T]}} \{ \alpha \Xi^u_t\psi_0 \mid t\in [0,T]\}
$$
are contained in countable unions of compact subsets of
$H^{2(k-1)}(\Omega,\mathbf{C})$ and, in particular, they have dense complement in $H^{2(k-1)}$. 

For any $\varepsilon\in (0,1)$, if $\psi_0$ in $H^{2(k-\varepsilon)}(\Omega,\mathbf{C})$, the set
\[
 \bigcup_{\alpha \geq 0} 
\bigcup_{T\geq 0} 
\bigcup_{u \in BV([0,T],\R)} \{ \alpha \Upsilon^u_t\psi_0 \mid t\in [0,T]\}\\
\]
is contained in a countable union of compact subsets of
$H^{2(k-\varepsilon)}(\Omega,\mathbf{C})$ and, in particular, it has dense complement in $H^{2(k-\varepsilon)}(\Omega,\mathbf{C})$. 
\end{proposition}
\begin{proof}
The first statement is an adaptation of Proposition~\ref{rem:compactness} and Corollary~\ref{Cor:NoExactControllability}, see Remark \ref{Rem:CompactnessRadonMidlyCoupled}.
The second statement follows from Corollary~\ref{COR_int_ens-attain_B_borne}.
The last statement is a consequence of Corollary~\ref{Cor:NoExactControllability2}.
\end{proof}
Notice that from the compactness of the Sobolev embeddings and the conservation of the regularity we can deduce a  result weaker than Proposition~\ref{prop:smoothpotentials} such as the fact
\[
 \bigcup_{\alpha \geq 0} 
\bigcup_{T\geq 0} 
\bigcup_{u \in \radon{[0,T]}} \{ \alpha \Upsilon^u_t\psi_0 \mid t\in [0,T]\}
\]
is contained in a countable union of totally bounded sets of
$H^{2(k-1)-\delta}$ for any $\delta \in (0,1)$ whenever $\psi_0$ in $H^{2(k-1)}$.

 \subsection{Potential well with dipolar interaction}\label{Sec:PotentialWell}
 
  In this example, $\Omega=(0,\pi)$ is endowed with the standard Lebesgue
 measure, $V$ is the constant zero function and $W$ is some function of class $C^k$, for some integer $k$. This academic
 example  is a simplification of the harmonic oscillator,  presented in 
 Section~\ref{SEC_harmonic_oscillator} below, in the sense that $\Omega$ is bounded. It
 has been thoroughly studied by K. Beauchard in~\cite{beauchard,camillo}. These works give the first (and, at this
 time, almost the only one) satisfying description of the reachable set  with $L^2$ controls from the
 first eigenvector for systems of the type of \eqref{eq:blse}.
 Using Lyapunov techniques, V. Nersesyan~\cite{nersesyan} gave practical algorithms  for
 approximate controllability.

  Equation \eqref{eq:blse} writes
  \begin{equation}\label{EQ_potential_well}
  \rmi \frac{\p \psi}{\p t} = - \frac 1 2 \frac{\p^{2} \psi}{\p x^{2}} - u(t)W(x)
 \psi
  \end{equation}
  with boundary conditions $\psi(0)=\psi(\pi)=0$.
  
   The linear operators $A=\frac{\rmi}{2} \Delta$ defined on $D(A)=(H^2\cap
 H^1_0)((0,\pi),\C)$ and $B:\psi\mapsto \rmi W \psi$ are skew symmetric in
the
 Hilbert space $\mathcal{H}=L^2(\Omega,\C)$ endowed with the hermitian product 
  $L^2(\Omega,\C),$
  $$\langle f,g \rangle =\int_{0}^{\pi}\overline{f(x)}g(x)
 \mathrm{d}x.$$
  
  Define, for every $k$ in $\N$,
  $$
  \phi_k:x\mapsto  \sqrt{\frac{{2}}{\pi}} \sin (k x),
   $$
  the family $\Phi=(\phi_k)_{k \in \N}$ is an orthonormal basis of $\Hc$ made of
 eigenvectors of $A$.

The triple $(A,B,\R)$ satisfies Assumption~\ref{ass:ass}. 

Classical results  of interpolation~\cite[Chapter 1]{Lions_Magenes} allow to find the domain of fractional derivative operators. In particular, for any $k$ in $\mathbf{N}$ and $0\leq s <1$, we get following (\cite{Nochetto_Otarola_Salgado}): 
$$\begin{array}{lcll}
D(|A|^k)&=&\{\psi \in H^{2k} | \psi^{[2l]}(0)=\psi^{[2l]}(\pi)=0, l=0,\dots,k-1\}& \mbox{ for } k \in \mathbf{N} \\
D(|A|^{k+s})&=&D(|A|^k) \cap H^{2s} & \mbox{ for } s<1/4\\
D(|A|^{k+\frac{1}{4}})&=&\{\psi \in D(|A|^k)| |A|^k \psi \in H^{\frac{1}{2}}_{00}\} \\
D(|A|^{k+s})&=&\{\psi \in D(|A|^k)| |A|^k \psi \in H_0^{2s}\} & \mbox{ for } 1/4<s<1/2\\
D(|A|^{k+s})&=&\{\psi \in D(|A|^k)| |A|^k \psi \in H^{2s}\cap H^1_0\} & \mbox{ for } 1/2\leq s\leq 1
\end{array} 
$$
where 
$$H^{\frac{1}{2}}_{00}=\left \{ \psi \in H^{\frac{1}{2}} \Big | \int_0^{\pi}\!\!\! \psi^2(x) \frac{\mathrm{d}x}{\sin(x)} <+\infty \right \}$$ is the Lions-Magenes space.

\begin{lemma}\label{LEM_reg_W_puits_potentiel}
Let $p$ in $\N \cup \{0\}$, $W:[0,\pi]\to \R$ be $C^3 \cap C^{2p+1}$. If $p>0$, assume moreover that that $W^{(2l+1)}(0)=W^{(2l+1)}(\pi)=0$ for $l=0,\dots, p-1$. Then $B$ is bounded from $D(|A|^{a})$ to $D(|A|^{a})$ for every $a<p+1 +\frac{1}{4}$. 
\end{lemma}

\begin{proof}
Since $W$ is $C^{2p+1}$, $B$ leaves invariant $H^s$ for every $s\leq 2p+1$.  If $a$ is an integer less than or equal to $p+1$, the result follows from the Leibniz rule, using the vanishing of the derivatives of odd orders less than $2p$ (if any) of $W$ on the boundary of $[0,\pi]$. The result for $a-\lfloor a \rfloor< 1/4$ follows directly from the equalities above with no additional boundary conditions to check.
\end{proof}

 Theorem 3.6 in \cite{BMS} by Ball, Marsden and Slemrod  implies (see~\cite{turinici}) that equation \eqref{eq:blse} is not
 controllable in (the Hilbert unit sphere of) $L^{2}(\Omega)$ when 
 $\psi\mapsto W \psi$ is bounded in $L^2(\Omega)$.
 Moreover, in the case in which  $\Omega$ is a domain of $\R^n$ and $W:\Omega\to \R$ is $C^2$, if
 the control $u$ belongs to $L^p([0,+\infty),\R)$ with $p>1$, then
 equation~\eqref{eq:blse}
 is neither controllable in the Hilbert sphere $\mathbf{S}$ of $L^2(\Omega)$ nor
 in the natural functional space  where the problem is formulated, namely the
 intersection of  $\mathbf{S}$ with the Sobolev spaces $H^2(\Omega)$ and
 $H^1_0(\Omega)$.

The fact that the present system is not more than $5/2$-{mildly coupled} is the purpose of the following lemmas.
\begin{lemma}\label{Lem:IPP}
Let $k\in \N\cup\{0\}$. Let $F:[0,\pi]\to \R$ be of class $C^{2k+3}$ with
$|F^{(2k+1)}(\pi)|+|F^{(2k+1)}(0)|\neq 0$ 
and, 
if $k\neq 0$, 
$F^{(2j+1)}(0)=F^{(2j+1)}(\pi)=0$ for $j=0,\ldots,k-1$.
Then
$F \phi_1$ is not in $D(|A|^{a})$ if $a\geq k+\frac54$ .
\end{lemma}

\begin{proof}
Consider, for any integer $n$, the following quantity
\[
I_n(F):=\frac\pi2\langle F\phi_1,\phi_n\rangle=\int_0^\pi F(x)\sin(x)\sin(n x)\,\mathrm{d}x.
\]
Then we have that $I_n(F)=\frac12 ( J_{n-1}(F)-J_{n+1}(F))$ with
\begin{align*}
J_\ell(F)&:=\int_0^\pi F(x)\cos(\ell x)\,\mathrm{d}x=-\frac1{\ell}\int_0^\pi F'(x)\sin(\ell x)\,\mathrm{d}x\\
&=\frac1{\ell^2}\left((-1)^\ell F'(\pi)-F'(0)\right)-\frac1{\ell^2}J_{\ell}(F'').
\end{align*}
Now assume that $F^{(2j+1)}(0)=F^{(2j+1)}(\pi)=0$ for $j=0,\ldots,k-1$, hence
\[
J_\ell(F)=\frac1{\ell^{2k+2}}\left((-1)^\ell F^{(2k+1)}(\pi)-F^{(2k+1)}(0)\right)-\frac1{\ell^{2k+2}}J_{\ell}(F^{(2k+2)}).
\]
It follows that
\begin{align*}
 I_n(F)&=\frac12 \left(\frac1{(n-1)^{2k+2}}-\frac1{(n+1)^{2k+2}}\right)\left(-(-1)^nF^{(2k+1)}(\pi)-F^{(2k+1)}(0)\right)\\ 
& -\frac12\frac1{(n-1)^{2k+2}}J_{n-1}(F^{(2k+2)})+\frac12\frac1{(n+1)^{2k+2}}J_{n+1}(F^{(2k+2)})\\
&=\frac12 \left(\frac1{(n-1)^{2k+2}}-\frac1{(n+1)^{2k+2}}\right)\left(-(-1)^nF^{(2k+1)}(\pi)-F^{(2k+1)}(0)\right)\\ 
& +\frac12\frac1{(n-1)^{2k+3}}\int_0^\pi F^{(2k+3)}(x)\sin((n-1)x)\,\mathrm{d}x-\frac12\frac1{(n+1)^{2k+3}}\int_0^\pi F^{(2k+3)}(x)\sin((n+1)x)\,\mathrm{d}x
\end{align*}
As
\[
 \frac1{(n-1)^{2k+2}}-\frac1{(n+1)^{2k+2}}=\frac{(n+1)^{2k+2}}{(n^2-1)^{2k+2}}\left(1-\left(\frac{n-1}{n+1}\right)^{2k+2}\right)\underset{n\to\infty}{\sim} \frac{4k+4}{n^{2k+3}} 
\]

If $|F^{(2k+1)}(\pi)|+|F^{(2k+1)}(0)|\neq 0$, then either $F^{(2k+1)}(\pi)-F^{(2k+1)}(0)\neq 0$ or $F^{(2k+1)}(\pi)+F^{(2k+1)}(0)\neq 0$,
 and, due to Riemann--Lebesgue Lemma, 
\begin{itemize}
 \item if $F^{(2k+1)}(\pi)+F^{(2k+1)}(0)\neq 0$ then 
\[
 I_{2n}(F)\underset{n\to\infty}{\sim} -\frac{2k+2}{(2n)^{2k+3}}\left(F^{(2k+1)}(\pi)+F^{(2k+1)}(0)\right),
\]
and hence, $(n^{2a}I_n(F))_{n\in\N}$ is not square integrable if $2a-2k-3\geq -\frac12$ and consequently $F \phi_1$ is not in $D(|A|^{a})$ if $a\geq k+\frac54$ 
\item if $F^{(2k+1)}(\pi)-F^{(2k+1)}(0)\neq 0$ then 
\[
 I_{2n+1}(F)\underset{n\to\infty}{\sim} \frac{2k+2}{(2n+1)^{2k+3}}\left(F^{(2k+1)}(\pi)-F^{(2k+1)}(0)\right)
\]
and similarly $F \phi_1$ is not in $D(|A|^{a})$ if $a\geq k+\frac54$.\qedhere
\end{itemize}
\end{proof}

\begin{lemma}\label{Lem:DomainStability}
Let $k\in \N\cup\{0\}$. Let $W:[0,\pi]\to \R$ be of class $C^{2k+3}$ with 
$|W^{(2k+1)}(\pi)|+|W^{(2k+1)}(0)|\neq 0$ 
and, if $k\neq 0$, 
$W^{(2j+1)}(0)=W^{(2j+1)}(\pi)=0$ for $j=0,\ldots,k-1$.
Then for every $a$ in $(0,+\infty)$, 
\[ 
e^{\mathrm{i}W}\phi_1 \in D(|A|^a) \Leftrightarrow a<\frac{5}{4}+k.
\]
\end{lemma}
\begin{proof}
 Set $F= e^{{\mathrm i}W}$ and recall Fa\`a di Bruno formula
\[
(e^{\mathrm{i}W})^{(n)}(x)=\sum \frac{n!}{m_1!\,1!^{m_1}\,m_2!\,2!^{m_2}\,\cdots\,m_n!\,n!^{m_n}} e^{\mathrm{i}W(x)}\prod_{j=1}^n\left({(\mathrm{i}W)}^{(j)}(x)\right)^{m_j}, 
\]
where the sums is over the $n$-uplets $(m_1,\ldots, m_n)$ in $\N\cup\{0\}$ such that: $1m_1+2m_2+3m_3+\cdots+nm_n=n$.

If $n$ is odd and $(m_1,\ldots, m_n)$ is an $n$-uplets of $\N\cup\{0\}$ such that $1m_1+2m_2+3m_3+\cdots+nm_n=n$ there exists $\ell$ such that $2\ell+1\leq n$ and $m_{2\ell+1}\neq 0$. It follows
that $F:[0,\pi]\to \R$ is of class $C^{2k+3}$ with $F^{(2j+1)}(0)=F^{(2j+1)}(\pi)=0$ for $j=0,\ldots,k-1$ and $|F^{(2k+1)}(\pi)|+|F^{(2k+1)}(0)|\neq 0$.

Then the conclusion follows from Lemma~\ref{LEM_reg_W_puits_potentiel} and Lemma~\ref{Lem:IPP}.
\end{proof}

We sum up our results in the following
\begin{proposition}
Let $k\in \N\cup\{0\}$. Let $W:[0,\pi]\to \R$ of class $C^{2k+3}$ with 
$|W^{(2k+1)}(\pi)|+|W^{(2k+1)}(0)|\neq 0$ 
and, if $k\neq 0$, 
$W^{(2j+1)}(0)=W^{(2j+1)}(\pi)=0$ for $j=0,\ldots,k-1$. 
Then
\begin{align*}
{\mathcal Att}_{\mathcal R}(\phi_1)&= \bigcup_{T\geq 0}
 \bigcup_{u \in \mathcal{R}([0,T])}\{\Upsilon^u_{t,0}\phi_1|0\leq t \leq T\}\subset \bigcap_{s<\frac{5}{4}+k} D(|A|^s),\\
{\mathcal Att}^\Xi_{\mathcal R}(\phi_1)&= \bigcup_{T\geq 0}
 \bigcup_{u \in \mathcal{R}([0,T])}\{\Xi^u_{t,0}\phi_1|0\leq t \leq T\}\subset \bigcap_{s<\frac{5}{4}+k} D(|A|^s),
\end{align*}
and both attainable sets are contained in a countable union of relatively compact subsets of $D(|A|^s)$, for any $s<\frac{5}{4}+k$.

Moreover, we have
\begin{equation*}
{\mathcal Att}_{\mathcal R}(\phi_1) \not \subset  D(|A|^{\frac{5}{4}+k})\quad\mbox{ and }\quad
{\mathcal Att}^\Xi_{\mathcal R}(\phi_1) \not \subset  D(|A|^{\frac{5}{4}+k}). 
\end{equation*}
\end{proposition}
Recall that $\Upsilon$ is defined in Proposition~\ref{prop:extensionBVRadon} and $\Xi$ in Theorem~\ref{Thm:ExistenceBoundedPotential}.
\begin{proof}
From Lemma~\ref{Lem:IPP}, $B$ is bounded from $D(|A|^{a})$ to $D(|A|^{a})$ for every $a<p+\frac{5}{4}$ and hence $(A,B)$ is $a$-{mildly coupled}, for every $a<p +\frac{5}{4}$, by Proposition \ref{Prop:ThmLumerPhillips}. Then Proposition~\ref{prop:0987} provides the first statement. While following Remark~\ref{Rem:RemarkBoundedOperator}, Theorem~\ref{Thm:ExistenceBoundedPotential} provides the second statement.

The relative compactnesses follow from Proposition~\ref{PRO_cont_entree_sortie_Hk} (similarly to Proposition~\ref{rem:compactness}) and Theorem~\ref{Thm:ExistenceBoundedPotential}, respectively.

From \eqref{Eq:InteractionFramework}, with $u=\pi \delta_{t_0}$ for some $t_0>0$ and Lemma~\ref{Lem:DomainStability} we deduce the first noninclusion statement. From \eqref{Eq:DeltaDyson} and Lemma~\ref{Lem:IPP}, we deduce the last assertion.
\end{proof}

\begin{remark}
Notice that \cite[Theorem 2]{camillo} states the exact controllability of \eqref{EQ_potential_well} in $D(|A|^{\frac{5}{2}})$ with $H^1_0$ controls and $W:x\mapsto x^2$. While Proposition \ref{PRO_cont_entree_sortie_Hk-BV} implies nonexact controllability of 
\eqref{EQ_potential_well} in $D(|A|^{s})$, $s<\frac{9}{4}$, with BV controls for example with $W:x\mapsto x^2$. Whether this $1/4$ discrepancy is optimal is still an open question.

Similarly \cite[Theorem 1]{camillo} states the exact controllability of \eqref{EQ_potential_well} in $D(|A|^{\frac{3}{2}})$ with $L^2$ controls and $W:x\mapsto x$. While Proposition~\ref{PRO_cont_entree_sortie_Hk} states the nonexact controllability of 
\eqref{EQ_potential_well} in $D(|A|^{s})$, $s<\frac{5}{4}$, with Radon controls and $W:x\mapsto x^2$. But this time, from  Proposition~\ref{PRO_cont_entree_sortie_Hk}, we deduce that the $1/4$ discrepancy is optimal.
\end{remark}

From \cite{Schrod2}, we know that $\{(k,k+1)|k\in \mathbf{N}\} $ is a nondegenerate chain of connectedness 
for $(A+\eta B,B)$ for almost every real $\eta$. 
Hence Proposition \ref{prop:controlHsFEPS} guarantees  the approximate controllability of the system 
(\ref{EQ_potential_well}) from $\phi_1$ in $D(|A+\eta B|^{r})=
D(|A|^{r})$, for $\frac{3}{2}<r<\frac{5}{4}+1$. The global exact controllability in 
$D(|A|^\frac{3}{2})$ (inside the unit sphere) with explicit 
controls follows from Proposition~\ref{prop:controlHsFEPS}, in order to reach a neighborhood of the target in $D(|A|^{r})$, for $\frac{3}{2}<r<\frac{5}{4}+1$ (see for instance \cite{ACC_QG}). It is then enough to concatenate the dynamics with $L^2$ controls given by \cite{camillo} for exact local controllability. This explicit construction provides estimates on control time and norms, see~\cite{duca:hal-01520173}.

 \subsection{Quantum harmonic oscillator}\label{SEC_harmonic_oscillator}
 
In this section, we present an example of $s$-{mildly coupled} system, for any $s>0$, with an unbounded control potential, in contrast with the examples in the previous sections.
 
 The quantum harmonic oscillator with angular frequency $\omega$ describes the
 oscilations of a particle of mass $m$ subject to the potential
 $V(x)=\frac{1}{2}m \omega x^2$.
 The corresponding uncontrolled Schr\"odinger equation is
 $$
 \rmi \frac{\partial \psi}{\partial t}=-\frac{\hbar^2}{2m} \Delta \psi(x,t)
 + \frac{1}{2}m \omega x^2 \psi(x,t).
 $$
 With a suitable choice of units, it reads
 $$
 \rmi \frac{\partial \psi}{\partial t}=-\frac{1}{2} \Delta \psi(x,t) +
 \frac{1}{2}
 x^2 \psi(x,t).
 $$
 The operator $A=\displaystyle{\frac{\rmi}{2}\Delta-\frac{x^2}{2}}$ is
 self-adjoint
 on $L^2(\mathbf{R},\mathbf{C})$ and it has
 a pure discrete spectrum.
 The $k^{th}$ eigenvalue (corresponding to the $k^{th}$ energy level)  is equal
 to $\displaystyle{\frac{2k+1}{2}\mathrm{i}}$ and it is associated with the eigenstate
 $$
 \phi_k:x\mapsto \frac{1}{\sqrt{2^k k!\sqrt{\pi}}} \exp \left ( -\frac{x^2}{2} \right
 ) H_k(x),
 $$
 where $H_k$ is the $k^{th}$ Hermite polynomial, namely
 $\displaystyle{ H_k(x)=(-1)^k e^{x^2} \frac{d^k}{dx^k} \left ( e^{-x^2} \right ).}$

 When considering the classical dipolar interaction, the control potential $W$
 takes the form $W(x)=x$ for every $x$ in $\R$. It is well known (see
 \cite{Rouchon} and references therein)  that the resulting control system
 \eqref{eq:blse} is not controllable in any reasonable sense. Indeed the system
 splits in two uncoupled subsystems. The first one is a finite dimensional
 classical harmonic oscillator which is controllable. The second one is a free
 (that is, without control) quantum harmonic oscillator, whose evolution does
not
 depend on the control and  is therefore not controllable.
 
In~\cite[Section IV]{BoussaidCaponigroChambrion}, we show that $(\rmi(-\Delta+V),\rmi W)$ is $s$-{mildly coupled} for every $s>0$. 
 The proof given in~\cite{Rouchon, illner} (and especially the
 decomposition of the system in two decoupled systems) does not require more to
 the control than to be the derivative of a derivable function. Using the
 continuity in Proposition~\ref{prop:extensionBVRadon}, the noncontrollability result  can  be extended
 to Radon measures.
  \begin{proposition}
 The system~\eqref{eq:blse} with $\Omega=\mathbf{R}$, $V:x\mapsto x^2$ and
 $W:x\mapsto x$ is not approximately controllable by means of Radon measures.
 \end{proposition}

Although this example is not approximately controllable, any arbitrarily small perturbation of $W$ by some
smooth localized function $W_2$ restores this feature, see~\cite[Proposition 6.4]{Schrod}. Nonetheless, the approximate controllability in arbitrarily small time is not possible, see~\cite{BeauchardCoronTeismann}, recently extended in~\cite{beauchard:hal-01333537}. This does not affect the mild coupling at any order as 
$(A,\rmi W_2)$ is also {mildly coupled} at any order and $W_2$ commutes with $W$ which ensures that $(A,\rmi (W+W_2))$ is $s$-{mildly coupled} for every $s>0$.

Note that existence of the dynamics is obtained in~\cite{Fujiwara} for measurable in time and  locally bounded in space-time control potentials. It can be extended to Radon measures controls using Section~\ref{sec:radon}. Note that in the case of Radon measures without atoms, for instance $L^1$-controls, the resulting propagator is a weak solution of \eqref{Eq:Potential}, see~Proposition~\ref{prop:propagator-radon2} and Remark~\ref{Rem:RightContinuity}.

\appendix

\section{Notations and Definitions}\label{sec:notations}

Here $T$ is a positive real and $I$ an interval of $\R$.

\paragraph{Bounded operators space.} Let $\X$ and $\Y$ be two Banach spaces,
$L(\X,\Y)$ is the space of linear bounded operator acting on $\X$ with values in
$\Y$. If $\X=\Y$ we write $L(\X) := L(\X,\Y)$.

\paragraph{Weak and strong topology.} Let $(A_n)_{n\in\N}$ a sequence in
$L(\X,\Y)$, let $A$ in $L(\X,\Y)$. We say that $A_n$ converges to $A$ in the strong sense, or strongly,
if for any $\psi$ in $\X$, $(A_n\psi)_{n\in\N}$ converges to $A\psi$ in $\Y$. We say that $A_n$
converges to $A$ in the weak sense, or weakly, if for any $\psi$ in $\X$ and $\phi$ in ${\Y}^{\ast}$,
 the topological dual of $\Y$, $(\phi(A_n\psi))_{n\in\N}$ converges to $\phi(A\psi)$ 
 in $\C$.

\paragraph{Maximal dissipative operators on Hilbert spaces.} An operator $A$ on a Hilbert space $\Hc$ is dissipative
if for any $\phi\in D(A)$, $\Re\langle \phi, A\phi\rangle \leq 0$. It is maximal dissipative if it has no proper 
dissipative extension.

\paragraph{Graph topology.} Consider an operator $A$  on a
Hilbert space $\Hc$ with domain $D(A)$, the graph topology on $D(A)$ is the
topology associated with the norm $\psi\in D(A)\mapsto \|\psi\|_\Hc+\|A\psi\|_\Hc \in
[0,\infty)$. 

\paragraph{Bounded variation functions.}
Let $E \subset \X$ for $\X$ Banach space. A family $t\in I \mapsto u(t) \in E $ is in $BV(I,E)$, i.e. is a bounded variation function from the interval $I$ to $E$,
if there exists $N\geq 0$ such that
$$
 \sum_{j=1}^n \|u(t_j)-u(t_{j-1})\|_\X \leq N,
$$
for any partition $(t_{i})_{i=0}^{n}$ of $I$.
The mapping 
$$
u\in BV(I,E) \mapsto \sup_{(t_{i})_{i}}\sum_{j=1}^n
\|u(t_j)-u(t_{j-1})\|_{\X}
$$
is a semi-norm on $BV(I,E)$
denoted by $\TV{\cdot}{I}{E}$ and it is  called \emph{total variation}.

The space $BV(I,E)$ endowed with the norm $\|\cdot \|_{BV(I)} :=\|\cdot\|_{L^1} + \TV{\cdot}{I}{E}$ is a Banach space.

On $BV(I,E)$, we consider  the convergence of sequences given by:
$(u_{n})_{n\in \N} \in BV(I,E)$ \emph{converges} to $u \in BV(I,E)$ if
$(u_{n})_{n\in N} $ is a bounded sequence in $BV(I,E)$ pointwise convergent to $u \in BV(I,E)$.

Notice that the convergence in the norm $\|\cdot \|_{BV(I)}$ implies pointwise convergence.

\noindent The Jordan Decomposition Theorem provides that any bounded variation function is the difference of two nondecreasing bounded functions. This fact, together with Helly's Theorem provides the well-known Helly's Selection Theorem (see for example~\cite{Helly, Natanson}).
\begin{unumberedtheorem}[Helly's Selection Theorem]
Let $I$ be a  compact interval and $(f_n)_{n\in\N}$ be a sequence in $BV(I,\R)$. If
\begin{itemize}
 \item[$(i)$] there exists $M>0$ such that for all $n\in\N$, $\TV{f_n}{I}{\R}<M$,
 \item[$(ii)$] there exists $x_0\in I$ such that $(f_n(x_0))_{n\in\N}$ is bounded.
\end{itemize}
Then $(f_n)_{n\in\N}$ has a pointwise convergent subsequence.
\end{unumberedtheorem}

\paragraph{Radon measures.}
We consider the space $\radon{I}$ of (signed) Radon measures on $I$.  Recall that a positive Radon measure is a Borel measure which is locally finite and inner regular. Using Hahn decomposition~\cite{Doss1980}  
any  signed Radon measure $\mu$ is
 the difference $\mu=\mu^+-\mu^-$
of two positive Radon measures $\mu^+$ and $\mu^-$ (at least one being finite) with disjoint 
support. We denote the total variation of $\mu$  by $|\mu|(I)$, where  $|\mu|=\mu^++\mu^-$. In this work we consider  Radon measures with bounded total variation. In particular both $\mu^+$  and $\mu^-$ are finite.

Here we only consider finite measures on $I$ so the inner regularity requirement in the definition can be dropped. In the more general $\sigma$-finite case, this requirement can be dropped as well. 
In the first case, a positive Radon measures is a finite Borel measures, while in the second case a positive Radon measure is a locally finite Borel measure. 
Note that, sometimes Borel measures are by definition locally finite. Sometimes the outer regularity is added to the definition of Radon measures, which again is redundant for finite measures.

We say that
$(\mu_{n})_{n\in \N} \in \radon{[0,T]}$ \emph{converges} to $\mu \in \radon{[0,T]}$ if
$\sup_{n}|\mu_{n}|([0,T]) < +\infty$ (i.e. $(\mu_{n})_{n\in \N}$ has uniformly bounded total variations) and $\mu_{n}((0,t]) \to \mu((0,t])$ for  every
$t \in (0,T]$ as $n$ tends to $\infty$. 
Note that this convergence is \emph{not} the one associated with the norm of total variation, see also Remark \ref{Rem:NonUniformDensity}.
 Notice, moreover, that this notion of convergence  is stronger than the weak {convergence of} measures, see~\cite[Section 1.9]{EvansGariepy} and weaker than the strong or total variation {convergence}. It is also stronger than the narrow convergence (also called weak  convergence in \cite{Billingsley, Klenke, Mattila}). For instance, the sequence $\left(\delta_{\frac1n}\right)_{n\in\N}$ converges narrowly to $\delta_0$ but is not convergent according to our definition.

The cumulative function $u(t) = \mu((0,t])$ of a Radon measure $\mu$ is locally of bounded variation
and the associated total variation (which does not depend on the choice of the
cumulative function) coincides with the total variation of the Radon measure.

Every function $u\in L^{1}_{\mathrm{loc}}(I,\R)$ can be seen as the density of an absolutely continuous Radon measure $\mu$, namely $\mu(J) = \int_J u \mathrm{d}\lambda$, (where  $\lambda$ denotes the Lebesgue measure) for every $J\subset I$ borelian. When it does not create ambiguity we identify the function $u$ and the
associated Radon measure $\mu$. Moreover we have the following convergence.
\begin{UnnumberedLemma} 
 Let $(u_n)_{n\in\N} \subset L^{1}(I,\R)$ and $u\in L^{1}(I,\R)$ such that $u_{n} \to u$ in $L^{1}(I,\R)$ as $n$ tends to $\infty$. 
 Let $(\mu_n)_{n\in\N} \subset \radon{I}$ and $\mu \in \radon{I}$ be the associated Radon measures.
 Then  $(\mu_n)_{n\in\N}$ converges to $\mu$ in $\radon{I}$.
\end{UnnumberedLemma}

Note that for $u$ in $L^{1}(I,\R)$ the total variation of the associated Radon measure is the $L^1$-norm of $u$ and hence $L^{1}(I,\R)$ is closed for the total variation topology.

\paragraph{Other notations.} For any interval $I \subset \R$, we define
$$
\Delta_{I}:=\{(s,t)\in I^2\mid \, s\leq t\,\}.
$$

\noindent In a metric space $E$, the notation $B_E(v_0,r)$ stands for the open ball of radius $r$ and center $v_0$ in $E$.

\noindent For a densely defined operator $B$ on a Hilbert space, $B^\ast$ stands for its adjoint. 
Recall that $B^\ast$ is densely defined if and only if $B$ is closable, in any case $B^\ast$ is closed.

\noindent The set $C_0^{1}(I,\X)$ is the set of of functions from an interval $I$ to a Banach $\X$ of class $C^1$ with compact support in the interior of $I$.

\section{Interpolation}\label{sec:interpolation}

\subsection{Convergence of sequences}
Through the present analysis, the following simple interpolation lemma is useful.
\begin{lemma}\label{LEM_interpolation}
Let $A$ be a skew-adjoint operator, let $S$ be a set and $(u_n)_{n\in
\N}$ take value in the set of functions from $S$ to $D(|A|^k)$, such that
$(u_n)_{n \in N}$ is uniformly bounded in $S$ for the norm of $D(|A|^k)$, $k >
0$. If $(u_n)_{n\in \N}$ tends to zero in $\Hc$ uniformly in $S$, then
$(u_n)_{n\in \N}$ tends to zero in $D(|A|^l)$, uniformly in $S$ for every $l<k$.
\end{lemma}
\begin{proof}
The proof follows from the logarithmic convexity of $l\in[0,k]\mapsto
\||A|^lu\|$. Indeed 
\[
 \||A|^{\frac{l+j}{2}}u\|=\sqrt{\langle
|A|^lu,|A|^ju\rangle} \leq \|
|A|^lu\|^{1/2}\||A|^ju\|^{1/2}.
\]
If $l<k$ then 
$$
\||A|^l u_n\|\leq
\|u_n\|^{\frac{k-l}{k}}\||A|^{k} u_n \|^\frac{l}{k}.
$$
Let $C=\sup_{n\in\N}\||A|^{k} u_n \|^2$ and $N>0$ such that for any
$n>N$, $\|u_n\|^2\leq \varepsilon$ we obtain 
\[
  n>N\Longrightarrow \||A|^l u_n\|^2 \leq
\varepsilon^{\frac{k-l}{k}}C^\frac{l}{k} , \]
which provides the lemma.
\end{proof}

\subsection{Interpolation of fractional powers of operators}
Let us now state a more sophisticated result. The following result can also be deduced from the content of \cite[Section 2.8]{amrein} and its proof is an extension to the unbounded case of the result by \cite{pedersen}. 
\begin{proposition}\label{Prop:Interpolation}
Let $A$ and $B$ be two self-adjoint positive operators in $\Hc$ such that there
exists $c>0$ with 
\[
 c\leq B \leq A
\]
in the form sense. Then, for any $\alpha\in (0,1)$, 
\[
 c^\alpha\leq B^\alpha \leq A^\alpha.
\]
\end{proposition}

The proof of Proposition~\ref{Prop:Interpolation} follows from the following series of lemmas.

For a selfadjoint operator $A$ and $z\in \C\setminus\R$,
the functional calculus is the extension of the mapping
\[
 \left\{x\in \R \mapsto (x-z)^{-1}\right\}\in B(\R) \to (A-z)^{-1} \in B(\Hc)
\]
as a strong continuous 
$\ast$-algebra homomorphism on the space $B(\R)$ of bounded borelian functions on the real line to $B(\Hc)$.

Let us recall the following functional calculus identity based on the Poisson
formula, see~\cite[Lemma 6.1.1]{amrein}.
\begin{lemma}\label{Lem:PoissonFormula}
Let $A$ be a selfadjoint operator in $\Hc$.
 Let $f$ be a bounded borelian function.
 Then $f(A)$ is the weak-limit as $\varepsilon \to 0^+$ of
  \[
  \frac{1}{2 \rmi \pi}\int_{\R} f(\lambda)
\Im(A-\lambda-i \varepsilon)^{-1}\,d\lambda.
 \]
\end{lemma}
We also recall, for $\alpha\in (0,1)$ and $x>0$, the formula
\[
 x^{-\alpha} = \frac{\pi}{\sin(\pi\alpha)}\int_{0}^{\infty}
\frac{w^{-\alpha}}{x+w}\,dw.
\]
Then from the Fubini theorem and Lemma~\ref{Lem:PoissonFormula} we obtain the following.
\begin{lemma}
Let $A$ be a positive selfadjoint operator in $\Hc$.
Then for $\alpha\in (0,1)$
\[
 A^\alpha = \frac{\pi}{\sin(\pi\alpha)}\int_{0}^{\infty} \frac{w^{-1+\alpha}
A}{A+w}\,dw
\]
on $D(A)$.
\end{lemma}
The domain of validity of the above identity can be extended to any core of $A^\alpha$ that makes the integral strongly convergent.

\begin{lemma}
Let $A$ and $B$ be two self-adjoint positive operators in $\Hc$ such that there
exists $c>0$ with 
\[
 c\leq B \leq A.
\]
Then
\[
 A^{-1} \leq B^{-1}.
\]
\end{lemma}
\begin{proof}
First notice that both $A$ and $B$ are invertible from their domains to $\Hc$
as well as their square roots. Then from
\[
 \sqrt{c}\|u\|\leq \|\sqrt{B}u \| \leq \|\sqrt{A}u\|,
\]
we deduce that $\sqrt{B}\sqrt{A}^{-1}$ is a bounded operator with norm
at most $1$.

In the other hand the operator $\sqrt{A}^{-1}\sqrt{B}$ defined on $D(\sqrt{B})$
extends as the adjoint of $\sqrt{B}\sqrt{A}^{-1}$ to a closed operator on $\Hc$
and hence is bounded with norm at most $1$ and
\[
\|\sqrt{A}^{-1}\sqrt{B}u \| \leq \|u\|, \forall u \in D(\sqrt{B})
\]
and thus
\[
\|\sqrt{A}^{-1}u \| \leq \|\sqrt{B}^{-1}u\|.
\]
and the result follows.
\end{proof}

\begin{proof}[Proof of Proposition~\ref{Prop:Interpolation}]
We have that
\[
 c\leq B \leq A,
\]
which implies, for any $w>0$, that
\[
 1-w(B+w)^{-1} \leq 1-w(A+w)^{-1}.
\]
and thus
\[
\frac{w^{-1+\alpha}
B}{B+w}\leq
 \frac{w^{-1+\alpha}
A}{A+w}.
\]
Integrating on $w>0$ (first restricted to $D(A)\times D(A)$) gives the desired inequality by density.
\end{proof}

Proposition~\ref{Prop:Interpolation} can be extend to the case $c=0$ by replacing $A$ and $B$ by $A+\epsilon$ and $B+\epsilon$ as in~\cite{pedersen},
we then obtain
\[
0\leq B^\alpha\leq  (B+\epsilon)^\alpha\leq  (A+\epsilon)^\alpha.
\]
The second inequality is immediate. We have that
\[
0\leq (A+\epsilon)^{-\alpha/2}B^\alpha(A+\epsilon)^{-\alpha/2}\leq  1,
\]
so that, taking $\epsilon$ to $0$, gives
\begin{equation}\label{eq:Balfa}
 0\leq B^\alpha\leq A^\alpha.
\end{equation}
We hence deduce the following corollary.
\begin{corollary}\label{Lem:DomainInterpolation}
 Let $A$ and $B$ be two positive self-adjoint operators sharing the 
 same domains. For any $\alpha\in (0,1)$, we have :
 \[
  D(A^\alpha)=D(B^\alpha)
 \]
\end{corollary}
\begin{proof}
As $B$ is closed it is a bounded operator from $D(A)$ to $\Hc$. Thus
 \[
\exists c>0,\forall \phi\in D(A),  \|B\phi\|\leq c\|A\phi\|.
 \]
 Hence 
 \[
  B^2\leq c^2 A^2,
 \]
 and, from~\eqref{eq:Balfa},  we have that $B^{2\alpha/2}$ is bounded from $D(A^{2\alpha/2})$ to $\Hc$.
We conclude by noticing that the proof is  symmetric in $A$ and $B$.
\end{proof}

\section{Sufficient conditions for approximate controllability with bounded variation controls}
The aim of this Section is to recall approximate controllability results obtained in other 
contexts and how this results may be adapted to the framework of the present analysis. 

We first recall the following definitions from~\cite{Schrod}.
\begin{definition}
 Let $(A,B,\R)$ satisfy Assumptions~\ref{ass:ass}
 such that $A$ and $B$ are skew-symmetric. Let $\Phi=(\phi_k)_k$ be a
Hilbert basis of $\Hc$ made of eigenvectors of $A$, $A\phi_k=\mathrm{i}\lambda_k
\phi_k$ for every $k$ in $\mathbf{N}$. A pair $(j,k)$ of integers is a  
\emph{nondegenerate transition} of $(A,B,\Phi)$ if $(i)$  $\langle \phi_j,B\phi_k \rangle \neq 
 0$ and $(ii)$ for every $(l,m)$ in $\mathbf{N}^2$, $|\lambda_j-\lambda_k|=|
\lambda_l-\lambda_m|$ implies $(j,k)=(l,m)$ or $\langle \phi_l, B \phi_m\rangle =0$ or 
$\{j,k\}\cap \{l,m\}=\emptyset$.
\end{definition}
 \begin{definition}
 Let $(A,B,\R)$ satisfy Assumptions~\ref{ass:ass}
 such that $A$ and $B$ are skew-symmetric. Let $\Phi=(\phi_k)_k$ be a
Hilbert basis of $\Hc$ made of eigenvectors of $A$, $A\phi_k=\mathrm{i}\lambda_k
\phi_k$ for every $k$ in $\mathbf{N}$. A subset $S$ of $\mathbf{N}^2$ 
is a \emph{nondegenerate chain of connectedness} of $(A,B,\Phi)$ if
$(i)$ for every $(j,k)$ in $S$, $(j,k)$ is a nondegenerate transition of $(A,B)$ and $(ii)$ 
for every $r_a,r_b$ in $\mathbf{N}$, there exists a finite sequence $r_a=r_0,r_1,
\ldots,r_p=r_b$ in $\mathbf{N}$ such that, for every $j\leq p-1$, $(r_j,r_{j+1})$ belongs 
to $S$.   
\end{definition}
\begin{proposition}\label{prop:controlHs}
 Let $(A,B,\R)$ satisfy Assumptions \ref{ass:ass}
 such that $A$ and $B$ are skew-symmetric. Let $\Phi=(\phi_k)_k$ be a
Hilbert basis of $\Hc$ made of eigenvectors of $A$, $A\phi_k=\mathrm{i}\lambda_k
\phi_k$ for every $k$ in $\mathbf{N}$. Let $S$ be a nondegenerate chain of
connectedness of $(A,B)$. 
 Then, for every
$\eta>0$, $(A,B)$ is simultaneously approximately controllable in
$D(|A|^{1-\eta})$. 
\end{proposition}
\begin{proof}
 First of all, it is enough to prove the result for target propagators 
$\hat{\Upsilon}$ leaving invariant the space of co-dimension 2 spanned by
$(\phi_j,\phi_k)$ for $(j,k)$ in $S$
$$
\hat{\Upsilon}=e^{\mathrm{i}\nu_l}  (\cos(\theta) \phi_l^\ast \phi_l +
\sin(\theta) \phi_l^\ast \phi_k) +
e^{\mathrm{i}\nu_k}  (-\sin(\theta) \phi_k^\ast \phi_l + \cos(\theta)
\phi_l^\ast \phi_k).
$$
The result in $\Hc$-norm is a consequence of \cite[Theorem 1]{periodic}: for every piecewise
constant $u^{\ast}:\R\rightarrow \R$,
$2\pi/|\lambda_j-\lambda_k|$-periodic such that 
$$
\int_0^{\frac{2\pi}{|\lambda_j-\lambda_k|}}
u^{\ast}(\tau)e^{\mathrm{i}(\lambda_j-\lambda_k)\tau}\mathrm{d}\tau \neq 0,
$$
and 
$$
\int_0^{\frac{2\pi}{|\lambda_j-\lambda_k|}}
u^{\ast}(\tau)e^{\mathrm{i}(\lambda_l-\lambda_m)\tau}\mathrm{d}\tau =0,
$$
for every $l,m$ such that $(\lambda_l-\lambda_m)\in \mathbf{Z}
(\lambda_j-\lambda_k)$ and $b_{l,m}\neq 0$, there exists $T^\ast$ such that 
$\Upsilon^{u^\ast/n}(nT^\ast,0)$ tends to $\hat{\Upsilon}$ as $n$ tends to
infinity.

The conclusion follows using Lemma~\ref{LEM_interpolation} and the estimate in 
$A$-norm of Theorem~\ref{thm:kato}.
\end{proof}

Let us just mention the following result in the case of higher regularity.
\begin{proposition}\label{prop:controlHsFEPS}
Let $k$ be a positive real.  Let $(A,B,\R)$ satisfy Assumptions \ref{ass:ass}
 such that $(A,B)$ is $k$-mildly coupled. Let $\Phi=(\phi_k)_k$ be a
Hilbert basis of $\Hc$ made of eigenvectors of $A$, $A\phi_k=\mathrm{i}\lambda_k
\phi_k$ for every $k$ in $\mathbf{N}$. Let $S$ be a nondegenerate chain of
connectedness of $(A,B)$ such that, for every $(j,k)$ in $S$, the set
$\{(l,m)\in \mathbf{N}^2| (\lambda_l-\lambda_m)\in \mathbf{Z}
(\lambda_j-\lambda_k) \mbox{ and } \langle \phi_l,B\phi_m\rangle\neq 0\}$ is finite. Then, for every
$\eta>0$, $(A,B)$ is simultaneously approximately controllable in
$D(|A|^{k/2+1-\eta})$. 
\end{proposition}
\begin{proof}
The proof differs from the previous one for the interpolation step and for the use of 
Proposition~\ref{PRO_cont_entree_sortie_Hk-BV}.
\end{proof}

\section{Analytical perturbations}\label{SEC_analytic_perturbations}
To apply sufficient condition for approximate controllability (Proposition
\ref{prop:controlHsFEPS}), we need to find a nonresonant chain of connectedness,
which may require some work on practical examples. A classical idea already
used in this study is to introduce a new control $\tilde{u}=u-\bar{u}$ and to
consider the system $x'=(A+\bar{u} B)+ (u-\bar{u})B$ for a suitably chosen
constant $\bar{u}$.

We have the following results by Kato~\cite[Section VII.2]{Kato}.
\begin{definition}
Let $D_0$ be a domain of the complex plane, a family $(T(z))_{z\in D_0}$ of
closed operators from a Banach space $X$ to a Banach space $Y$ is said to
be \emph{a holomorphic family of type (A)} if
\begin{enumerate}
 \item $D(T(z))=D$ is independent of $z$,
 \item $T(z)u$ is holomorphic for $z$ in $D_0$ for every $u$ in $D$.
\end{enumerate}
\end{definition}

\begin{theorem}[\protect{\cite[Theorem VII.3.9]{Kato}}]
 Let $T(z)$ be a \emph{selfadjoint} holomorphic family of type (A) defined for
$z$ in a neighborhood of an interval $I_0$ of the real axis such that
$T(z)^\ast=T(\bar z)$. Furthermore,
let $T(z)$ have a compact resolvent. Then all eigenvalues of $T(z)$ can be
represented by functions which are holomorphic in $I_0$\footnote{Each of them
is
holomophic in some neighborhood of $I_0$ but possibly different for each in
such a way that their intersection is just $I_0$.}.

More precisely, there is a sequence of scalar-valued functions $(z\mapsto
\lambda_n(z))_{n \in \N}$ and operator-valued functions $(z \mapsto \phi_n(z))_{n
\in \N}$, all holomorphic on $I_0$, such that for
$z$ in $I_0$, the sequence $(\lambda_n(z))_{n \in \N}$ represents all the repeated
eigenvalues of $T(z)$
and $(\phi_n(z))_{n \in \N}$ forms a complete orthonormal family of the
associated
eigenvectors of $T(z)$.
\end{theorem}

\begin{proposition}
If $(A,B,K)$ satisfies Assumptions~\ref{ass:ass} then the family $\rmi(A+ z B)_{z \in \C,|z|<1/\|B\|_A}$ is
holomorphic of type (A).
\end{proposition}

\begin{proof}
The question of domain is solved by the Kato--Rellich Theorem. The holomorphy is 
immediate as the family $\rmi(A+ z B)$ is affine in $z$.
\end{proof}

\bibliographystyle{alpha}
\bibliography{biblioteca}

\end{document}